%% file: virtual_domination_III_7.tex
\documentclass{amsart}
\usepackage{amssymb,amsmath,mathrsfs,amsfonts}
\usepackage{amscd, amssymb, amsmath, amsthm, graphics,graphicx}
\usepackage{amsmath,amsfonts,amsthm,amssymb}
\usepackage{slashbox}
\usepackage{fancyhdr}
\usepackage[all]{xy}
\usepackage[colorlinks=true]{hyperref}
\usepackage[small,nohug,heads=vee]{diagrams}
\usepackage{multirow}
\usepackage{enumitem}
\usepackage{psfrag}
\diagramstyle[labelstyle=\scriptstyle]
\numberwithin{equation}{section}

\newtheorem{theorem}{Theorem}[section]
\newtheorem{lemma}[theorem]{Lemma}
\newtheorem{proposition}[theorem]{Proposition}

\newtheorem{construction}[theorem]{Construction}

\newtheorem{notation}[theorem]{Notation}
\newtheorem{parameter}[theorem]{Parameter}
\newtheorem{assumption}[theorem]{Assumption}

\theoremstyle{definition}
\newtheorem{definition}[theorem]{Definition}

\newtheorem{question}[theorem]{Question}
\newtheorem{remark}[theorem]{Remark}

\begin{document}
\sloppy

\title[Virtual domination of $3$-manifolds III]{Virtual domination of $3$-manifolds III}

    %Information for first author
\author{Hongbin Sun}
\address{Department of Mathematics, Rutgers University - New Brunswick, Hill Center, Busch Campus, Piscataway, NJ 08854, USA}
\email{hongbin.sun@rutgers.edu}

%    General info

\thanks{The author is partially supported by Simons Collaboration Grants 615229.}

\date{\today}
\begin{abstract}
We prove that for any oriented cusped hyperbolic $3$-manifold $M$ and any compact oriented $3$-manifold $N$ with tori boundary, there exists a finite cover $M'$ of $M$ that admits a degree-$8$ map $f:M'\to N$, i.e. $M$ virtually $8$-dominates $N$. 
\end{abstract}

\maketitle
\vspace{-.5cm}
%\newpage
\section{Introduction}

In this paper, we assume all manifolds are compact, connected and oriented, unless otherwise indicated. By a cusped hyperbolic $3$-manifold, we mean a compact $3$-manifold with nonempty tori boundary, such that its interior admits a complete hyperbolic structure with finite volume, unless otherwise indicated.

For two closed oriented $n$-manifolds $M,N$ and a map $f:M\to N$, a natural quantity associated to $f$ is its mapping degree. The mapping degree of $f$ is $d\in \mathbb{Z}$ if $f_*([M])=d[N]$ for oriented fundamental classes $[M]\in H_n(M;\mathbb{Z})$ and $[N]\in H_n(N;\mathbb{Z})$.  The notion of mapping degree can be generalized to proper maps between manifolds with boundary. For two compact oriented $n$-manifolds $M$ and $N$ with boundary, a map $f:M\to N$ is {\it proper} if $f^{-1}(\partial N)=\partial M$ holds. The mapping degree of a proper map $f:M\to N$ is $d\in \mathbb{Z}$ if $f_*([M,\partial M])=d[N,\partial N]$ for oriented relative fundamental classes  $[M,\partial M]\in H_n(M,\partial M;\mathbb{Z})$ and $[N,\partial N]\in H_n(N,\partial N;\mathbb{Z})$. In either of the above cases, $f$ is a {\it non-zero degree map} if the degree of $f$ is not zero. If the mapping degree $d\ne 0$, we say that $M$ {\it $d$-dominates} $N$, and we say that $M$ {\it dominates} $N$ if $M$ $d$-dominates $N$ for some non-zero integer $d$. In this paper, we will work on $3$-manifolds with nonempty boundary, and all maps $f:M\to N$ between such $3$-manifolds are proper, unless otherwise indicated.

Roughly speaking, if $M$ dominates (or $1$-dominates) $N$, then $M$ is topologically more complicated than $N$. For certain invariants of manifolds, e.g. ranks of fundamental groups, betti numbers, simplicial volumes, representation volumes, etc, this impression on behavior of topological invariants under non-zero degree maps (or degree-$1$ maps) are classical results. However, for some other invariants, e.g. Heegaard genera and Heegaard Floer homology of $3$-manifolds, it is unknown whether the above impression is correct.

In \cite{CT}, the authors asked the following question: Whether there is an easily described class $\mathscr{C}$ of closed oriented $n$-manifolds, such that any closed oriented $n$-manifold is dominated by some $M\in \mathscr{C}$. In \cite{Gai}, Gaifullin proved that, for any positive integer $n$, there exists a closed oriented $n$-manifold $M_0$, such that any closed oriented $n$-manifold is dominated by a finite cover of $M_0$ (virtually dominated by $M_0$), i.e. we can take $\mathscr{C}$ to be the set of all finite covers of $M_0$. In \cite{Sun2,LS,Sun5}, the author and Liu proved the following result. 

\begin{theorem}\label{closedmanifold}\cite{Sun2, LS, Sun5}
For any closed oriented $3$-manifold $M$ with positive simplicial volume and any closed oriented $3$-manifold $N$, there exists a finite cover $M'$ of $M$ that admits a degree-$1$ map $f:M'\to N$. 
\end{theorem}

So for any closed oriented $3$-manifold $M$ with positive simplicial volume, we can take $\mathscr{C}$ to be the set of all finite covers of $M$.

Note the condition that $M$ has positive simplicial volume is necessary for Theorem \ref{closedmanifold}, since a manifold with zero simplicial volume does not dominate any manifold with positive simplicial volume, and the simplicial volume has the covering property.

In this paper, we generalize the above virtual domination result from closed $3$-manifolds to $3$-manifolds with tori boundary. The following theorem is the main result of this paper.

\begin{theorem}\label{main2}
For any oriented cusped hyperbolic $3$-manifold $M$ and any compact oriented $3$-manifold $N$ with tori boundary, there exists a finite cover $M'$ of $M$ that admits a proper map $f:M'\to N$ with $\text{deg}(f)=8$.
\end{theorem}

The proof of Theorem \ref{main2} can also be applied to prove a similar result on certain mixed $3$-manifolds. Here a mixed $3$-manifold is a compact oriented irreducible $3$-manifold with empty or tori boundary, such that it has nontrivial JSJ decomposition and at least one hyperbolic JSJ piece.

\begin{theorem}\label{mixed}
For any compact oriented mixed $3$-manifold $M$ with tori boundary such that a hyperbolic piece of $M$ intersects with $\partial M$, and  any compact oriented $3$-manifold $N$ with tori boundary, there exists a finite cover $M'$ of $M$ that admits a  proper map $f:M'\to N$ with $\text{deg}(f)= 8$.
\end{theorem}

At first, we can not prove virtual $1$-domination for Theorems \ref{main2} and \ref{mixed}. Although we can prove virtual $1$-, $2$, or $4$-domination in certain special cases, we do need to state our result as virtual $8$-domination.

For technical reason, we can not prove Theorem \ref{mixed} for other mixed $3$-manifolds with tori boundary, although we do expect the virtual domination result still holds in that case. To fully resolve this problem, it remains to study mixed $3$-manifolds such that all of their boundary components are contained in Seifert pieces.
	
\begin{question}\label{question1}
Let $M$ be a compact oriented $3$-manifold with nonempty tori boundary and positive simplicial volume. Does $M$ virtually ($1$-)dominate all compact oriented $3$-manifolds with tori boundary?
\end{question}

For a statement as Theorem \ref{main2}, we do not have to restrict to compact oriented $3$-manifolds with tori boundary, and we ask what happens for all compact oriented $3$-manifolds with nonempty (possibly higher genus) boundary. 

\begin{question}\label{question2}
Which compact oriented $3$-manifold $M$ with boundary virtually dominates all compact oriented $3$-manifolds with boundary?
\end{question}

Two necessary conditions for Questions \ref{question2} are: $M$ has a boundary component of genus at least $2$, and the double of $M$ has positive simplicial volume. If the boundary of $M$ only consists of $2$-spheres and tori, so does any finite cover $M'$ of $M$. Then $M'$ does not dominate any $3$-manifold with higher genus boundary, by considering the restriction map on the boundary. Moreover, if $M$ virtually dominates $N$ and both manifolds have boundary, then $D(M)$ virtually dominates $D(N)$. Since we can choose $N$ so that $D(N)$ has positive simplicial volume, then so does $D(M)$.

Before we sketch the proof of Theorem \ref{main2}, let's first recall the proof of virtual domination results (Theorem \ref{closedmanifold}) of closed $3$-manifolds in \cite{Sun2}, \cite{LS} and \cite{Sun5}. All these three proofs roughly follow the same circle of ideas, and we sketch the proof of the most general result in \cite{Sun5} here. At first, by \cite{BW}, we can assume the target manifold $N$ is a closed hyperbolic $3$-manifold, and we take a geometric triangulation of $N$. Since $M$ has positive simplicial volume, let $M_0$ be a hyperbolic JSJ piece of a prime summand of $M$. Then we construct a map $j^{1}:N^{(1)}\to M_0$ from the $1$-skeleton $N^{(1)}$ of $N$ to $M_0$, such that $j^{1}$ maps the boundary of each triangle $\Delta$ in $N$ to a null-homologous closed curve in $M_0$. For each triangle $\Delta$ in $N$, we construct a compact orientable surface $S_{\Delta}$ with connected boundary and a map $S_{\Delta}\looparrowright M_0$ that maps $\partial S_{\Delta}$ to $j^{1}(\partial \Delta)$, so that $S_{\Delta}$ is mapped to a nearly geodesic subsurface in $M_0$. Then the maps $j^{1}:N^{(1)}\to M_0$ and $\{S_{\Delta}\looparrowright M_0\}$ together give a map $j:Z\looparrowright M_0$ from a $2$-complex $Z$ to $M_0$. If we construct the maps $\{S_{\Delta}\looparrowright M_0\}$ carefully enough, $j:Z\looparrowright M_0$ induces an injective homomorphism on $\pi_1$. Since $j_*(\pi_1(Z))<\pi_1(M_0)<\pi_1(M)$ is a separable subgroup in $\pi_1(M)$ (by \cite{Sun3}, which generalizes Agol's celebrated result on LERFness of hyperbolic $3$-manifold groups in \cite{Agol2}), the map $j:Z\looparrowright M$ lifts to an embedding $j':Z\hookrightarrow M'$ into a finite cover $M'$ of $M$. A neighborhood of $Z$ in $M'$ is a compact oriented $3$-manifold $\mathcal{Z}$ with boundary, and is homeomorphic to the manifold obtained from a neighborhood $\mathcal{N}(N^{(2)})$ of $N^{(2)}$ in $N$, by replacing each $\Delta \times I$ by $S_{\Delta}\times I$. Then there is a proper degree-$1$ map $g:\mathcal{Z}\to \mathcal{N}(N^{(2)})$ that maps each $S_{\Delta} \times I\subset \mathcal{Z}$ to $\Delta\times I \subset \mathcal{N}(N^{(2)})$. This proper degree-$1$ map $g:\mathcal{Z}\to \mathcal{N}(N^{(2)})$ extends to a degree-$1$ map $f:M'\to N$, by mapping each component of $M'\setminus \mathcal{Z}$ to the union of some components of $N\setminus \mathcal{N}(N^{(2)})$ (each component is a $3$-ball) and a finite graph in $N$.

In the context of manifolds with boundary, the above proof fails in the last step, but we need to fix it from the very first step. For example, if we apply the above approach to manifolds with boundary, it is possible that some component $C$ of $M'\setminus \mathcal{Z}$ does not intersect with $\partial M'$, but a component of $N\setminus \mathcal{N}(N^{(2)})$ intersecting $g(\partial C)$ may contain some component of $\partial N$. In this case $g:\mathcal{Z}\to \mathcal{N}(N^{(2)})$ does not extend to a proper map $f:M'\to N$. Moreover, even if each component $C$ of $M'\setminus \mathcal{Z}$ intersects with $\partial M'$, it is also difficult to construct the desired extension $f:M'\to N$. So we need to take a more careful construction for proving Theorem \ref{main2}, which is sketched in the following.

In Section \ref{reduction}, we reduce the proof of Theorem \ref{main2} to $3$-manifolds $M$ and $N$ satisfying the following extra assumptions.
\begin{itemize} 
\item $M$ has two components $T_1,T_2$, such that the kernel of $H_1(T_1\cup T_2;\mathbb{Z})\to H_1(M;\mathbb{Z})$ contains an element with nontrivial components in both $H_1(T_1;\mathbb{Z})$ and $H_1(T_2;\mathbb{Z})$.
\item $N$ is a finite volume hyperbolic $3$-manifold with a single cusp.
\end{itemize}
In Section \ref{initialdata}, we take a geometric cellulation of a compact core $N_0$ of $N$ which has extra edges than a geometric triangulation, such that each triangle contained in $\partial N_0$ is almost an equilateral triangle. In Section \ref{constructZ2section}, we construct two maps $j^{(1)}_s:N^{(1)}\to M$ for $s=1,2$ such that the following hold:
\begin{enumerate}
\item For each triangle $\Delta$ of $N_0$ contained in $\partial N_0$, $j^{(1)}_s(\partial \Delta)$ bounds a geodesic triangle in $M$. 
\item For each $s=1,2$, the union of geodesic triangles in $M$ bounded by $j_s^{(1)}(\partial \Delta)$ in item (1) gives a mapped-in torus $T\to M$ homotopic into $T_s$.
\item For each triangle or bigon $\Delta$ of $N_0$ not contained in $\partial N_0$, $j_1^{(1)}(\partial \Delta)\cup j_2^{(1)}(\partial \Delta)$ is null-homologous in $M$.
\end{enumerate}
For each triangle $\Delta$ of $N_0$ as in item (3), we construct a compact orientable surface $S_{\Delta}$ and a nearly geodesic immersion $S_{\Delta}\looparrowright M$ bounded by two copies of $j_1^{(1)}(\partial \Delta)\cup j_2^{(1)}(\partial \Delta)$. Then two copies of  $j_s^{(1)}:N^{(1)}\to M$ with $s=1,2$, two copies of the tori in item (2) and the maps $\{S_{\Delta}\looparrowright M\}$ together give a $2$-complex $Z$ and a map $j:Z\looparrowright M$. In Section \ref{pi1inj2}, we prove that if the construction is done carefully, $j:Z\looparrowright M$ is $\pi_1$-injective. After this step, the construction of the virtual domination (proper) map is similar to the closed manifold case. We first use Agol's result (\cite{Agol2}) that $j_*(\pi_1(Z))<\pi_1(M)$ is a separable subgroup to lift $Z$ to an embedded $2$-complex in a finite cover $M'$ of $M$, and take a neighborhood of $Z$ in $M'$ denoted by $\mathcal{Z}$. Then we have a proper degree-$4$ map $g:\mathcal{Z}\to \mathcal{N}(N^{(2)})$, such that the following hold.
\begin{itemize}
\item For the component $T'=\partial N_0$ of $\partial \mathcal{N}(N^{(2)})$, each component of $g^{-1}(T')$ is a torus in $M'$ parallel to a component of $\partial M'$.
\end{itemize}
This key property implies that $g$ can be extended to a proper degree-$4$ map $f:M'\to N$, as desired (see Section \ref{constructdomination2}).

Note that the $\pi_1$-injectivity of $j:Z\looparrowright M$ can not be proved by exactly the same way as in \cite{Sun2,LS,Sun5}. In \cite{Sun2,LS,Sun5}, we equipped $Z$ with a natural metric and proved that the map $\tilde{j}:\tilde{Z}\to \tilde{M}=\mathbb{H}^3$ on universal covers is a quasi-isometric embedding. However, in the current case, $\tilde{j}:\tilde{Z}\to \tilde{M}$ is not a quasi-isometric embedding anymore, since $j(Z)$ contains some tori homotopic into $\partial M$. To prove the $\pi_1$-injectivity of $j$, we modify $Z$ as following. For each torus $T$ in $Z$ as in item (2) above (that is homotopic into a horotorus in $M$), we add the cone of $T$ to $Z$ with the cone point deleted, and get an ideal $3$-complex $Z^{3}$ (a $3$-complex with certain vertices deleted). The map $j:Z\looparrowright M$ extends to a map $j_1:Z^{3}\looparrowright M$ that maps ideal vertices of $Z^{3}$ to corresponding ends of $M$. In Section \ref{pi1inj2}, we prove the $\pi_1$-injectivity of $j:Z\looparrowright M$ by proving that $\tilde{j_1}:\tilde{Z}^{3}\to \tilde{M}=\mathbb{H}^3$ is a quasi-isometric embedding.

\bigskip

Although the above description of $j:Z\looparrowright M$ is mostly topological, we actually need geometric methods to construct it. Our main geometric tool for constructing various geometric objects is the good pants construction. Roughly speaking, the good pants construction is a tool box that uses so called good curves, good pants and other good objects to construct geometrically nice objects in hyperbolic $3$-manifolds. The good pants construction was initiated by Kahn and Markovic in \cite{KM}, for constructing nearly geodesic $\pi_1$-injective immersed closed subsurfaces in closed hyperbolic $3$-manifolds, with good pants as building blocks. Then in \cite{KW}, Kahn and Wright generalized Kahn-Markovic's work to construct nearly geodesic $\pi_1$-injective immersed closed subsurfaces in cusped hyperbolic $3$-manifolds. These geometrically nice subsurfaces of Kahn-Wright are basic pieces for constructing our $2$-complex $j:Z\looparrowright M$ in cusped hyperbolic $3$-manifolds. More details on the good pants construction can be found in Section \ref{pregoodpants}.

Now we summarize the organization of this paper. In Section \ref{pregoodpants}, we review the good pants construction in closed and cusped hyperbolic $3$-manifolds, including works in \cite{KM, LM, KW, Sun4}. In Section \ref{prehyperbolic}, we review and prove some elementary geometric estimates in hyperbolic geometry. In Section \ref{reduction}, we prove preparational results that reduce the domain and target manifolds in Theorems \ref{main2} and \ref{mixed}. The technical heart of this paper is in Sections \ref{topo2} and \ref{pi1inj2}. In Section \ref{topo2}, we construct the mapped-in $2$-complex $j:Z\looparrowright M$ and the virtual domination map from $M$ to $N$, modulo the $\pi_1$-injectivity of $j:Z\looparrowright M$ (Theorem \ref{pi1injectivity}). The $\pi_1$-injectivity of $j$ will be proved in Section \ref{pi1inj2}.

\bigskip
\bigskip

\section{Preliminary on the good pants construction}\label{pregoodpants}

In this section, we review the good pants construction on finite volume hyperbolic $3$-manifolds, including constructions of nearly geodesic subsurfaces (\cite{KM} and \cite{KW}), works on panted cobordism groups (\cite{LM} and \cite{Sun4}) and the connection principle of cusped hyperbolic $3$-manifolds (\cite{Sun4}). 

\subsection{Constructing nearly geodesic subsurfaces in finite volume hyperbolic $3$-manifolds}\label{constructsurface}
In \cite{KM}, Kahn and Markovic proved the following surface subgroup theorem. This work initiates the development of the good pants construction, and it was the first step of Agol's proof of Thurston's virtual Haken and virtual fibering conjectures (\cite{Agol2}).
	\begin{theorem}[Surface subgroup theorem \cite{KM}]\label{surface}
		For any closed hyperbolic $3$-manifold $M$,
		there exists an immersed closed hyperbolic subsurface $f\colon S\looparrowright M$,
		such that $f_*\colon \pi_1(S)\rightarrow 	\pi_1(M)$ is injective.
	\end{theorem}
The immersed subsurface of Kahn and Markovic is geometrically nice, and it is built by pasting a large collection of {\it $(R,\epsilon)$-good pants} along {\it $(R,\epsilon)$-good curves} in a nearly geodesic way. These terminologies are summarized in the following. 

We fix a closed oriented hyperbolic $3$-manifold $M$, a small number $\epsilon>0$ and a large number $R>0$. 
\begin{definition}\label{goodcurves}
An {\it $(R,\epsilon)$-good curve} is an oriented closed geodesic in $M$ with complex length satisfying $|{\bf l}(\gamma)-2R|<2\epsilon$. The (finite) set consisting of all such $(R,\epsilon)$-good curves is denoted by ${\bold \Gamma}_{R,\epsilon}$.
\end{definition}

Here the complex length of $\gamma$ is defined by ${\bf l}(\gamma)=l+i\theta \in \mathbb{C}/2\pi i\mathbb{Z}$, where $l\in \mathbb{R}_{>0}$ is the length of $\gamma$, and $\theta\in \mathbb{R}/2\pi \mathbb{Z}$ is the rotation angle of the loxodromic isometry of $\mathbb{H}^3$ corresponding to $\gamma$. In this paper, we adopt the convention in \cite{KW} that good curves have length close to $2R$, instead of the convention in \cite{KM} that good curves have length close to $R$.

\begin{definition}\label{goodpants}
We use $\Sigma_{0,3}$ to denote the oriented topological pair of pants.
A pair of {\it $(R,\epsilon)$-good pants} is a homotopy class of immersion $\Sigma_{0,3} \looparrowright M$, denoted by $\Pi$, such that all three cuffs of $\Sigma_{0,3}$ are mapped to $(R,\epsilon)$-good curves $\gamma_1,\gamma_2,\gamma_3\in {\bold \Gamma}_{R,\epsilon}$, and the complex half length $\bold{hl}_{\Pi}(\gamma_i)$ of each $\gamma_i$ with respect to $\Pi$ satisfies
		$$\left|\bold{hl}_{\Pi}(\gamma_i)-R\right|<\epsilon.$$
We use ${\bold \Pi}_{R,\epsilon}$ to denote the finite set of all $(R,\epsilon)$-good pants. 
\end{definition}

Here the complex half length ${\bf hl}_{\Pi}(\gamma_i)$ measures the complex distance between two vectors $\vec{v}_{i-1},\vec{v}_{i+1}$ along $\gamma_i$, where $\vec{v}_{i-1}$ and $\vec{v}_{i+1}$ are tangent vectors of oriented common perpendicular segments (seams) from $\gamma_i$ to $\gamma_{i-1}$ and $\gamma_{i+1}$ respectively. See Section 2.1 of \cite{KM} for the precise definition of complex half length.
% is defined as the following. Let $s_{i-1}$ and $s_{i+1}$ be the common perpendicular segments (seams) from $\gamma_i$ to $\gamma_{i-1}$ and $\gamma_{i+1}$ respectively, let $\vec{v}_{i-1}$ and $\vec{v}_{i+1}$ be the tangent vectors of $s_{i-1}$ and $s_{i+1}$ at their intersections with $\gamma_i$ respectively. Then the complex half length $\bold{hl}_{\Pi}(\gamma_i)$ is defined to be the complex distance between $\vec{v}_{i-1}$ and $\vec{v}_{i+1}$ along $\gamma_i$, i.e. $\bold{hl}_{\Pi}(\gamma_i)=l+i\theta \in \mathbb{C}/2\pi i\mathbb{Z}$ where $l$ is the length of the oriented geodesic subsegment of $\gamma_i$ from $s_{i-1}$ to $s_{i+1}$, and $\theta$ is the angle from the parallel transport of $\vec{v}_{i-1}$ along this geodesic segment to $\vec{v}_{i+1}$. 
If $\gamma\in {\bf \Gamma}_{R,\epsilon}$ is a cuff of $\Pi\in {\bf \Pi}_{R,\epsilon}$, then $\bold{hl}_{\Pi}(\gamma)$ is uniquely determined by ${\bf l}(\gamma)$, and we denote this value by $\bold{hl}(\gamma)$ if no confusion is caused.

%For $\gamma\in {\bf \Gamma}_{R,\epsilon}$, we use ${\bf  \Pi}_{R,\epsilon}(\gamma)$ to denote the set of pairs $(\Pi,c)$ such that $\Pi\in {\bf  \Pi}_{R,\epsilon}$ and $c$ is an oriented boundary component of $\Sigma_{0,3}$ (simply write $c\in \partial \Pi$) such that $c$ is mapped to $\gamma$ via the immersion. Each $(\Pi,c)\in {\bf  \Pi}_{R,\epsilon}(\gamma)$ is called a {\it marked $(R,\epsilon)$-good pants} based on $\gamma$.

For $\gamma\in {\bf \Gamma}_{R,\epsilon}$, we can identify its unit normal bundle as $N^1(\gamma)=\mathbb{C}/({\bf l}(\gamma)\mathbb{Z}+2\pi i\mathbb{Z})$, then its half-unit normal bundle is defined to be $$N^1(\sqrt{\gamma})=\mathbb{C}/({\bf hl}(\gamma)\mathbb{Z}+2\pi i\mathbb{Z}).$$ Given $\Pi\in {\bf \Pi}_{R,\epsilon}$ with one cuff $\gamma=\gamma_i$, the pair of normal vectors $\vec{v}_{i-1},\vec{v}_{i+1}$ used to define ${\bf hl}_{\Pi}(\gamma)$ gives a unique vector ${\bf foot}_{\gamma}(\Pi)\in N^1(\sqrt{\gamma})$, called the {\it formal foot} of $\Pi$ on $\gamma$.

In \cite{KM}, to obtain the nearly geodesic subsurface, $(R,\epsilon)$-good pants are pasted along $(R,\epsilon)$-good curves with nearly $1$-shifts, rather than exactly matching seams along common cuffs. More precisely, in the nearly geodesic subsurface $S\looparrowright M$, for any two $(R,\epsilon)$-good pants $\Pi_1\in {\bf \Pi}_{R,\epsilon}, \Pi_2\in {\bf \Pi}_{R,\epsilon}$ in $S$ pasted along $\gamma\in {\bf \Gamma}_{R,\epsilon}$, such that $\gamma$ is an oriented boundary of $\Pi_1$, after identifying $N^1(\sqrt{\gamma})$ with $N^1(\sqrt{\bar{\gamma}})$ naturally, it is required that 
$$|{\bf foot}_{\gamma}(\Pi_1)-{\bf foot}_{\bar{\gamma}}(\Pi_2)-(1+\pi i)|<\frac{\epsilon}{R}\ \ \text{in}\ N^1(\sqrt{\gamma}).$$

This nearly $1$-shift is a crucial condition to guarantee the injectivity of $f_*: \pi_1(S)\rightarrow \pi_1(M)$. Kahn and Markovic showed that, for any $(R,\epsilon)$-good curve $\gamma$, the formal feet of $(R,\epsilon)$-good pants with cuff $\gamma$ are nearly evenly distributed along $\gamma$. So $M$ contains a large collection of $(R,\epsilon)$-good pants, and they can be pasted together by nearly $1$-shifts. Therefore, the asserted $\pi_1$-injective immersed closed subsurface can be constructed. 

\bigskip

In \cite{KW}, Kahn and Wright generalized Kahn and Markovic's surface subgroup theorem in closed hyperbolic $3$-manifolds (Theorem \ref{surface}) to cusped hyperbolic $3$-manifolds.

\begin{theorem}\label{surfaceincusp}(Theorem 1.1 of \cite{KW})
		Let $\Gamma<PSL_2(\mathbb{C})$ be a Kleinian group and assume that $\mathbb{H}^3/\Gamma$ has finite volume and is not compact. Then for all $K>1$, there exists $K$-quasi-Fuchsian (closed) surface subgroups in $\Gamma$.
	\end{theorem}
	
The main difficulty for proving Theorem \ref{surfaceincusp} is that, for cusped hyperbolic $3$-manifolds, good pants are not evenly distributed along good curves, especially for those good curves that run into cusps very deeply (with high heights). 

We first define the height function on a cusped hyperbolic $3$-manifold $M$. 
By the Margulis lemma, there exists $\epsilon_0>0$, such that the subset of $M$ consisting of points of injectivity radii at most $\epsilon_0$ is a disjoint union of solid tori and cusp neighborhoods of ends (simply called cusps). For any point in $M$ not belonging to any cusps, we define its height to be $0$. For any point $p$ in a cusp $C\subset M$, we define the height of $p$ to be the distance between $p$ and the boundary of $C$. For a compact geodesic segment or a closed geodesic in $M$, we define its {\it height} to be the maximal height of points on it. For a pair of $(R,\epsilon)$-good pants in $M$, we define its height to be the maximal height of its three cuffs. 

For any $h>0$, we use ${\bf \Gamma}_{R,\epsilon}^{<h}$ (${\bf \Pi}_{R,\epsilon}^{<h}$) to denote the set of all $(R,\epsilon)$-good curves (the set of all $(R,\epsilon)$-good pants) in $M$ with height less than $h$. We can define ${\bf \Gamma}_{R,\epsilon}^{\geq h}$ and ${\bf \Pi}_{R,\epsilon}^{\geq h}$ similarly.

To construct nearly geodesic subsurfaces in cusped hyperbolic $3$-manifolds, Kahn and Wright introduced a new geometric object called {\it $(R,\epsilon)$-good hamster wheel}.  For a positive integer $R$, let $Q_R$ be the oriented hyperbolic pants with cuff lengths $2,2$ and $2R$. The {\it $R$-perfect hamster wheel} $H_R$ is the cyclic $R$-sheet regular cover of $Q_R$ with $R+2$ boundary components, such that all cuffs of $H_R$ have length $2R$. 
%The preimage of the length-$2R$ cuff consists of $R$ components, and each of them is called an {\it inner cuff}; the other two cuffs of $H_R$ are called {\it outer cuffs}.
An {\it $(R,\epsilon)$-good hamster wheel} (or simply an $(R,\epsilon)$-hamster wheel) ${\bf H}$ is a map $f:H_R\to M$ up to homotopy, such that the image of each cuff of $H_R$ lies in ${\bf \Gamma}_{R,\epsilon}$, and $f$ is approximately a totally geodesic immersion. For each $(R,\epsilon)$-good hamster wheel ${\bf H}$ and each cuff $\gamma\in {\bf \Gamma}_{R,\epsilon}$ of ${\bf H}$, a foot ${\bf foot}_{\gamma}(\Pi)\in N^1(\sqrt{\gamma})$ can be defined that approximates the tangent direction of ${\bf H}$.
See Section 2.9 of \cite{KW} for the precise definition of $(R,\epsilon)$-hamster wheels and their feet.

%An $(R,\epsilon)$-good component is either an $(R,\epsilon)$-good pants or an $(R,\epsilon)$-good hamster wheel.

%An marked $(R,\epsilon)$-hamster wheel is a pair $({\bf H},c)$ consists of an $(R,\epsilon)$-hamster wheel ${\bf H}$ and an oriented boundary component $c\subset \partial H_R$. If $c$ is mapped to $\gamma\in {\bf \Gamma}_{R,\epsilon}$, then we say that the marked $(R,\epsilon)$-hamster wheel $({\bf H}, c)$ is based at $\gamma\in {\bf \Gamma}_{R,\epsilon}$. 
%For any marked $(R,\epsilon)$-hamster wheel $({\bf H}, c)$ based at $\gamma\in {\bf \Gamma}_{R,\epsilon}$, it is also equipped with a {\it slow and constant turning normal field} ${\bf v}$ on $N^1(\gamma)$ that approximates the (inward) tangent direction of ${\bf H}$. If $c$ is an inner cuff mapped to $\gamma\in {\bf \Gamma}_{R,\epsilon}$, there are two other inner cuffs that are $Ce^{-\frac{R}{2}}$-close to $c$, and they give two feet on $N^1(\gamma)$ that are close to the ${\bf hl}(\gamma)$-translation to each other. Then Kahn and Wright defined the {\it formal feet} ${\bf foot}_{\gamma}({\bf H},c)\in N^1(\sqrt{\gamma})$ to be an appropriate average of these two feet, and requires that it lies in the slow and constant turning normal field. See Section 2.9 of \cite{KW} for the precise definition of formal feets on inner cuffs of $(R,\epsilon)$-hamster wheels. Note that for each $(R,\epsilon)$-good pants, its formal foot on each cuff gives a slow and constant turning normal field.

In \cite{KW}, Kahn and Wright defined the {\it $(R,\epsilon)$-well-matched condition} for pasting finitely many $(R,\epsilon)$-good pants and $(R,\epsilon)$-good hamster wheels together in a nearly geodesic manner. 
%To paste $(\Sigma_1,c_1)$ and $(\Sigma_2,c_2)$ along $\gamma\in {\bf \Gamma}_{R,\epsilon}$, Kahn and Wright require:
%\begin{enumerate}
%\item If $(\Sigma_1,c_1)$ and $(\Sigma_2,c_2)$ have formal feet on $\gamma$ and $\bar{\gamma}$ respectively, it is required that $|{\bf foot}_{\gamma}(\Sigma_1,c_1)-{\bf foot}_{\bar{\gamma}}(\Sigma_2,c_2)-(1+\pi i)|<\frac{\epsilon}{R}$.
%\item Otherwise, the slow and constant turning fields form a bend of at most $100\epsilon$.
%\end{enumerate}
%According to \cite{KW}, for an $(R,\epsilon)$-good hamster wheel ${\bf H}$ and an outer cuff $c$ mapped to $\gamma\in {\bf \Gamma}_{R,\epsilon}$, we can define ${\bf foot}_{\gamma}({\bf H},c)$ to be any vector in $N^1(\sqrt{\gamma})$ lying in the slow and constant turning vector field on $\gamma$ associated to $({\bf  H},c)$. Unless otherwise indicated, we take ${\bf foot}_{\gamma}({\bf H},c)$ to be $\frac{\epsilon}{R}$-close to the tangent vector of an orthogeodesic from $c$ to an inner cuff of ${\bf H}$. Note that these feet on outer cuffs are not {\it formal feet}, and we use item (2) for the $(R,\epsilon)$-well-matched condition in this case.
A {\it good assembly} in a cusped hyperbolic $3$-manifold is a compact oriented subsurface (possibly with boundary) obtained by pasting finitely many $(R,\epsilon)$-good pants and $(R,\epsilon)$-good hamster wheels according to the $(R,\epsilon)$-well-matched condition. Then Kahn and Wright proved that an immersed subsurface in a cusped hyperbolic $3$-manifold arised from a good assembly is $\pi_1$-injective.

To construct a closed subsurface in a cusped hyperbolic $3$-manifold arised from a good assembly, Kahn and Wright defined a more complicated geometric object called an {\it umbrella}. An umbrella ${\bf U}$ consists of a compact planar surface $U$ decomposed as a finite union of subsurfaces homeomorphic to $H_R$ and a map $f:U\to M$, such that the restriction of $f$ on each $H_R$ subsurface (under the decomposition) gives an $(R,\epsilon)$-good hamster wheel, and these $(R,\epsilon)$-good hamster wheels are $(R,\epsilon)$-well-matched with each other. For each umbrella ${\bf U}$ and each cuff $\gamma\in {\bf \Gamma}_{R,\epsilon}$ of ${\bf U}$, we define ${\bf foot}_{\gamma}({\bf U})\in N^1(\sqrt{\gamma})$ to be the foot of the $(R,\epsilon)$-hamster wheel in ${\bf U}$ containing $\gamma$. Umbrellas are used to take care of the non-desired property that feet of good pants are not evenly distributed on $N^1(\sqrt{\gamma})$ for some $\gamma\in {\bf \Gamma}_{R,\epsilon}$, especially when $\gamma$ has high height. 

%Actually, when constructing an umbrella, we always paste inner cuffs of $(R,\epsilon)$-good hamster wheels with outer cuffs of $(R,\epsilon)$-good hamster wheels. In this case, since outer cuffs do not have formal feet, the $(R,\epsilon)$-well-matched condition only requires the $(100\epsilon)$-bending condition for slow and constant turning normal fields. Moreover, by the construction in Section 4.1 of \cite{KW}, for two adjacent $(R,\epsilon)$-good hamsters wheel sharing $\gamma\in {\bf \Gamma}_{R,\epsilon}$, we require that the basepoints of their feet on $\gamma$ have distance $\frac{\epsilon}{R}$-close to $1$ (the proof of Theorem 3.8 in \cite{KW}).

%An marked umbrella is a pair $({\bf U},c)$ consisting of an umbrella ${\bf U}$ and an oriented boundary component $c\subset \partial U$. If $c$ is mapped to $\gamma\in {\bf \Gamma}_{R,\epsilon}$, then we say that the marked umbrella $({\bf U}, c)$ is based at $\gamma$. For any boundary component $c$ of ${\bf U}$ mapped to $\gamma\in {\bf \Gamma}_{R,\epsilon}$, it is the boundary component of a unique $(R,\epsilon)$-good hamster wheel ${\bf H}$ contained in ${\bf U}$. Then we define ${\bf foot}_{\gamma}({\bf U},c)={\bf foot}_{\gamma}({\bf H},c)\in N^1(\sqrt{\gamma})$.
%Let $H_R\subset S$ be the subsurface of $S$ containing $c$, then the foot of ${\bf U}$ on $\gamma$, ${\bf foot}_{\gamma}({\bf U},c)\in N^1(\sqrt{\gamma})$, is defined to be any choice of ${\bf foot}_{\gamma}({\bf H},c)\in N^1(\sqrt{\gamma})$.

In \cite{KW}, Kahn and Wright took constants $h_T\geq 6\log{R}$ and $h_c\geq h_T+44\log{R}$. 
%Here the conditions $h_T\geq 6 \log{R}$ and $h_c\geq h_T+44\log{R}$ are given in Theorem 4.15 and equation (5.4.1) of \cite{KW} respectively.
Then they considered the collection of all $(R,\epsilon)$-good pants $\Pi$ with at least one cuff of height less than $h_c$.  For any $(\Pi,\gamma)$ such that $\Pi\in {\bf \Pi}_{R,\epsilon}^{\geq h_c}$ and $\gamma\in {\bf \Gamma}_{R,\epsilon}^{<h_c}$ is a cuff of $\Pi$, in Theorem 4.15 of \cite{KW}, Kahn and Wright constructed a $\mathbb{Q}_+$-combination of umbrellas $\hat{U}(\Pi,\gamma)$ with coefficients sum to $1$ such that the following hold:
\begin{enumerate}
\item As a $\mathbb{Q}_+$-linear combination of umbrellas, the boundary of $\hat{U}(\Pi,\gamma)$ contains one copy of $\gamma$, and all of its other boundary components have height less than $h_T$.
\item $\hat{U}(\Pi,\gamma)$ is $(R,\epsilon)$-well-matched with any $(R,\epsilon)$-good pants that is $(R,\epsilon)$-well-matched with $\Pi$ along $\gamma$.
\end{enumerate} 
Then they used $\hat{U}(\Pi,\gamma)$ to replace $\Pi$ in the above collection of good pants. 
% Let $c_0$ be the boundary component of $\hat{U}(\Pi,c,\gamma)$ mapped to $\gamma_0$ (with coefficents sum to $1$), since it is flexible to choose feet for umbrellas, we can and will take 
%$${\bf foot}_{\gamma}(\hat{U}(\Pi,c,\gamma),c_0)={\bf foot}_{\gamma}(\Pi,c ).$$ Then we choose other feet of $\hat{U}(\Pi,c,\gamma)$ arbitrarily.
 
 After the above replacement process, we obtain two finite linear combinations of $(R,\epsilon)$-good objects. The first one is the sum of all $(R,\epsilon)$-good pants in ${\bf \Pi}_{R,\epsilon}^{<h_c}$, and  the second one is the sum of $\mathbb{Q}_+$-linear combinations of umbrellas constructed above:
 $$A_0=\sum_{\Pi\in {\bf \Pi}_{R,\epsilon}^{<h_c}}\Pi,\ \ \ A_1=\sum_{\gamma_\in {\bf \Gamma}_{R,\epsilon}^{<h_c}}\sum_{\Pi \in {\bf \Pi}_{R,\epsilon}^{\geq h_c}, \gamma\subset \partial \Pi}\hat{U}(\Pi,\gamma).$$ 
 	%Now we rewrite $A=A_0+A_1$ as a $\mathbb{Q}_+$-linear combination of $(R,\epsilon)$-good pants and umbrellas $\sum_{i=1}^Na_i{\bf \Sigma}_i$, such that each $a_i\in \mathbb{Q}_+$ and each ${\bf \Sigma}_i$ is an $(R,\epsilon)$-good pants or an umbrella. For each ${\bf \Sigma}_i$, we consider it as a map $f_i:\Sigma_i\to M$ defined on an oriented surface $\Sigma_i$. Then for any $\gamma_0\in {\bf \Gamma}_{R,\epsilon}^{<h_c}$, we define $\partial A_{\gamma_0}$ to be the following measure on $N^1(\sqrt{\gamma_0})$: for any $B\subset N^1(\sqrt{\gamma_0})$, its measure is 
%$$\partial A_{\gamma_0}(B)=\sum_{i=1}^Na_i \cdot \#\{c\ |\ c\ \text{ is\ an\ oriented}\ \partial\text{-component\ of\ }\Sigma_i, f_i(c)=\gamma_0, {\bf foot}_{\gamma}({\bf \Sigma}_i,c)\in B\}.$$
Then Kahn and Wright proved that, for any $\gamma\in {\bf \Gamma}_{R,\epsilon}^{<h_c}$, the feet of $(R,\epsilon)$-pants and umbrellas in $A=A_0+A_1$ are evenly distributed on $N^1(\sqrt{\gamma})$. After eliminating denominators in $A$ by multiplying a large integer, they could paste good pants and umbrellas in $A\cup \bar{A}$ ($\bar{A}$ denotes the orientation reversal of $A$) to get the desired nearly geodesic closed subsurface.

\bigskip

\subsection{Panted cobordism groups of finite volume hyperbolic $3$-manifolds}\label{prepantedcobordism}

In \cite{LM}, Liu and Markovic introduced panted cobordism groups of closed oriented hyperbolic $3$-manifolds and computed these groups. In \cite{Sun4}, the author generalized some results in \cite{LM} to oriented cusped hyperbolic $3$-manifolds. In this section, we review these results and their consequence Proposition \ref{boundingsurface}, which is the main input from the good pants construction to this work.

We first fix a closed oriented hyperbolic $3$-manifold $M$, a small number $\epsilon>0$ and a large number $R>0$. Let $\mathbb{Z}{\bf \Gamma}_{R,\epsilon}$ be the free abelian group generated by ${\bf \Gamma}_{R,\epsilon}$, modulo the relation $\gamma+\bar{\gamma}=0$ for all $\gamma\in {\bf \Gamma}_{R,\epsilon}$. Here $\bar{\gamma}$ denotes the orientation reversal of $\gamma$. Let $\mathbb{Z}{\bf \Pi}_{R,\epsilon}$ be the free abelian group generated by ${\bf \Pi}_{R,\epsilon}$, modulo the relation $\Pi+\bar{\Pi}=0$ for all $\Pi\in {\bf \Pi}_{R,\epsilon}$. By taking the oriented boundary of $(R,\epsilon)$-good pants, we get a homomorphism $\partial:\mathbb{Z}{\bf \Pi}_{R,\epsilon}\to \mathbb{Z}{\bf \Gamma}_{R,\epsilon}$. The panted cobordism group $\Omega_{R,\epsilon}(M)$ is defined as the following in \cite{LM}.

\begin{definition}\label{closedcobordismgroup} 
The {\it panted cobordism group} $\Omega_{R,\epsilon}(M)$ is defined to be the cokernel of the homomorphism $\partial$, i.e. $\Omega_{R,\epsilon}(M)$ fits into the following exact sequence 
$$\mathbb{Z}{\bf \Pi}_{R,\epsilon}\xrightarrow{\partial} \mathbb{Z}{\bf \Gamma}_{R,\epsilon}\to \Omega_{R,\epsilon}(M)\to 0.$$
\end{definition}

To state the result in \cite{LM}, we need the following definition.

\begin{definition}\label{framebundle}
For an oriented hyperbolic $3$-manifold $M$ and a point $p\in M$, a \emph{special orthonormal frame} (or simply a frame) of $M$ at $p$ is a triple of unit tangent vectors $(\vec{t}_p,\vec{n}_p,\vec{t}_p\times \vec{n}_p)$ such that $\vec{t}_p,\vec{n}_p\in T_p^1M$ with $\vec{t}_p\perp \vec{n}_p$, and $\vec{t}_p\times \vec{n}_p\in T_p^1M$ is the cross product  with respect to the orientation of $M$. We use $\text{SO}(M)$ to denote the {\it frame bundle} of $M$ consisting of all special orthonormal frames of $M$.
\end{definition}

For simplicity, we denote each element in $\text{SO}(M)$ by its basepoint and the first two vectors of the frame, as $(p,\vec{t}_p,\vec{n}_p)$, since the third vector is determined by the first two. We call $\vec{t}_p$ and $\vec{n}_p$ the tangent vector and the normal vector of this frame, respectively.

In \cite{LM}, Liu and Markovic proved the following result on $\Omega_{R,\epsilon}(M)$.
	
\begin{theorem}\label{homology}(Theorem 5.2 of \cite{LM})
For any closed oriented hyperbolic $3$-manifold $M$, small enough $\epsilon>0$ depending on $M$ and large enough $R>0$ depending on $\epsilon$ and $M$, there is a natural isomorphism 
		$$\Phi:\Omega_{R,\epsilon}(M)\rightarrow H_1(\text{SO}(M);\mathbb{Z}).$$
	\end{theorem}
	
In \cite{Sun4}, the author generalized Theorem \ref{homology} to oriented cusped hyperbolic $3$-manifolds. The corresponding result in \cite{Sun4} has some height conditions on involved curves and pants, and we need the following definition.

For any $h'>h>0$, $\mathbb{Z}{\bf \Gamma}_{R,\epsilon}^{<h}$ is naturally a subgroup of $\mathbb{Z}{\bf \Gamma}_{R,\epsilon}^{<h'}$. For the boundary homomorphism $\partial: \mathbb{Z}{\bf \Pi}_{R,\epsilon}^{<h'}\to \mathbb{Z}{\bf \Gamma}_{R,\epsilon}^{<h'}$, we use $\mathbb{Z}{\bf \Pi}_{R,\epsilon}^{h,h'}$ to denote the $\partial$-preimage of $\mathbb{Z}{\bf \Gamma}_{R,\epsilon}^{<h}<\mathbb{Z}{\bf \Gamma}_{R,\epsilon}^{<h'}$ in $\mathbb{Z}{\bf \Pi}_{R,\epsilon}^{<h'}$. We first recall the following definition in \cite{Sun4}.

\begin{definition}\label{cuspedcobordismgroup}
For an oriented cusped hyperbolic $3$-manifold $M$ and any $h'>h>0$, we define the {\it $(R,\epsilon)$-panted cobordism group of height $(h,h')$}, denoted by $\Omega_{R,\epsilon}^{h,h'}(M)$, to be the cokernel of the homomorphism $\partial|: \mathbb{Z}{\bf \Pi}_{R,\epsilon}^{h,h'}\to  \mathbb{Z}{\bf \Gamma}_{R,\epsilon}^{<h}$. Thus $\Omega_{R,\epsilon}^{h,h'}(M)$ fits into the following exact sequence $$\mathbb{Z}{\bf \Pi}_{R,\epsilon}^{h,h'}\xrightarrow{\partial|}  \mathbb{Z}{\bf \Gamma}_{R,\epsilon}^{<h}\to \Omega_{R,\epsilon}^{h,h'}(M)\to 0. $$
\end{definition}

%More precisely, an $(R,\epsilon)$-multicurve $L$ of height at most $h$ represents the zero element in $\Omega_{R,\epsilon}^{h,h'}(M)$ if and only if there is an $(R,\epsilon)$-panted subsurface of height at most $h'$, such that its oriented boundary is $L$.

In \cite{Sun4}, the author proved the following analogy of Theorem \ref{homology} for oriented cusped hyperbolic $3$-manifolds.

\begin{theorem}\label{cuspedhomology} (Theorem 1.1 of \cite{Sun4})
For any oriented cusped hyperbolic 3-manifold M, any numbers
$\beta>\alpha\geq 4$ with $\beta-\alpha\geq 3$ and any $\epsilon\in (0,10^{-2})$, there exists $R_0 = R_0(M, \epsilon) > 0$,
such that for any $R > R_0$, we have a natural isomorphism
$$\Phi:\Omega_{R,\epsilon}^{\alpha\log{R},\beta\log{R}}(M)\to H_1(\text{SO}(M); \mathbb{Z}).$$
\end{theorem}

Moreover, for Theorem \ref{cuspedhomology} (and Theorem \ref{homology}), if we compose the isomorphism $\Phi:\Omega_{R,\epsilon}^{\alpha\log{R},\beta\log{R}}(M)\to H_1(\text{SO}(M); \mathbb{Z})$ with the homomorphism $\pi_*:H_1(\text{SO}(M);\mathbb{Z})\to H_1(M;\mathbb{Z})$ induced by the bundle projection, $\pi_*\circ \Phi$ maps the equivalent class of each $(R,\epsilon)$-multicurve to its homology class in $H_1(M;\mathbb{Z})$. 

To give the geometric meaning of $\Omega_{R,\epsilon}^{h,h'}(M)$, we define the following two types of subsurfaces in an oriented cusped hyperbolic $3$-manifold $M$. 
\begin{definition}\label{subsurfaces}
For any small $\epsilon>0$ and large number $R>0$, we define the following terms.
\begin{enumerate}
\item An {\it $(R,\epsilon)$-panted subsurface} in a hyperbolic $3$-manifold $M$ consists of a (possibly disconnected) compact oriented surface $F$ with a pants decomposition and an immersion $i: F\looparrowright M$, such that the restriction of $j$ to each pair of pants in the pants decomposition of $F$ gives a pair of $(R,\epsilon)$-good pants.
%\item An {\it $(R,\epsilon)$-good assembly} in a cusped hyperbolic $3$-manifold $M$ is a finite connection of $(R,\epsilon)$-good pants and $(R,\epsilon)$-good hamster wheels that are pasted together along boundary, such that the pasting satisfyies the $(R,\epsilon)$-well-matched condition.
\item If $R$ is also an integer, an {\it $(R,\epsilon)$-nearly geodesic subsurface} in a cusped hyperbolic $3$-manifold $M$ consists of a compact oriented surface $F$ decomposed as pants and $R$-hamster wheels (by a family of disjoint essential curves $\mathcal{C}$), and an immersion $i:F\looparrowright M$, such that the following hold. The restriction of $i$ on each pants or $R$-hamster wheel subsurface of $F$ is an $(R,\epsilon)$-good pants or an $(R,\epsilon)$-good hamster wheel respectively, and these $(R,\epsilon)$-good components are pasted together by the $(R,\epsilon)$-well-matched condition.
\end{enumerate}
\end{definition}

The $(R,\epsilon)$-panted subsurface was originally defined by Liu and Markovic in \cite{LM}, and it does not require any feet-matching condition when two $(R,\epsilon)$-good pants are pasted along an $(R,\epsilon)$-good curve. Geometrically, an $(R,\epsilon)$-multicurve $L\in \mathbb{Z}{\bf \Gamma}_{R,\epsilon}^{<h}$ represents the trivial element in $\Omega_{R,\epsilon}^{h,h'}(M)$ if and only if it bounds an $(R,\epsilon)$-panted subsurface of height at most $h'$. The definition of an $(R,\epsilon)$-nearly geodesic subsurface is same as an $(R,\epsilon)$-good assembly in \cite{KW}, but we stick to this terminology since we have been using it throughout \cite{Sun2, LS, Sun5}.
 Theorem 2.2 of \cite{KW} implies that, if $\epsilon>0$ is small enough and $R>0$ is large enough, an $(R,\epsilon)$-nearly geodesic subsurface is $\pi_1$-injective.

In Proposition 3.11 of \cite{Sun5}, the author proved the following result, which generalizes Corollary 2.11 of \cite{Sun2}. 

\begin{proposition}\label{boundingsurface} (Proposition 3.11 of \cite{Sun5})
Let $M$ be an oriented cusped hyperbolic $3$-manifold. Then for any constant $\alpha\geq 4$, any small $\epsilon> 0$ depending on $M$ and any large real number $R > 0$ depending on $M$ and $\epsilon$, the following statement holds. For any null-homologous oriented $(R,\epsilon)$-multicurve $L \in \mathbb{Z}{\bf \Gamma}_{R,\epsilon}^{<\alpha\log{R}}$, there is a nontrivial invariant $\sigma(L)\in \mathbb{Z}_2$ such that $\sigma(L_1\cup L_2)=\sigma(L_1)+\sigma(L_2)$ and the following hold.

If $\sigma(L)=0$, for any integer $R'\geq R$, $L$ is the oriented boundary of an immersed subsurface $f:S\looparrowright M$ satisfying the following conditions.
\begin{enumerate}
\item If we write $L$ as a union of its components $L=L_1\cup \cdots\cup L_k$, then $S$ is decomposed as oriented subsurfaces $S=(\cup_{i=1}^k \Pi_i)\cup S'$ with disjoint interior, such that $\Pi_i\cap \partial S$ is a single curve $c_i$ that is mapped to $L_i$.
\item The restriction $f|_{\Pi_i}:\Pi_i\looparrowright M$ is a pair of pants such that $|{\bf hl}_{\Pi_i}(L_i)-R|<\epsilon$, and $|{\bf hl}_{\Pi_i}(s)-R'|<\epsilon$ holds for any other component $s\subset \partial \Pi_i$.
\item If we fix a normal vector $\vec{v}_i\in N^1(\sqrt{L_i})$ for each component $L_i$ of $L$, then we can make sure $|{\bf foot}_{L_i}(\Pi_i)-\vec{v}_i|<\epsilon$ holds for all $i$.
\item The restriction $f|_{S'}:S'\looparrowright M$ is an oriented $(R',\epsilon)$-nearly geodesic subsurface.
\item For any component $s\subset S'\cap \Pi_i$ that is mapped to $\gamma\in {\bf \Gamma}_{R',\epsilon}$, we take its orientation induced from $\Pi_i$, then we have $$|{\bf foot}_{\gamma}(\Pi_i)-{\bf foot}_{\bar{\gamma}}(S')-(1+\pi i)|<\frac{\epsilon}{R}.$$
\end{enumerate}
\end{proposition}

We call the immersed pants in condition (2) {\it $(R,R',\epsilon)$-good pants}, and we call the immersed subsurface $f:S\looparrowright M$ constructed in Proposition \ref{boundingsurface} {\it an $(R',\epsilon)$-nearly geodesic subsurface with $(R,\epsilon)$-good boundary}. We can also assume that $S$ has no closed component.

 \begin{remark}\label{combinatorialdistance}
For the immersed subsurface $S\looparrowright M$ constructed in Corollary \ref{boundingsurface}, the collection of curves $\mathcal{C}\subset S$ (giving the decomposition of $S$) gives a graph-of-space structure on $S$ with dual graph $\Gamma$: each component of $S\setminus \mathcal{C}$ gives a vertex of $\Gamma$, and each component of $\mathcal{C}$ gives an edge of $\Gamma$. Let $v_i$ be the vertex of $\Gamma$ corresponding to $\Pi_i\subset S$. We can further modify the nearly geodesic subsurface $S\looparrowright M$ as in Section 3.1 Step IV of \cite{Sun1}, such that the combinatorial length of any topological essential path in $\Gamma$ from $v_i$ to $v_j$ (possibly $i=j$) is at least $R'e^{\frac{R'}{2}}$.

Moreover, if we endow $S$ a hyperbolic metric such that all $\partial$-curves of $S$ have length $2R$ and all curves in $\mathcal{C}$ have length $2R'$. Since all seams (shortest geodesic segments between boundary components) have lengths at least $e^{-\frac{R'}{2}}$ and each geodesic segment in a pair of pants or a hamster wheel from a cuff to itself has length at least $R$, any proper essential path in $S$ from $L_i$ to $L_j$ (possibly $i=j$) has length at least $R$.
\end{remark}

\subsection{The connection principle of finite volume hyperbolic $3$-manifolds}\label{connprinciple}

The connection principle is a fundamental tool that constructs geometric segments and $\partial$-framed segments in finite volume hyperbolic $3$-manifolds. The idea of connection principle was initiated in \cite{KM}, and the first officially stated connection principle is given in Lemma 4.15 of \cite{LM}. In \cite{Sun4}, the author proved a version of connection principle for oriented cusped hyperbolic $3$-manifolds, which is the connection principle will be used in this paper. In \cite{Liu} and \cite{Sun5}, connection principles with homological control in frame bundles are obtained for oriented closed and cusped hyperbolic $3$-manifolds respectively.

At first, we recall the definition of oriented {\it $\partial$-framed segments} and associated objects (Definition 4.1 of \cite{LM}). They are the geometric objects constructed by our connection principle.
\begin{definition}\label{segment} 
An {\it oriented $\partial$-framed segment} in $M$ is a triple 
$$\mathfrak{s}=(s,\vec{n}_{\text{ini}},\vec{n}_{\text{ter}}),$$ 
such that $s$ is an immersed oriented compact geodesic segment (simply called a geodesic segment), $\vec{n}_{\text{ini}}$ and $\vec{n}_{\text{ter}}$ are unit normal vectors of $s$ at its initial and terminal points respectively. 
\end{definition}

We have the following objects associated to an oriented $\partial$-framed segment $\mathfrak{s}$:
\begin{itemize}
\item The {\it carrier segment} of $\mathfrak{s}$ is the (oriented) geodesic segment $s$, and the {\it height} of $\mathfrak{s}$ is the height of $s$.	
\item The {\it initial endpoint} $p_{\text{ini}}(\mathfrak{s})$ and the {\it terminal endpoint} $p_{\text{ter}}(\mathfrak{s})$ are the initial and terminal points of $s$ respectively.
\item The {\it initial framing} $\vec{n}_{\text{ini}}(\mathfrak{s})$ and the {\it terminal framing} $\vec{n}_{\text{ter}}(\mathfrak{s})$ are the unit normal vectors $\vec{n}_{\text{ini}}$ and $\vec{n}_{\text{ter}}$ respectively.
\item The {\it initial direction} $\vec{t}_{\text{ini}}(\mathfrak{s})$ and the {\it terminal direction} $\vec{t}_{\text{ter}}(\mathfrak{s})$ are the unit tangent vectors in the direction of $s$ at $p_{\text{ini}}(\mathfrak{s})$ and $p_{\text{ter}}(\mathfrak{s})$ respectively.
\item The {\it initial frame} and the {\it terminal frame} of $\mathfrak{s}$ are $(p_{\text{ini}}(\mathfrak{s}),\vec{t}_{\text{ini}}(\mathfrak{s}),\vec{n}_{\text{ini}}(\mathfrak{s}))$ and $(p_{\text{ter}}(\mathfrak{s}),\vec{t}_{\text{ter}}(\mathfrak{s}),\vec{n}_{\text{ter}}(\mathfrak{s}))$ respectively.
\item The {\it length} $l(\mathfrak{s})\in(0,\infty)$ of $\mathfrak{s}$ is the length of its carrier $s$, the {\it phase} $\varphi(\mathfrak{s})\in \mathbb{R}/2\pi \mathbb{Z}$ of $\mathfrak{s}$ is the angle from the parallel transport of $\vec{n}_{\text{ini}}$ along $s$ to $\vec{n}_{\text{ter}}$.
\item The {\it orientation reversal} of $\mathfrak{s}=(s,\vec{n}_{\text{ini}},\vec{n}_{\text{ter}})$ is defined to be 
$$\bar{\mathfrak{s}}=(\bar{s},\vec{n}_{\text{ter}},\vec{n}_{\text{ini}}).$$ 
\item For any angle $\phi\in \mathbb{R}/2\pi \mathbb{Z}$, the frame rotation of $\mathfrak{s}$ by $\phi$ is defined to be 
$$\mathfrak{s}(\phi)=\big(s,\cos{\phi}\cdot\vec{n}_{\text{ini}}+\sin{\phi}\cdot(\vec{t}_{\text{ini}}\times \vec{n}_{\text{ini}}),\cos{\phi}\cdot\vec{n}_{\text{ter}}+\sin{\phi}\cdot (\vec{t}_{\text{ter}}\times \vec{n}_{\text{ter}})\big).$$
\end{itemize}

Now we state the connection principle in Theorem 3.7 of \cite{Sun5}. Since we do not need a homological statement in frame bundles, we only state a weaker version of condition (3) here.

\begin{theorem}\label{connectionprinciple} (Theorem 3.7 of \cite{Sun5})
Let $M$ be an oriented cusped hyperbolic $3$-manifold, and let ${\bf p}=(p,\vec{t}_p,\vec{n}_p), {\bf q}=(q,\vec{t}_p, \vec{n}_p)\in \text{SO}(M)$ be two frames based at $p,q\in M$ respectively. Let $\xi \in H_1(M,\{p,q\};\mathbb{Z})$ be a relative homology class with boundary $\partial \xi =[q]-[p]$. 

Then for any $\delta\in(0,10^{-2})$, there exists $T=T(M,\xi,\delta)$ depending on $M, \xi$ and $\delta$, such that for any $t>T$, there is a $\partial$-framed segment $\mathfrak{s}$ from $p$ to $q$ such that the following hold.
\begin{enumerate}
\item The heights of $p,q$ are at most $\log{t}$, and the height of $\mathfrak{s}$ is at most $2\log{t}$.
\item The length and phase of $\mathfrak{s}$ are $\delta$-close to $t$ and $0$ respectively. The initial and terminal frames of $\mathfrak{s}$ are $\delta$-close to ${\bf p}$ and ${\bf q}$ respectively.
\item The relative homology class of the carrier of $\mathfrak{s}$ equals $\xi\in H_1(M,\{p,q\};\mathbb{Z})$.
\end{enumerate}
\end{theorem}

\bigskip
\bigskip

\section{Preliminary on hyperbolic geometry}\label{prehyperbolic}

In this section, we give some geometric estimates on $\partial$-framed segments and geodesic segments, by using elementary hyperbolic geometry. Most of these results can be found in Section 3 of \cite{Sun4}, while some of them were originally proved in \cite{LM}. We have a new result (Proposition \ref{lengthshort}) that estimates the length of a consecutive chain of geodesic segments (see definition below), where some involved geodesic segments can be short.

We first need a few geometric definitions on $\partial$-framed segments from Section 4 of \cite{LM}. 
\begin{definition}\label{chainsandcycles}
Let $0<\delta<\frac{\pi}{3}$, $L>0$ and $0<\theta<\pi$ be three constants.
\begin{enumerate}
\item Two oriented $\partial$-framed segments $\mathfrak{s}$ and $\mathfrak{s}'$ are {\it $\delta$-consecutive} if the terminal point of $\mathfrak{s}$ is the initial point of $\mathfrak{s}'$, and the terminal framing of $\mathfrak{s}$ is $\delta$-close to the initial framing of $\mathfrak{s}'$. The {\it bending angle} between $\mathfrak{s}$ and $\mathfrak{s}'$ is the angle between the terminal direction of $\mathfrak{s}$ and the initial direction of $\mathfrak{s}'$.

\item A {\it $\delta$-consecutive chain} of oriented $\partial$-framed segments is a finite sequence $\mathfrak{s}_1,\cdots,\mathfrak{s}_m$ such that each $\mathfrak{s}_i$ is $\delta$-consecutive to $\mathfrak{s}_{i+1}$ for $i=1,\cdots,m-1$. It is a {\it $\delta$-consecutive cycle} if furthermore $\mathfrak{s}_m$ is $\delta$-consecutive to $\mathfrak{s}_1$. A $\delta$-consecutive chain or cycle is {\it $(L,\theta)$-tame} if each $\mathfrak{s}_i$ has length at least $2L$ and each bending angle is at most $\theta$.

\item For an $(L,\theta)$-tame $\delta$-consecutive chain $\mathfrak{s}_1,\cdots,\mathfrak{s}_m$, the {\it reduced concatenation}, denoted by $\mathfrak{s}_1\cdots \mathfrak{s}_m$, is the oriented $\partial$-framed segment defined as the following. The carrier segment of $\mathfrak{s}_1\cdots \mathfrak{s}_m$ is homotopic to the concatenation of carrier segments of $\mathfrak{s}_1,\cdots,\mathfrak{s}_m$, with respect to endpoints. The initial and terminal framings of $\mathfrak{s}_1\cdots\mathfrak{s}_m$ are the closest unit normal vectors to the initial framing of $\mathfrak{s}_1$ and the terminal framing of $\mathfrak{s}_m$ respectively.

\item For an $(L,\theta)$-tame $\delta$-consecutive cycle $\mathfrak{s}_1,\cdots,\mathfrak{s}_m$, the {\it reduced cyclic concatenation}, denoted by $[\mathfrak{s}_1\cdots \mathfrak{s}_m]$, is the oriented closed geodesic freely homotopic to the cyclic concatenation of carrier segments of $\mathfrak{s}_1,\cdots,\mathfrak{s}_m$, assuming it is not null-homotopic.
\end{enumerate}
\end{definition}

Without considering initial and terminal framings, we can also talk about the following terms on geodesic segments: consecutive geodesic segments and their bending angles, a consecutive chain and a consecutive cycle of geodesic segments and their $(L,\theta)$-tameness, the reduced concatenation of a consecutive chain of geodesic segments, and the reduced cyclic concatenation of a consecutive cycle of geodesic segments.

The following lemma from \cite{LM} is very useful for estimating length and phase of a concatenation of oriented $\partial$-framed segments. The function $I(\cdot)$ is defined by $I(\theta)=2\log{(\sec{\frac{\theta}{2}})}$.

\begin{lemma}\label{lengthphase} (Lemma 4.8 of \cite{LM})
Given any positive constants $\delta,\theta,L$ with $0<\theta<\pi$ and $L\geq I(\theta)+10\log{2}$, the following statements hold in any oriented hyperbolic $3$-manifold.
\begin{enumerate}
\item If $\mathfrak{s}_1,\cdots,\mathfrak{s}_m$ is an $(L,\theta)$-tame $\delta$-consecutive chain of oriented $\partial$-framed segments, denoting the bending angle between $\mathfrak{s}_i$ and $\mathfrak{s}_{i+1}$ by $\theta_i\in[0,\theta)$, then 
$$\Big|l(\mathfrak{s}_1\cdots \mathfrak{s}_m)-\sum_{i=1}^ml(\mathfrak{s}_i)+\sum_{i=1}^{m-1}I(\theta_i)\Big|<\frac{(m-1)e^{(-L+10\log{2})/2}\sin{(\theta/2)}}{L-\log{2}}$$
and 
$$\Big|\varphi(\mathfrak{s}_1\cdots\mathfrak{s}_m)-\sum_{i=1}^m\varphi(\mathfrak{s}_i)\Big|<(m-1)(\delta+e^{(-L+10\log{2})/2}\sin{(\theta/2)}),$$
where $|\cdot|$ on $\mathbb{R}/2\pi\mathbb{Z}$ is understood as the distance from zero valued in $[0,\pi]$.

\item If $\mathfrak{s}_1,\cdots,\mathfrak{s}_m$ is an $(L,\theta)$-tame $\delta$-consecutive cycle of oriented $\partial$-framed segments, denoting the bending angle between $\mathfrak{s}_i$ and $\mathfrak{s}_{i+1}$ by $\theta_i\in[0,\theta)$ with $\mathfrak{s}_{m+1}$ equal to $\mathfrak{s}_1$ by convention, then 
$$\Big|l([\mathfrak{s}_1\cdots \mathfrak{s}_m])-\sum_{i=1}^ml(\mathfrak{s}_i)+\sum_{i=1}^{m}I(\theta_i)\Big|<\frac{me^{(-L+10\log{2})/2}\sin{(\theta/2)}}{L-\log{2}}$$
and 
$$\Big|\varphi([\mathfrak{s}_1\cdots\mathfrak{s}_m])-\sum_{i=1}^m\varphi(\mathfrak{s}_i)\Big|<m(\delta+e^{(-L+10\log{2})/2}\sin{(\theta/2)}),$$
where $|\cdot|$ on $\mathbb{R}/2\pi\mathbb{Z}$ is understood as the distance from zero valued in $[0,\pi]$.
\end{enumerate}
\end{lemma}

For an $(L,\theta)$-tame $\delta$-consecutive chain of $\partial$-framed segments $\mathfrak{s}_1,\cdots,\mathfrak{s}_m$, we need the following lemma in \cite{Sun5} to bound the difference between initial frames of $\mathfrak{s}_1$ and $\mathfrak{s}_1\cdots \mathfrak{s}_m$.

\begin{lemma}\label{directiondifference} (Lemma 3.4 of \cite{Sun5})
Let $\delta,\theta,L$ be positive constants with $0<\theta<\pi$ and $L\geq I(\theta)+10\log{2}$. If $\mathfrak{s}_1,\cdots,\mathfrak{s}_m$ is an $(L,\theta)$-tame $\delta$-consecutive chain of oriented $\partial$-framed segments, then the distance between the initial frames of $\mathfrak{s}_1$ and $\mathfrak{s}_1\cdots \mathfrak{s}_m$ in $\text{SO}(M)_{p_{\text{ini}}(\mathfrak{s}_1)}$ is at most $8e^{-L}$.
\end{lemma}

The following lemma in \cite{Sun4} bounds the distance between a $\delta$-consecutive cycle of geodesic segments and the corresponding closed geodesic, which is useful for bounding heights of closed geodesics arised from geometric constructions. 

\begin{lemma}\label{distance} (Lemma 3.7 of \cite{Sun4})
Given any positive constants $\theta,L$ where $0<\theta<\pi$ and $L\geq 4(I(\theta)+10\log{2})$, the following statement holds in any oriented hyperbolic $3$-manifold. If $s_1,\cdots,s_m$ is an $(L,\theta)$-tame cycle of geodesic segments with $m\leq L$, and the bending angle between $s_i$ and $s_{i+1}$ lies in $[0,\theta)$ for each $i$, with $s_{m+1}$ equal to $s_1$ by convention, then the closed geodesic $[s_1\cdots s_m]$ lies in the $1$-neighborhood of the union $\cup_{i=1}^m s_i$. %Moreover, for
%$$K=\max_{i=1,\cdots,m}{\Big\{\log{\frac{1+\tan{\frac{\theta_i}{4}}}{1-\tan{\frac{\theta_i}{4}}}}\Big\}},$$ $[s_1\cdots s_m]$ and $\cup_{i=1}^m s_i$ lie in the 
%$(K+\frac{1}{10})$-neighborhood of each other.
\end{lemma}

The following result generalizes Lemma \ref{lengthphase} (1), which estimates the length of a consecutive chain of geodesic segments where some involved geodesic segments are short.

\begin{proposition}\label{lengthshort}
Given any positive constants $\theta,L$ with $0<\theta<\frac{\pi}{2}$ and 
$$L\geq \max\{12 I(\pi-\theta)+80\log{2}, 24\log{2}-16\log{(\frac{\pi}{2}-\theta)}\},$$ 
the following statement holds in any hyperbolic $3$-manifold. Let $s_1,\cdots,s_m$ be a consecutive chain of geodesic segments such that one of the following hold for each $i=1,\cdots,m-1$.
\begin{enumerate}
\item Either both $s_i$ and $s_{i+1}$ have length at least $L$, and the bending angle between $s_i$ and $s_{i+1}$ lies in $[0,\pi-\theta]$.
\item Or exactly one of $s_i$ and $s_{i+1}$ has length at least $L$, and the bending angle between $s_i$ and $s_{i+1}$ lies in $[0,\frac{\pi}{2}-\theta]$.
\end{enumerate}
Then we have
$$l(s_1\cdots s_m)\geq \frac{1}{2}\sum_{i=1}^ml(s_i).$$
\end{proposition}

\begin{proof}
At first, we lift the consecutive chain of geodesic segments $s_1,\cdots,s_m$ to the universal cover, and work on a consecutive chain of geodesic segments in $\mathbb{H}^3$. If $m=1$ or $2$, the result follows directly from the cosine law of hyperbolic geometry (see also estimates below), so we assume that $m\geq 3$.

Let $x_1$ be the initial point of $s_1$ and let $x_{m+1}$ be the terminal point of $s_m$. For any $i=2,\cdots,m$, let $x_i$ be the terminal point of $s_{i-1}$, which is also the initial point of $s_i$.
Let $y_1=x_1, y_{m}=x_{m+1}$, and let $y_i$ be the middle point of $s_i$ for each $i=2,\cdots,m-1$. For any $i=1,\cdots,m-1$, let $t_i$ be the geodesic segment from $y_i$ to $y_{i+1}$. We will use Lemma \ref{lengthphase} (1) to estimate $l(t_1\cdots t_{m-1})=l(s_1\cdots s_m)$.

We need to estimate lengths of $t_i$ and bending angles between $t_i$ and $t_{i+1}$. At first, we claim that 
\begin{align}\label{3.1}
l(t_i)\geq \frac{2}{3}(d(y_i,x_{i+1})+d(x_{i+1},y_{i+1}))\geq \frac{1}{3}L.
\end{align}
{\bf Case I.} Both $l(s_i)$ and $l(s_{i+1})$ are at least $\frac{L}{4}$. Then $d(y_i,x_{i+1}),d(x_{i+1},y_{i+1})\geq \frac{L}{8}$ and by assumption at least one of them is greater than $\frac{L}{2}$. By Lemma 4.10 (2) of \cite{LM}, we have 
\begin{align*}
l(t_i) &=d(y_i,y_{i+1})\geq d(y_i,x_{i+1})+d(x_{i+1},y_{i+1})-I(\pi-\angle y_ix_{i+1}y_{i+1})\\
&\geq d(y_i,x_{i+1})+d(x_{i+1},y_{i+1})-I(\pi-\theta)\\
& \geq \frac{2}{3}(d(y_i,x_{i+1})+d(x_{i+1},y_{i+1}))\geq \frac{1}{3}L.
\end{align*}
{\bf Case II.} Otherwise, one of $l(s_i),l(s_{i+1})$ is at most $\frac{L}{4}$, and we assume that $l(s_i)\leq \frac{L}{4}$. By assumption of this lemma, we have $l(s_{i+1})>L$ and $\angle y_ix_{i+1}y_{i+1}>\frac{\pi}{2}+\theta$. So we have $d(y_i,x_{i+1})\leq l(s_i)\leq \frac{L}{4}$ and $d(x_{i+1},y_{i+1})\geq \frac{1}{2}l(s_{i+1})\geq \frac{L}{2}$. Since $\angle y_ix_{i+1}y_{i+1}>\frac{\pi}{2}$, we have
\begin{align*}
l(t_i)=d(y_i,y_{i+1})\geq d(x_{i+1},y_{i+1})\geq \frac{2}{3}(d(y_i,x_{i+1})+d(x_{i+1},y_{i+1}))\geq \frac{1}{3}L.
\end{align*}
So equation (\ref{3.1}) holds in both cases.

Then we claim that 
$\angle y_iy_{i+1}x_{i+1}<\frac{\pi}{2}-\theta.$
In Case I above, we apply Lemma 4.10 (1) of \cite{LM} to get $$\angle y_iy_{i+1}x_{i+1}<e^{(-\frac{L}{8}+3\log{2})/2}<\frac{\pi}{2}-\theta.$$
In Case II above, $\angle y_ix_{i+1}y_{i+1}>\frac{\pi}{2}+\theta$ implies $\angle y_iy_{i+1}x_{i+1}<\frac{\pi}{2}-\theta$ directly. 

The same argument implies $\angle y_{i+2}y_{i+1}x_{i+2}<\frac{\pi}{2}-\theta$. So we get
\begin{align}\label{3.2}
\angle y_iy_{i+1}y_{i+2}\geq \pi -\angle y_iy_{i+1}x_{i+1}-\angle y_{i+2}y_{i+1}x_{i+2}\geq \pi-2(\frac{\pi}{2}-\theta)=2\theta.
\end{align}

For the consecutive chain of geodesic segments $t_1,\cdots,t_{m-1}$, by equations (\ref{3.1}) and (\ref{3.2}), each segment has length at least $\frac{1}{3}L$ and each bending angle is at most $\pi-2\theta$. By Lemma \ref{lengthphase} (1), we have 
\begin{align*}
& l(s_1\cdots s_m) = l(t_1\cdots t_{m-1}) \geq \sum_{i=1}^{m-1}l(t_i)-(m-2)I(\pi-2\theta)-(m-2)\frac{e^{(-\frac{L}{6}+10\log{2})/2}}{\frac{L}{6}-\log 2}\\
\geq\ & \sum_{i=1}^{m-1}l(t_i)-(m-2)(I(\pi-2\theta)+1)\geq \frac{3}{4}\sum_{i=1}^{m-1}l(t_i).
\end{align*}
Here the last inequality holds since $\frac{1}{4}l(t_i)\geq \frac{1}{12}L\geq I(\pi-2\theta)+1$ for each $t_i$, by equation (\ref{3.1}). By equation (\ref{3.1}) again, we have 
\begin{align*}
& l(s_1\cdots s_m) \geq \frac{3}{4}\sum_{i=1}^{m-1}l(t_i) \geq \frac{1}{2}\sum_{i=1}^{m-1}(d(y_i,x_{i+1})+d(x_{i+1},y_{i+1}))\\
=\ &\frac{1}{2}\Big(d(y_1,x_2)+\sum_{i=2}^{m-1}(d(x_i,y_i)+d(y_i,x_{i+1}))+d(x_m,y_m)\Big)
= \frac{1}{2}\sum_{i=1}^m l(s_i).
\end{align*}
\end{proof}

\bigskip
\bigskip

\section{Reduction of the domain and target in Theorem \ref{main2}}\label{reduction}

In this section, we prove a few preparational results that reduce the domain and target manifolds $M$ and $N$ in Theorem \ref{main2} to some convenient form. Recall that when talking about a cusped hyperbolic $3$-manifold, we mean a compact $3$-manifold with tori boundary whose interior admits a complete hyperbolic metric with finite volume, unless otherwise indicated.
%All the constructions in this section are topological and do not involve much geometry.

\subsection{Reducing the domain manifold $M$}
At first, we prove that any cusped hyperbolic $3$-manifold $M$ has a finite cover that satisfies a convenient homological condition. 

\begin{proposition}\label{homology2boundary}
For any oriented cusped hyperbolic $3$-manifold $M$, it has a finite cover $M'$ with two distinct boundary components $T_1,T_2\subset \partial M'$, such that the kernel of 
$$H_1(T_1\cup T_2;\mathbb{Z})\to H_1(M';\mathbb{Z})$$ 
contains an element $\alpha_1+\alpha_2\in H_1(T_1\cup T_2;\mathbb{Z})$ such that $0\ne \alpha_1\in H_1(T_1;\mathbb{Z})$ and $0\ne \alpha_2\in H_1(T_2;\mathbb{Z})$.
\end{proposition}

Note that we do need at least two boundary components of $M'$ in Proposition \ref{homology2boundary}. If all boundary components of $M$ are $H_1$-injective (e.g. $M$ is the complement of a two-component hyperbolic link with non-zero linking number), then any boundary component of any finite cover of $M$ is $H_1$-injective.

\begin{proof}
We take a boundary component $T$ of $M$, a slope $c$ on $T$, and a slope $l$ on $T$ that intersects $c$ once. 

By Proposition 4.6 of \cite{PW} and its proof, there is a geometrically finite $\pi_1$-injective connected oriented immersed subsurface $i:S\looparrowright M$, such that the following hold:
\begin{itemize}
\item $S$ has exactly two oriented boundary components $C_1,C_2$.
\item $i$ maps both $C_1$ and $C_2$ to $T$, and $i_*[C_1]=-i_*[C_2]=d[c]$ for some positive integer $d$.
\end{itemize}

For any positive integer $D$, let $T_{D}$ be the covering space of $T$ corresponding to $\langle dc, Dl\rangle<\langle c,l\rangle\cong \pi_1(T)$. Let $S_{D}$ be the $2$-complex obtained by pasting $S$ and two copies of $T_D$ (denoted by $T_{D,1}$ and $T_{D,2}$), such that $C_1,C_2\subset \partial S$ are pasted with curves in $T_{D,1},T_{D,2}$ corresponding to $\pm dc$, respectively. Then the map $i:S\looparrowright M$ and covering maps $T_{D,1},T_{D,2}\to T\subset M$ together induce a map $i_D:S_D\looparrowright M$. By Theorem 1.1 of \cite{MP}, when $D$ is large enough, $i_D$ is a $\pi_1$-injective map. 

Since hyperbolic $3$-manifold groups are LERF (\cite{Agol2}) and $S_D$ embeds into the covering space of $M$ corresponding to $(i_D)_*(\pi_1(S_D))<\pi_1(M)$, there is a finite cover $M'$ of $M$, such that $i_D:S_D\looparrowright M$ lifts to an embedding $\tilde{i}_D:S_D\hookrightarrow M'$. Since $\tilde{i}_D$ is an embedding, $T_1=\tilde{i}_D(T_{D,1})$ and $T_2=\tilde{i}_D(T_{D,2})$ are two distinct boundary components of $M'$. Since $\tilde{i}_D|_{S}:S\hookrightarrow M'$ gives an embedded oriented subsurface of $M'$ whose boundary consists of a pair of essential curves on $T_1$ and $T_2$ respectively, 
$$H_1(T_1\cup T_2;\mathbb{Z})\to H_1(M';\mathbb{Z})$$ 
is not injective. Let $\alpha_1$ be the intersection of $\partial S \cap T_1$, and let $\alpha_2$ be the intersection of $\partial S \cap T_2$, then $\alpha_1+\alpha_2\in H_1(T_1\cup T_2;\mathbb{Z})$ is an element in the kernel with the desired form.

\end{proof}

A similar argument as in Proposition \ref{homology2boundary} proves a similar result on certain mixed $3$-manifolds.

\begin{proposition}\label{homology2boundarymixed}
Let $M$ be an oriented mixed $3$-manifold with tori boundary such that $\partial M$ intersects with a hyperbolic piece $M_0$ of $M$. Then $M$ has a finite cover $M'$ with a hyperbolic piece $M_0'$ of $M'$, such that $M_0'\cap \partial M'$ contains two components $T_1,T_2$, and the kernel of 
$$H_1(T_1\cup T_2;\mathbb{Z})\to H_1(M_0';\mathbb{Z})$$ 
contains an element $\alpha_1+\alpha_2\in H_1(T_1\cup T_2;\mathbb{Z})$ such that $0\ne \alpha_1\in H_1(T_1;\mathbb{Z})$ and $0\ne \alpha_2\in H_1(T_2;\mathbb{Z})$.
\end{proposition}

\begin{proof}
Let $T$ be a component of $M_0\cap \partial M$.
By the proof of Proposition \ref{homology2boundary}, there is a $\pi_1$-injective $2$-complex $i_D:S_D\looparrowright M_0\hookrightarrow M$, such that $S_D$ is a union of three surfaces $S,T_{D,1},T_{D,2}$, and both $T_{D,1},T_{D,2}$ are mapped to $T$ via covering maps.

Since $i_D:S_D\looparrowright M$ is mapped into a hyperbolic piece $M_0$ of $M$, by \cite{Sun3}, $(i_D)_*(\pi_1(S_D))$ is a separable subgroup of $\pi_1(M)$. Since $S_D$ embeds into the covering space of $M$ corresponding to $(i_D)_*(\pi_1(S_D))<\pi_1(M)$, there is a finite cover $M'$ of $M$, such that $i_D:S_D\looparrowright M$ lifts to an embedding $\tilde{i}_D:S_D\hookrightarrow M'$. Since $S_D$ is connected, the image of $\tilde{i}_D$ is contained in a hyperbolic piece $M_0'\subset M'$. Since $\tilde{i}_D$ is an embedding, $T_1=\tilde{i}_D(T_{D,1})$ and $T_2=\tilde{i}_D(T_{D,2})$ are two distinct boundary components of $M_0'$ and they are both contained in $M_0'\cap \partial M'$. Similar to the proof of Proposition \ref{homology2boundary}, the existence of the subsurface $\tilde{i}_D|_S:S\to M'$ implies that 
$$H_1(T_1\cup T_2;\mathbb{Z})\to H_1(M';\mathbb{Z})$$ 
is not injective, and the kernel contains an element in the desired form.
\end{proof}

\bigskip

\subsection{Reducing the target manifold $N$}

To reduce the target manifold $N$, we first need to prove two topological lemmas. The first one is quite elementary, while the second one uses results on branched coverings between $3$-manifolds.

\begin{lemma}\label{solidtorus}
For any compact oriented $3$-manifold $N$ with nonempty tori boundary, there exists a proper map $g:N\to D^2\times S^1$ such that the following hold.
\begin{enumerate}
\item The degree of $g$ is at least $3$.
\item $g$ induces a surjective homomorphism on $\pi_1$.
\item The restriction of $g$ to each boundary component of $N$ is a covering map to $S^1\times S^1\subset D^2\times S^1$ of positive degree.
 \end{enumerate}
\end{lemma}
Conditions (2) and (3) in this lemma imply that $g$ is an ``allowable primitive map'', according to the terminology in \cite{Edm}.

\begin{proof}
Let the boundary components of $N$ be $T_1,\cdots,T_k$. 
For each $i=1,\cdots, k$, let $j_i:T_i\to N$ be the inclusion map, then the free rank of the image of $(j_i)_*:H_1(T_i;\mathbb{Z})\to H_1(N;\mathbb{Z})$ is at least $1$. So there exists a surjective homomorphism $\alpha:H_1(N;\mathbb{Z})\to \mathbb{Z}$ such that $\alpha\circ (j_i)_*:H_1(T_i;\mathbb{Z})\to \mathbb{Z}$ is nontrivial for all $i$. We consider $\alpha$ as an element in $\text{Hom}(H_1(N;\mathbb{Z});\mathbb{Z})\cong H^1(N;\mathbb{Z})$, then the dual of $\alpha$ in $H_2(N,\partial N;\mathbb{Z})$ is represented by a compact oriented (possibly disconnected) proper subsurface $\Sigma\subset N$. By doing surgery, we can assume that for each $i$, $\Sigma\cap T_i$ consists of parallel essential circles with consistent orientation. By our choice of $\alpha$, $\Sigma$ intersects with each $T_i$ nontrivially.

We take a proper map $h_0:\Sigma\to D^2=D^2\times \{pt\}$ of degree at least $3$, such that for any $T_i$, the restriction of $h_0|:\Sigma\cap T_i\to S^1=\partial D^2$ on each component of $\Sigma\cap T_i$ has the same positive degree $d_i$. This map $h_0$ can be constructed by first pinching $\Sigma\setminus \mathcal{N}(\partial \Sigma)$ to a point, with the resulting space being a one-point union of discs, then each disc is mapped to $D^2$ by a branched cover of positive degree. The restriction $h_0|:\Sigma\cap T_i\to S^1=S^1\times \{pt\}$ can be extended to a covering map $h_i:T_i\to S^1\times S^1$ of positive degree as following. Since each component $A$ of $T_i\setminus (\Sigma\cap T_i)$ is an annulus, and its two boundary components are mapped to $S^1\times \{pt\}$ with the same degree $d_i$, we define $h_i|_A$ to be a (orientation preserving) covering map to $(S^1\times S^1)\setminus (S^1\times \{pt\})$ of degree $d_i$. Then we have $h_i^{-1}(S^1\times \{pt\})=\Sigma\cap T_i$ and $\text{deg}(h_i)=\text{deg}(h_0|:\Sigma\cap T_i\to S^1)$.

The maps $h_0$ and $h_i,\ i=1,\cdots,k$ together give a map 
$$h:\Sigma\cup(\cup_{i=1}^k T_i)\to (D^2\times \{pt\})\cup (S^1\times S^1)\subset D^2\times S^1.$$ 
The map $h$ extends to a proper map $g:N\to D^2\times S^1$, since it extends to a neighborhood of $\Sigma\cup (\cup T_i)\subset N$ and $(D^2\times S^1)\setminus \big((D^2\times \{pt\})\cup (S^1\times S^1)\big)$ is a $3$-ball.

By construction, we have $\text{deg}(g)=\sum_{i=1}^k\text{deg}(h_i)=\text{deg}(h_0)\geq 3$, thus condition (1) holds. 
For the composition $N\xrightarrow{g}D^2\times S^1\to S^1$, the preimage of $pt\in S^1$ is exactly $\Sigma$, so
the induced homomorphism $H_1(N;\mathbb{Z})\to H_1(S^1;\mathbb{Z})$ is same with $\alpha:H_1(N;\mathbb{Z})\to \mathbb{Z}$. Since $\alpha$ is surjective, $g$ induces a surjective homomorphism on both $H_1$ and $\pi_1$, thus condition (2) holds. Since $g:N\to D^2\times S^1$ is an extension of $h_i:T_i\to S^1\times S^1$, condition (3) holds.

\end{proof}

\begin{lemma}\label{reduction1cusplemma}
For any compact oriented $3$-manifold $N$ with nonempty tori boundary, there exists a compact oriented $3$-manifold $M$ with connected torus boundary, such that $M$ virtually properly $2$-dominates $N$.
\end{lemma}

\begin{proof}

We first take a proper map $g:N\to D^2\times S^1$ satisfying the conclusion of Lemma \ref{solidtorus}.
These conditions make Theorem 4.1 of \cite{Edm} applicable, so $g: N\to D^2\times S^1$ is homotopic to a branched covering map relative to boundary. By Theorem 6.5 of \cite{BE}, $g: N\to D^2\times S^1$ is further homotopic to a simple branched covering, and we still denote this map by $g$. Here by simple branched covering, we mean that $g$ is a branched covering of degree $d\in \mathbb{Z}_{>0}$, such that for any $p\in D^2\times S^1$, $g^{-1}(p)$ consists of at least $d-1$ points. 

Note that if $g:N\to D^2\times S^1$ is a non-branched cover, then $N=D^2\times S^1$ and we can simply take $M=D^2\times S^1$. So we can assume that $g$ is a genuine branched cover.

Let $L\subset D^2\times S^1$ be the branched locus of the simple branched covering $g:N\to D^2\times S^1$, and we write $L=L_1\cup\cdots\cup L_k$ as a union of its components. We take a tubular neighborhood $\mathcal{N}(L)=\mathcal{N}(L_1)\cup \cdots\cup \mathcal{N}(L_k)$ of $L=L_1\cup\cdots\cup L_k$. On the torus $\partial \mathcal{N}(L_i)$, we take an oriented meridian $m_i$ that bounds the meridian disc of $\mathcal{N}(L_i)$, and take an oriented longitude $l_i$ that intersects with $m_i$ exactly once. For each $L_i$, we take $M_i$ to be a copy of $\Sigma_{1,1}\times S^1$, and take a linear homeomorphism 
$$\phi_i:\partial M_i=\partial \Sigma_{1,1}\times S^1\to \partial \mathcal{N}(L_i)$$
that maps oriented curves $\partial \Sigma_{1,1}\times \{*\}$ and $\{*\}\times S^1$ to $m_i$ and $l_i$ respectively. Then we take $M$ to be 
$$M=(D^2\times S^1\setminus \mathcal{N}(L)) \cup_{\{\phi_i\}_{i=1}^k}(\cup _{i=1}^k M_i).$$
Now we need to construct a finite cover $M'\to M$ and a degree-$2$ map $M'\to N$.

\bigskip

Let $\mathcal{L}=g^{-1}(L)\subset N$. Since $g:N\to D^2\times S^1$ is a branched covering, there is a tubular neighborhood $\mathcal{N}(\mathcal{L})$ of $\mathcal{L}$ in $N$, such that the restriction map $g|:N\setminus \mathcal{N}(\mathcal{L})\to D^2\times S^1\setminus \mathcal{N}(L)$ is a covering map of degree $d=\text{deg}(g)$.

Let $\mathcal{L}_i=g^{-1}(L_i)$. Since $g:N\to D^2\times S^1$ is a simple branched covering, there is a unique component $\mathcal{L}_i^0$ of $\mathcal{L}_i$ such that $g$ is locally a $2$-to-$1$ map near $\mathcal{L}_i^0$, and $g$ is a local homeomorphism near any point in $\mathcal{L}_i\setminus \mathcal{L}_i^0=\cup_{j=1}^{n_i}\mathcal{L}_i^j.$ The restriction of $g$ to $\partial \mathcal{N}(\mathcal{L}_i^0)\to \partial N(L_i)$ is a finite cover corresponding to a subgroup of $\pi_1(\partial N(L_i))$ in one of the following two types:
\begin{enumerate}
\item $\langle 2m_i, k_il_i\rangle <\pi_1(\partial \mathcal{N}(L_i))$ or
\item $ \langle2m_i,k_il_i+m_i\rangle<\pi_1(\partial \mathcal{N}(L_i))$
\end{enumerate}
for some positive integer $k_i$. For any $j\in \{1,\cdots,n_i\}$, the restriction of $g$ to $\partial \mathcal{N}(\mathcal{L}_i^j)\to \partial \mathcal{N}(L_i)$ is a finite cover corresponding to subgroup 
\begin{enumerate}
\item[(3)] $\langle m_i, k_i^jl_i\rangle<\pi_1(\partial \mathcal{N}(L_i))$
\end{enumerate} for some positive integer $k_i^j$.

In case (3), let $\tilde{m}_i^j,\tilde{l}_i^j\subset \partial \mathcal{N}(\mathcal{L}_i^j)$ be one (oriented) component of the preimage of $m_i,l_i\subset \partial \mathcal{N}(L_i)$, respectively. Then $\tilde{m}_i^j\to m_i$ and $\tilde{l}_i^j\to l_i$ are covering maps of degree $1$ and $k_i^j$ respectively, and these two curves intersect once on $\partial \mathcal{N}(\mathcal{L}_i^j)$. 
We take $M_i^j=\Sigma_{1,1}\times S^1$ and a degree-$k_i^j$ covering map 
$$p_i^j:M_i^j=\Sigma_{1,1}\times S^1\to M_i=\Sigma_{1,1}\times S^1,$$ 
that is the product of $id:\Sigma_{1,1}\to \Sigma_{1,1}$ and the degree-$k_i^j$ covering map $S^1\to S^1$. Let 
$$\psi_i^j:\partial M_i^j=\partial \Sigma_{1,1}\times S^1\to \partial N(\mathcal{L}_i^j)$$ 
be the linear homeomorphism that maps oriented curves $\partial \Sigma_{1,1}\times \{*\}$ and $\{*\}\times S^1$ to $\tilde{m}_i^j$ and $\tilde{l}_i^j$ respectively.
Then we have the following commutative diagram
\begin{align}\label{4.1}
\begin{diagram}
\partial M_i^j=\partial \Sigma_{1,1}\times S^1 & \rTo^{\psi_i^j} & \partial \mathcal{N}(\mathcal{L}_i^j) & \subset &  N\setminus N(\mathcal{L})\\
\dTo^{p_i^j|} & & \dTo^{g|} & & \\
\partial M_i=\partial \Sigma_{1,1}\times S^1 & \rTo^{\phi_i} & \partial \mathcal{N}(L_i) & \subset & D^2\times S^1\setminus N(L).
\end{diagram}
\end{align}

In case (1) above, let $\tilde{m}_i^0,\tilde{l}_i^0\subset \partial \mathcal{N}(\mathcal{L}_i^0)$ be one (oriented) component of the preimage of $m_i,l_i\subset \partial \mathcal{N}(L_i)$, respectively. Then $\tilde{m}_i^0\to m_i$ and $\tilde{l}_i^0\to l_i$ are covering maps of degree $2$ and $k_i$ respectively, and these two curves intersect once on $\partial N(\mathcal{L}_i^0)$. We take $M_i^0=\Sigma_{2,2}\times S^1$, and take a degree-$4k_i$ covering map 
$$p_i^0:M_i^0=\Sigma_{2,2}\times S^1\to M_i=\Sigma_{1,1}\times S^1.$$ 
It is the product of the degree-$4$ covering map $\Sigma_{2,2}\to \Sigma_{1,1}$
that factors through $\Sigma_{1,2}$ (which restricts to a degree-$2$ cover on each boundary component), and the degree-$k_i$ covering map $S^1\to S^1$. Let 
$$\psi_i^0:\partial M_i^0=\partial \Sigma_{2,2}\times S^1\to \partial \mathcal{N}(\mathcal{L}_i^0)$$ 
be a linear map that restricts to a homeomorphism on each component of $\partial M_i^0$, such that it maps each oriented boundary component of $\partial\Sigma_{2,2}\times \{*\}$ to $\tilde{m}_i^0$, and maps oriented curves $\{*\}\times S^1$ on both components of $\partial M_i^0$ to $\tilde{l}_i^0$. Then we have the following commutative diagram
\begin{align}\label{4.2}
\begin{diagram}
\partial M_i^0=\partial \Sigma_{2,2}\times S^1 & \rTo^{\psi_i^0} & \partial N(\mathcal{L}_i^0) & \subset &  N\setminus N(\mathcal{L})\\
\dTo^{p_i^0|} & & \dTo^{g|} & & \\
\partial M_i=\partial \Sigma_{1,1}\times S^1 & \rTo^{\phi_i} & \partial N(L_i) & \subset & D^2\times S^1\setminus N(L).
\end{diagram}
\end{align}

In case (2) above, let $\tilde{m}_i^0,\tilde{l}_i^0\subset \partial N(\mathcal{L}_i^0)$ be one (oriented) component of the preimage of $m_i,l_i\subset \partial N(L_i)$, respectively. There is actually a unique $\tilde{l}_i^0$. Then $\tilde{m}_i^0\to m_i$ and $\tilde{l}_i^0\to l_i$ are covering maps of degree $2$ and $2k_i$ respectively, and these two curves intersect (algebraically) twice on $\partial N(\mathcal{L}_i^0)$. Here $\tilde{l}_i^0$ corresponds to $2(k_il_i+m_i)-(2m_i)=2k_il_i\in \pi_1(\partial \mathcal{N}(L_i))$. We take $M_i^0=\Sigma_{2,2}\times I/(x,0)\sim(\phi(x),1)$, where $\phi:\Sigma_{2,2}\to \Sigma_{2,2}$ is the nontrivial deck transformation of the double cover $q:\Sigma_{2,2}\to \Sigma_{1,2}$. Note that $q$ restricts to a degree-$2$ cover on each boundary component of $\Sigma_{2,2}$. Then we take the degree-$4k_i$ covering map 
$$p_i^0:M_i^0=\Sigma_{2,2}\times I/\sim\ \to M_i=\Sigma_{1,1}\times S^1$$ 
that is a composition $\Sigma_{2,2}\times I/\sim \ \to \Sigma_{1,2}\times S^1 \to \Sigma_{1,1}\times S^1$. Here the first map takes the  double covering map $q:\Sigma_{2,2}\to \Sigma_{1,2}$ on each fiber and takes the identity map on the base $S^1$, the second map is the product of the double cover $\Sigma_{1,2}\to \Sigma_{1,1}$ and the degree-$k_i$ covering map $S^1\to S^1$. 
Then each oriented component of the $p_i^0$-preimage of $\partial \Sigma_{1,1}\times \{*\}$ is a component of $\partial \Sigma_{2,2}\times \{*\}$. Each oriented component of the $p_i^0$-preimage of $\{*\}\times S^1$ is a flow line of $M_i^0$ along the $I$-direction, and it intersects the corresponding component of $\partial \Sigma_{2,2}\times \{*\}$ algebraically twice.

There exists a linear map $$\psi_i^0:\partial M_i^0=\partial \Sigma_{2,2}\times I/\sim \to \partial \mathcal{N}(\mathcal{L}_i^0)$$ 
that restricts to a homeomorphism on each component of $\partial M_i^0$, such that it maps each oriented boundary component of $\partial \Sigma_{2,2}\times \{*\}$ to $\tilde{m}_i^0$, and maps an $I$-flow line on each component of $\partial M_i^0$ to $\tilde{l}_i^0$. Then we have the following commutative diagram
\begin{align}\label{4.3}
\begin{diagram}
\partial M_i^0 & \rTo^{\psi_i^0} & \partial N(\mathcal{L}_i^0) & \subset &  N\setminus \mathcal{N}(\mathcal{L})\\
\dTo^{p_i^0|} & & \dTo^{g|} & & \\
\partial M_i & \rTo^{\phi_i} & \partial N(L_i) & \subset & D^2\times S^1\setminus \mathcal{N}(L).
\end{diagram}
\end{align}

\bigskip

Now we take two copies of $N\setminus N(\mathcal{L})$ and denote them by $(N\setminus N(\mathcal{L}))_1$ and $(N\setminus N(\mathcal{L}))_2$ respectively. For any $i=1,\cdots,k$ and $j=1,\cdots,n_i$, we take two copies of $M_i^j$ and denote them by $(M_i^j)_1$ and $(M_i^j)_2$. For any $i=1,\cdots,k$, $M_i^0$ has two boundary components, and we denote them by $(\partial M_i^0)_1$ and $(\partial M_i^0)_2$ respectively.  Then we take $M'$ to be the union of following manifolds
$$(N\setminus \mathcal{N}(\mathcal{L}))_1, (N\setminus \mathcal{N}(\mathcal{L}))_2, (M_i^j)_1, (M_i^j)_2, M_i^0\ \text{for}\ i=1,\cdots,k,\ j=1,\cdots,n_i,$$
by pasting maps 
$$(\psi_i^j)_1:\partial (M_i^j)_1\to (\partial \mathcal{N}(\mathcal{L}_i^j))_1 \subset (N\setminus \mathcal{N}(\mathcal{L}))_1,\ (\psi_i^j)_2:\partial (M_i^j)_2\to (\partial \mathcal{N}(\mathcal{L}_i^j))_2\subset (N\setminus \mathcal{N}(\mathcal{L}))_2,$$
$$(\psi_i^0|)_1:(\partial M_i^0)_1\to (\partial \mathcal{N}(\mathcal{L}_i^0))_1 \subset (N\setminus \mathcal{N}(\mathcal{L}))_1,\ (\psi_i^0|)_2:(\partial M_i^0)_2\to (\partial \mathcal{N}(\mathcal{L}_i^0))_2\subset (N\setminus \mathcal{N}(\mathcal{L}))_2.$$ 
The homeomorphisms $(\psi_i^j)_1,(\psi_i^j)_2$ denote copies of the map $\psi_i^j$ on the corresponding copy of $\partial M_i^j$, while $(\psi_i^0|)_1$ and $(\psi_i^0|)_2$ denote the restriction of $\psi_i^0$ on the corresponding component of $\partial M_i^0$.

The covering map $\pi:M'\to M$ is defined by the following covering maps on pieces of $M'$:
\begin{itemize}
\item $g|$ maps $(N\setminus \mathcal{N}(\mathcal{L}))_1,(N\setminus \mathcal{N}(\mathcal{L}))_2 \subset M'$ to $D^2\times S^2\setminus \mathcal{N}(L)\subset M$,
\item $p_i^j$ maps $(M_i^j)_1,(M_i^j)_2\subset M'$ to $M_i\subset M$ for $i=1,\cdots,k$ and $j=1,\cdots,n_i$,
\item $p_i^0$ maps $M_i^0\subset M'$ to $M_i\subset M$ for $i=1,\cdots,k$.
\end{itemize}
Here $\pi$ is a well-defined map because of three commutative diagrams (\ref{4.1}), (\ref{4.2}), (\ref{4.3}).

The degree-$2$ map $f:M'\to N$ is defined by the following maps on pieces of $M'$:
\begin{itemize}
\item The identity map that maps $(N\setminus \mathcal{N}(\mathcal{L}))_1,(N\setminus \mathcal{N}(\mathcal{L}))_2 \subset M$ to $N\setminus \mathcal{N}(\mathcal{L})\subset N$.
\item A pinching map that maps $(M_i^j)_1,(M_i^j)_2=\Sigma_{1,1}\times S^1 \subset M'$ to $\mathcal{N}(\mathcal{L}_i^j)=D^2\times S^1\subset N$. This map is the product of a degree-$1$ piniching map $\Sigma_{1,1}\to D^2$ and the identity map $S^1\to S^1$.
\item A pinching map that maps $M_i^0=\Sigma_{2,2}\times I/\sim \ \subset M'$ to $\mathcal{N}(\mathcal{L}_i^0)=D^2\times S^1\subset N$, here $\sim$ is induced by the identity map of $\Sigma_{2,2}$ or the nontrivial deck transformation of $\Sigma_{2,2}\to \Sigma_{1,2}$. Here we take a fixed degree-$2$ pinching map $\Sigma_{2,2}\to D^2$ on each fiber, such that it commutes with monodromy homeomorphisms of $M_i^0$ and $\mathcal{N}(\mathcal{L}_i^0)$, and restricts to a homeomorphism on each boundary component. 
\end{itemize}

Then $f$ is a degree-$2$ proper map from $M'$ to $N$.

\end{proof}

Now we are ready to prove the following result.

\begin{proposition}\label{reduction1cusp}
For any compact oriented $3$-manifold $N$ with nonempty tori boundary, there exists a  one-cusped oriented hyperbolic $3$-manifold $M$, such that $M$ virtually properly $2$-dominates $N$.
\end{proposition}

\begin{proof}

At first, by Lemma \ref{reduction1cusplemma}, $N$ is virtually $2$-dominated by a compact oriented $3$-manifold $N'$ with connected torus boundary. By Lemma 4.1 of \cite{Sun5}, $N'$ is $1$-dominated by a compact oriented irreducible $3$-manifold $N''$ with connected torus boundary.

This result follows from the proof of Proposition 3.2 of \cite{BW}, although the result in \cite{BW} is only stated for closed $3$-manifolds. By Theorem 7.2 of \cite{Mye}, there exists a hyperbolic knot $K\subset N''$ that is null-homotopic in $N''$. We take $M$ to be a hyperbolic Dehn-filling of $N''\setminus \mathcal{N}(K)$, then $M$ is a one-cusped hyperbolic $3$-manifold. The proof of Proposition 3.2 of \cite{BW} constructs a degree-$1$ map $f:M\to N''$. More precisely, the map $f$ is identity on $N''\setminus \mathcal{N}(K)$, it extends to the meridian disc of the filled-in solid torus since $K$ is null-homotopic in $N''$, and it extends to the whole solid torus since $N''$ is irreducible.
\end{proof}

\bigskip
\bigskip

\section{Topological construction of virtual domination}\label{topo2}

In this section, we give the topological part of the proof of Theorem \ref{main2}, and we also point out how to modify the works to prove Theorem \ref{mixed}.

To prove Theorem \ref{main2}, it suffices to prove the following result. 

\begin{theorem}\label{main3}
Let $M$ be a compact oriented hyperbolic $3$-manifold, such that $\partial M$ has two components $T_1,T_2$ and the kernel of $H_1(T_1\cup T_2;\mathbb{Z})\to H_1(M;\mathbb{Z})$ contains an element $\alpha_1+\alpha_2\in H_1(T_1\cup T_2;\mathbb{Z})$ with $0\ne \alpha_1\in H_1(T_1;\mathbb{Z})$ and $0\ne \alpha_2\in H_1(T_2;\mathbb{Z})$. Let $N$ be a compact oriented hyperbolic $3$-manifold with connected torus boundary. Then $M$ has a finite cover $M'$, such that there is a proper map $f:M'\to N$ with $\text{deg}(f)\in \{1,2,4\}$.
\end{theorem}

We first prove that Theorem \ref{main3} implies Theorem \ref{main2}.

\begin{proof}[Proof of Theroem \ref{main2} (by assuming Theorem \ref{main3}).]

Let $M$ be a compact oriented cusped hyperbolic $3$-manifold, and let $N$ be a compact oriented $3$-manifold with tori boundary, as in Theorem \ref{main2}. By Lemmas \ref{homology2boundary} and \ref{reduction1cusp}, there are compact oriented $3$-manifold $M_1$ and $N_1$ with tori boundary, such that the following hold:
\begin{enumerate}
\item $M_1$ is a finite cover of $M$ and $\partial M_1$ contains two components $T_1,T_2$, such that the kernel of $H_1(T_1\cup T_2;\mathbb{Z})\to H_1(M_1;\mathbb{Z})$ contains an element $\alpha_1+\alpha_2\in H_1(T_1\cup T_2;\mathbb{Z})$ with $0\ne \alpha_1\in H_1(T_1;\mathbb{Z})$ and $0\ne \alpha_2\in H_1(T_2;\mathbb{Z})$.
\item $N_1$ is an oriented one-cusped hyperbolic $3$-manifold that admits a finite cover $p:N_2\to N_1$ and a degree-$2$ map $g:N_2\to N$. 
\end{enumerate}
Theorem \ref{main3} implies that $M_1$ has a finite cover $M_2$ and there is a proper map $h:M_2\to N_1$ such that $\text{deg}(h)\in \{1,2,4\}$.

Let $q:M_3\to M_2$ be the covering space of $M_2$ corresponding to $(h_*)^{-1}(p_*(\pi_1(N_2)))$, then we have the following commutative diagram:
\begin{diagram}
M_3 & \rTo^{h'} & N_2 \\
\dTo^{q} &      & \dTo_{p} \\
M_2 & \rTo^{h} & N_1.
\end{diagram}
Here 
$$\text{deg}(q)=[\pi_1(M_2):q_*(\pi_1(M_3))]=[\pi_1(M_2):(h_*)^{-1}(p_*(\pi_1(N_2)))]$$ is a factor of
$\text{deg}(p)=[\pi_1(N_1):p_*(\pi_1(N_2))].$
Since $\text{deg}(h)\cdot \text{deg}(q)=\text{deg}(p)\cdot \text{deg}(h')$, $\text{deg}(h')$ is a factor of $\text{deg}(h)\in \{1,2,4\}$. So $f'=g\circ h':M_3\xrightarrow{h'}N_2\xrightarrow{g} N$ is a map such that $\text{deg}(f')\in \{2,4,8\}$.

%Let $q:N_3\to N$ be the finite cover corresponding to $f'_*(\pi_1(M_3))$, then we have a $\pi_1$-surjective lifted map $f'':M_3\to N_3$. Since $N_3$ has nonempty tori boundary, we have $b_1(N_3)\geq 1$, so it has a finite cover $r:N_4\to N_3$ of degree $\frac{8}{\text{deg}(f')}$. Let $M'$ be the finite cover of $M_3$ corresponding to $(f''_*)^{-1}(r_*(\pi_1(N_4)))$, and let $f''':M'\to N_4$ be the lifted map, then we have the following diagram.
%\begin{diagram}
%M'   & \rTo^{f'''}  & N_4 \\
%\dTo &             & \dTo^{r} \\
%M_3 & \rTo^{f''} & N_3 \\
  %    &  \rdTo_{f'}     & \dTo_{q} \\
%      &           & N.
%\end{diagram}

%Then we take $f=q\circ r\circ f''':M'\to N$. Since $\text{deg}(f'')=\text{deg}(f''')$ holds, we have 
%$$\text{deg}(f)=\text{deg}(q)\text{deg}(r)\text{deg}(f''')=\text{deg}(q)\text{deg}(f'')\text{deg}(r)=\text{deg}(f')\text{deg}(r)=8.$$ So $f$ is the desired map.

Since $M_3$ has tori boundary, $b_1(M_3)\geq 1$ holds. So $M_3$ has a cyclic cover $r:M'\to M_3$ of degree $\frac{8}{\text{deg}(f')}$. Then $f=f'\circ r:M'\to N$ is a proper map of degree $8$, as desired.
\end{proof}

The following three subsections are devoted to prove Theorem \ref{main3}, modulo a $\pi_1$-injectivity result (Theorem \ref{pi1injectivity}). We always assume that $M$ and $N$ satisfy the assumption of Theorem \ref{main3}.

\subsection{Initial data of the construction}\label{initialdata}

In this section, we first give some geometric data deduced from $N$, then we give some related geometric notions on $M$.

We first prove the following lemma on triangulation of flat tori. The resulting triangulation will be the restriction of our desired triangulation of a cusped hyperbolic $3$-manifold to its horotorus.
\begin{lemma}\label{existencealmostequilateral}
For any flat torus $T$, any primitive closed geodesic $l$ on $T$, and any $\epsilon\in (0,0.1)$, $T$ has a geometric triangulation such that the following hold.
\begin{enumerate} 
\item $l$ is contained in the $1$-skeleton of this triangulation.
\item There exists $r\in (0,\epsilon)$, such that all edges have length in $[r,(1+2\epsilon)r)$.
\item Any inner angle of any triangle is $\epsilon$-close to $\frac{\pi}{3}$. 
\end{enumerate}
\end{lemma}

\begin{proof}
Up to scaling, we can assume that $T$ is isometric to $\mathbb{C}/\mathbb{Z}\oplus \mathbb{Z}z_0$ for some complex number $z_0\in \mathbb{C}$ with $\text{Im}(z_0)>0$, and the closed geodesic $l$ corresponds to the edge from $0$ to $1$.

We consider the lattice $\Lambda_0=\mathbb{Z}\oplus \mathbb{Z}\omega_0$ of $\mathbb{C}$ for $\omega_0=\frac{1+\sqrt{3}i}{2}$. For large $N\in \mathbb{N}$, we take $\omega$ to be the point in $\Lambda_0$ closest to $Nz_0$. Note that $\text{Im}(\omega)>0$ holds if $N>\frac{1}{\text{Im}(z_0)}$.

Let $T:\mathbb{C}\to \mathbb{C}$ be the linear transformation that maps $N,\omega\in \Lambda_0$ to $N,Nz_0\in \mathbb{Z}\oplus \mathbb{Z}z_0$ respectively. Then $\frac{1}{N}T$ maps $\Lambda_0$ to a lattice of $\mathbb{C}$ that contains $\mathbb{Z}\oplus \mathbb{Z}z_0$. The equilateral triangulation of $\mathbb{C}$ corresponding to $\Lambda_0$ induces a triangulation of $T$, and $l$ is contained in the $1$-skeleton. 

It is straight forward to check that, if $N$ is large enough (say $N>\frac{6}{\epsilon \cdot \text{Im}(z_0)}$), then all inner angles of the above triangulation of $T$ are $\epsilon$-close to $\frac{\pi}{3}$. So we get an $\epsilon$-almost-equilateral geodesic triangulation of $T$, such that all triangles are isometric to each other. 

For each triangle in this triangulation of $T$, we take middle points of its edges and divide it into four smaller triangles, to get a finer $\epsilon$-almost-equilateral geodesic triangulation, such that all smaller triangles are similar to the original ones. We do this process repeatedly so that all edges have length at most $\epsilon$. Let $r$ be the length of the shortest edge, then the Euclidean sine law implies that all edges have length in $[r,(1+2\epsilon)r)$.
\end{proof}

Let $\epsilon_0>0$ be a constant smaller than the $3$-dimensional Margulis constant. For the one-cusped hyperbolic $3$-manifold $N$ as in Theorem \ref{main3} (considered as an open complete Riemannian manifold), let $N_c$ be the complement of the cusp end with injectivity radius at most $\frac{\epsilon_0}{10}$, and let $T_c$ be the boundary of $N_c$. By a classical application of the Lefschetz duality, there is a primitive closed geodesic $l$ on $T_c$ that spans $\text{ker}(H_1(T_c;\mathbb{R})\to H_1(N_c;\mathbb{R}))$.

\begin{construction}\label{triangulation}
We construct a geometric triangulation of a compact sub-manifold of $N$ containing $N_c$, whose geometry near $\partial N_c$ is quite special. Let $\epsilon\in (0,\frac{\epsilon_0}{100})$ be a constant smaller than the injectivity radius of $N_c$. 
\begin{enumerate} 

\item By Lemma \ref{existencealmostequilateral} (applied to $\frac{\epsilon}{100}$), the horotorus $T_c$ has a geometric triangulation, such that $l$ is contained in the $1$-skeleton, all edges have length in $[r,(1+2\frac{\epsilon}{100})r)$ for some $r\in(0,\frac{\epsilon}{100})$, and all inner angles of triangles are $\frac{\epsilon}{100}$-close to $\frac{\pi}{3}$.

\item Let $T_c'$ be the horotorus in $N_c$ that has distance $\frac{\sqrt{6}}{3}r$ from $T_c$. For any triangle $\Delta$ in $T_c$, we take its circumcenter, and let $v_{\Delta}$ be its closest point on $T_c'$. For any vertex $n$ of the triangulation of $T_c$, let its closest point on $T_c'$ be $v_n$.

\item We connect $v_{\Delta}$ to the three vertices of $\Delta$ and obtain a (hyperbolic) tetrahedron in $N$. All inner angles of all triangles on the boundary of this tetrahedron are $\epsilon$-close to $\frac{\pi}{3}$. Note that the triangle $\Delta$ contained in $T_c$ is not a face of this tetrahedron, since $\Delta$ is only a Euclidean triangle but not a hyperbolic one.

\item For any two triangles $\Delta_1, \Delta_2$ contained in $T_c$ that share an edge, we add an edge connecting $v_{\Delta_1}$ and $v_{\Delta_2}$. This edge and the edge $\Delta_1\cap \Delta_2$ together give a hyperbolic tetrahedron in $N$. 

\item For any vertex $n$ and two triangles $\Delta_1, \Delta_2$ contained in $T_c$, such that $\Delta_1\cap \Delta_2$ is an edge containing $n$, we get a hyperbolic tetrahedron with vertices $n,v_n,v_{\Delta_1},v_{\Delta_2}$. A picture of the tetrahedra we have constructed can be found in Figure \ref{figure1}, which gives a triangulation of a compact submanifold containing $T_c$.

\item Let $N_{collar}$ be the union of tetrahedra (with disjoint interior) constructed in previous steps, and let $N_0$ be the union of $N_c$ and $N_{collar}$, which is compact and is a deformation retract of $N$. Then we extend the above triangulation of $N_{collar}$ to a geometric triangulation of $N_0$. 
\end{enumerate}
\end{construction}

\begin{figure}
\centering
\label{new picture}
\def\svgwidth{.7\textwidth}
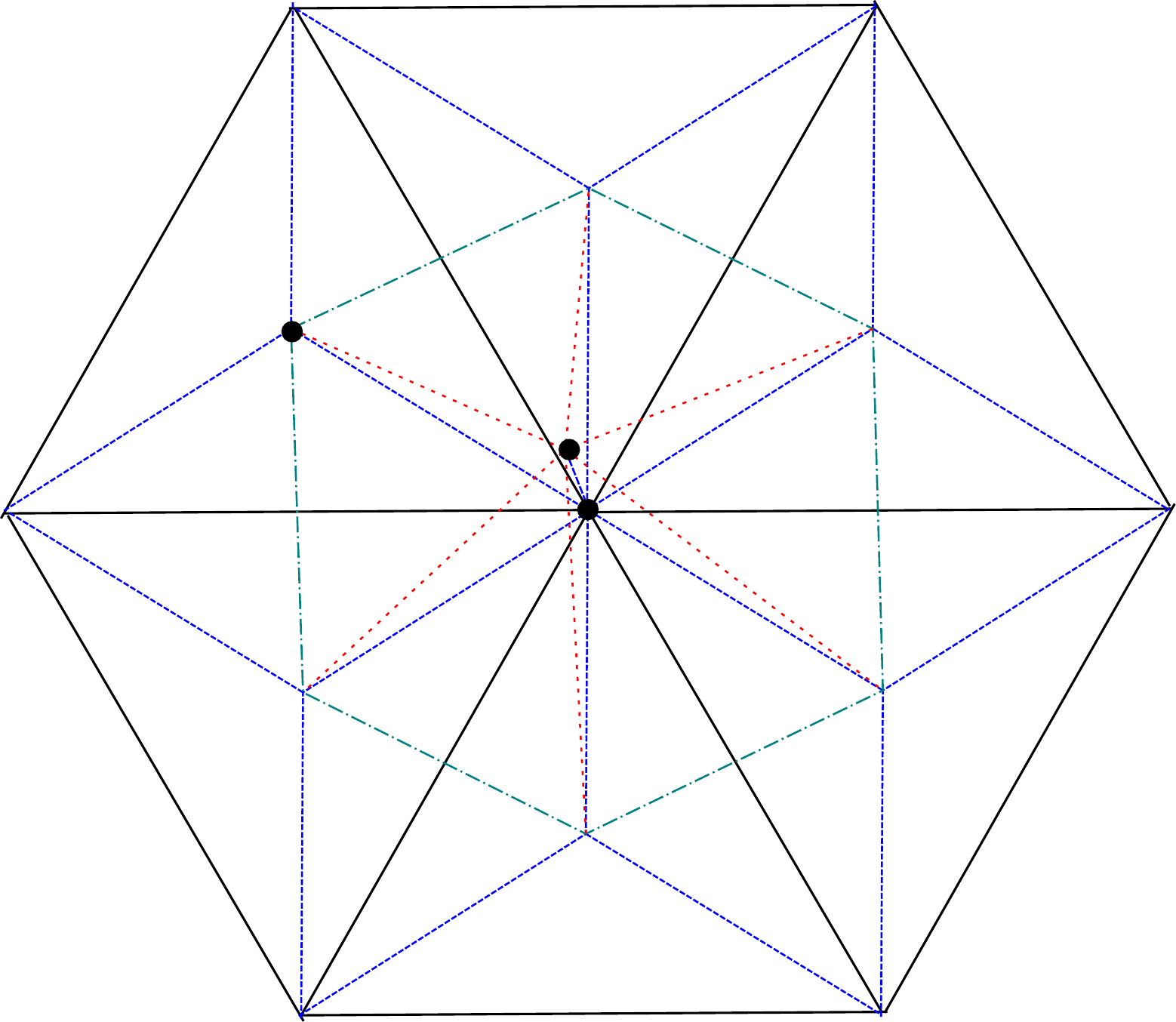
\caption{A picture of the triangulation near $\partial N_0$. Up to homotopy, black edges lie in $T_c$, blue edges connect $T_c$ and $T_c'$, red and green edges lie in $T_c'$.}
\label{figure1}
\end{figure}

For the above geometric triangulation of $N_0$, we use $V_N=\{n_1,n_2,\cdots,n_l\}$ to denote the set of vertices, and let $V_{N,\partial}=V_N\cap \partial N_0$. If there is an oriented edge from $n_i$ to $n_j$, we denote it by $e_{ij}$ and denote its orientation reversal by $e_{ji}$. For each triangle with vertices $n_i,n_j,n_k$, we denote the corresponding marked triangle by $\Delta_{ijk}$ (with an order on its vertices). We can naturally identify $\partial N_0$ with $\partial N_c=T_c$, and identify their triangulations.

Instead of directly working with the above geometric triangulation of $N_0$, we add more edges to get a cellulation of $N_0$. 

\begin{construction}\label{cellulation}
For any triangle $\Delta_{ijk}$ that only intersects with $\partial N_0$ along an edge $e_{ij}$, we add an oriented path $e_{ijk}$ in $\Delta_{ijk}$ from $n_i$ to $n_j$ of constant geodesic curvature, such that the tangent vector of $e_{ijk}$ at $n_i$ is $\frac{\epsilon}{200}$-close to the average of tangent vectors of $e_{ij}$ and $e_{ik}$, and the same for the tangent vector of $e_{ijk}$ at $n_j$.  See Figure \ref{figure2} for a picture of $e_{ijk}$. After this construction, there are two edges from $n_i$ to $n_j$.

The new edge $e_{ijk}$ divides $\Delta_{ijk}$ to a bigon and a triangle. We denote the bigon by $B_{ijk}$, and abuse notation to denote the new triangle by $\Delta_{ijk}$. For a triangle $\Delta_{ijk}$ obtained by this modification process, it is called  a {\it modified triangle}. The original triangle $\Delta_{ijk}$ (defined in Construction \ref{triangulation}) will not be used anymore.

After adding these edges to the triangulation of $N_0$ in Construction \ref{triangulation}, we get a cellulation of $N_0$, which is called a geometric cellulation. 
\end{construction}

\begin{figure}
\centering
\label{new picture}
\def\svgwidth{.5\textwidth}
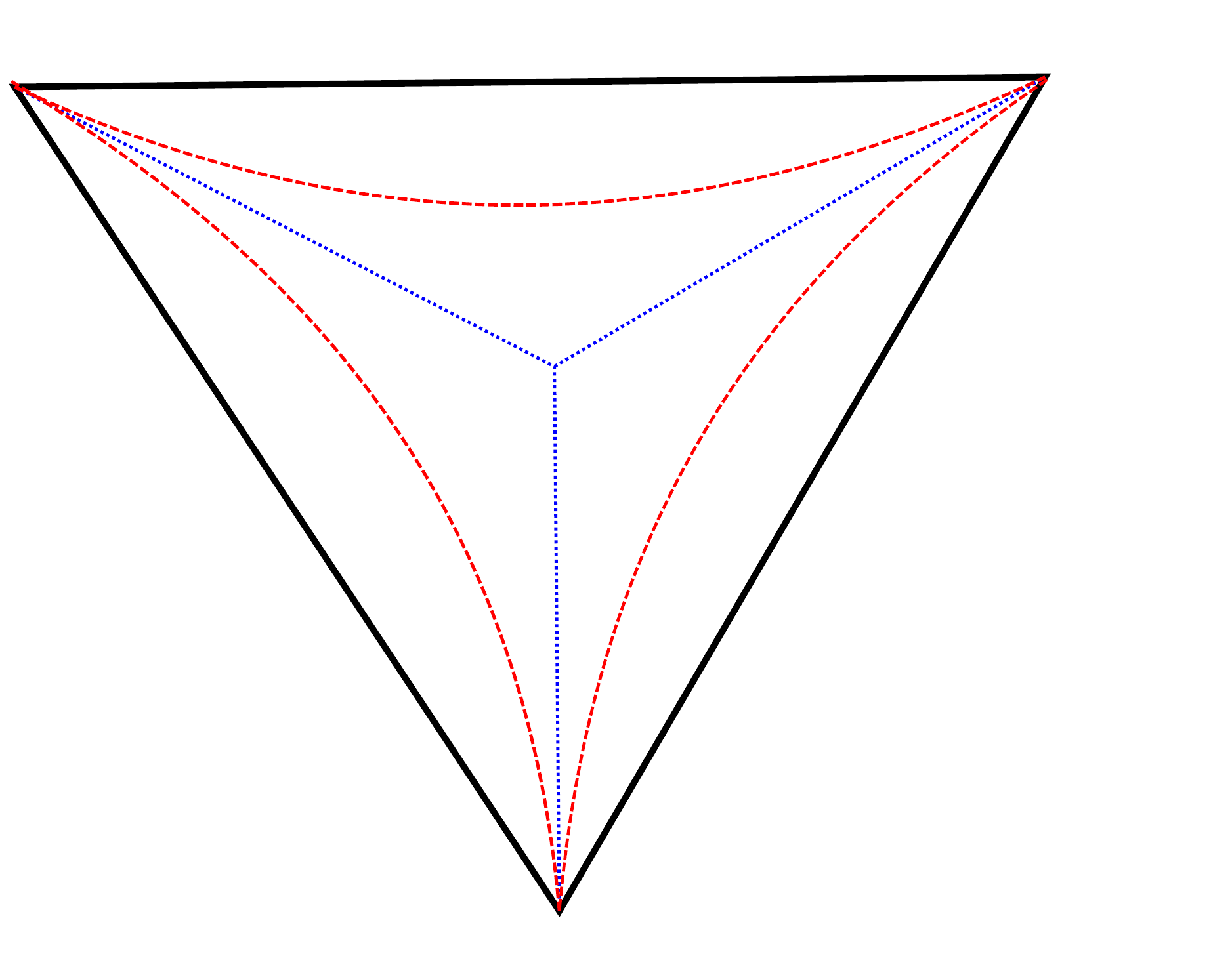
\caption{Construction of the new edge $e_{ijk}$ near $\partial N_0$, and the vertex $n_k$ is to the left of $e_{ij}$.}
\label{figure2}
\end{figure}

We use $N^{(1)}$ and $N^{(2)}$ to denote the $1$- and $2$-skeletons of the above geometric cellulation of $N$ respectively.
This cellulation also gives a handle structure of a neighborhood of $N_0$ in $N$, and we use $\mathcal{N}^{(1)}$ and $\mathcal{N}^{(2)}$ to denote the union of $0$-, $1$-handles, and $0$-, $1$-, $2$-handles.

Let $m$ be a simple closed curve on $T_c$ that intersects with $l$ exactly once. We isotopy $m$ to a curve $\gamma$ in $N_0\setminus N_{collar}$, so that it is disjoint from $N^{(1)}$, and intersects with all triangles in $N^{(2)}$ transversely. Let $\mathcal{N}(\gamma)$ be the union of all tetrahedra that intersect with $\gamma$, then we can assume that $\mathcal{N}(\gamma)$ is a neighborhood of $\gamma$ homeomorphic to the solid torus. Let $N_{\gamma}$ be $N_0\setminus \text{int}(\mathcal{N}(\gamma))$. The above cellulation of $N_0$ induces a cellulation of $N_{\gamma}$, and we denote the $2$-skeleton of $N_{\gamma}$ by $N_{\gamma}^{(2)}$.

Since the $2$-skeleton carries the first homology group, we have the following commutative diagram:
\begin{align}\label{5.1}
\begin{diagram}
H_1(N_{\gamma}^{(2)};\mathbb{Z}) & \rTo & H_1(N_{\gamma};\mathbb{Z}) & &\\
\dTo                                                 &          &  \dTo                                      & &\\
H_1(N^{(2)};\mathbb{Z})                 & \rTo   & H_1(N_0;\mathbb{Z})             & \rTo & H_1(N;\mathbb{Z}).  
\end{diagram}
\end{align}
Here all homomorphisms are induced by inclusions, and all horizontal homomorphisms are isomorphisms.

The following lemma provides some elementary properties of vertical homomorphisms in diagram (\ref{5.1}).
\begin{lemma}\label{simplehomology}
Let $i:N_{\gamma}^{(2)}\to N$ be the inclusion map, and let $c$ be the meridian of $\mathcal{N}(\gamma)$, then the following hold.
\begin{enumerate}
\item $i_*:H_1(N_{\gamma}^{(2)};\mathbb{Z})\to H_1(N;\mathbb{Z})$ is surjective.
\item The kernel of $i_*$ is spanned by a non-torsion element $[c]\in H_1(N_{\gamma}^{(2)};\mathbb{Z})$, and $[c]-[l]$ is a torsion element in $H_1(N_{\gamma}^{(2)}; \mathbb{Z})$.
\item The inclusion $\partial N_0\to N_{\gamma}^{(2)}$ induces an injective homomorphism on $H_1(\cdot;\mathbb{Z})$.
\end{enumerate}
\end{lemma}

\begin{proof}
Since horizontal homomorphisms in diagram  (\ref{5.1}) are isomorphisms, it suffices to study the inclusion $N_{\gamma}\to N_0$. We consider the M-V sequence given by $N_0=N_{\gamma}\cup \mathcal{N}(\gamma)$:
$$H_1(\partial \mathcal{N}(\gamma);\mathbb{Z})\to H_1(N_{\gamma};\mathbb{Z})\oplus H_1(\mathcal{N}(\gamma);\mathbb{Z})\to H_1(N_0;\mathbb{Z})\to 0.$$

Item (1) follows from the surjectivity of $H_1(N_{\gamma};\mathbb{Z})\oplus H_1(\mathcal{N}(\gamma);\mathbb{Z})\to H_1(N_0;\mathbb{Z})$ and the surjectivity of $H_1(\partial \mathcal{N}(\gamma);\mathbb{Z})\to H_1(\mathcal{N}(\gamma);\mathbb{Z})$. 

Now we prove item (2). By the M-V sequence, the kernel of $ H_1(N_{\gamma};\mathbb{Z})\to H_1(N_0;\mathbb{Z})$ is spanned by the meridian $c$ of $\mathcal{N}(\gamma)$. Since $[l]$ spans $\text{ker}(H_1(\partial N_0;\mathbb{R})\to H_1(N_0;\mathbb{R}))$, we take the minimal $d\in \mathbb{Z}_{>0}$ such that $d[l]=0\in H_1(N_0;\mathbb{Z})$. So $dl$ bounds a compact oriented surface $S$ in $N_0$, and the algebraic intersection number between $m\subset \partial N_0$ and $S$ is $d$. Since $\gamma$ is isotopic to $m$ in $N_0$, the algebraic intersection number between $\gamma$ and $S$ is $d$. So $S\cap N_{\gamma}$ is a compact oriented surface in $N_{\gamma}$ and it implies that $d[l]-d[c]=0\in H_1(N_{\gamma}; \mathbb{Z})$, so $[c]-[l]$ is a torsion element in $H_1(N_{\gamma}^{(2)}; \mathbb{Z})$. Moreover, the above argument also shows that $kl$ does not bound a compact oriented surface in $N_{\gamma}$ for any $k\in \mathbb{Z}_{>0}$. So $[l]$ is not a torsion element in $H_1(N_{\gamma};\mathbb{Z})$, and $[c]$ is not a torsion element either.

For the composition $H_1(\partial N_0;\mathbb{Z})\to H_1(N_{\gamma};\mathbb{Z})\to H_1(N_0;\mathbb{Z})$, its kernel is spanned by a multiple of $[l]$. Since item (2) implies that $[l]$ is not an torsion element in $H_1(N_{\gamma};\mathbb{Z})$, the homomorphism  $H_1(\partial N_0;\mathbb{Z})\to H_1(N_{\gamma};\mathbb{Z})$ is injective.

\end{proof}

Recall that $M$ has two boundary components $T_1,T_2$, such that the kernel of $H_1(T_1\cup T_2;\mathbb{Z})\to H_1(M;\mathbb{Z})$ contains $\alpha_1+\alpha_2 \in H_1(T_1\cup T_2;\mathbb{Z})$ with $0\ne \alpha_1\in H_1(T_1;\mathbb{Z})$ and $0\ne \alpha_2 \in H_1(T_2;\mathbb{Z})$. Now we treat $M$ as a non-compact open hyperbolic $3$-manifold, and consider $T_1$ and $T_2$ as two horotori of $M$.
The following lemma gives some data that will instruct us to construct the mapped-in $2$-complex $j:Z\looparrowright M$.

\begin{lemma}\label{preparation}
For any $\epsilon\in (0,10^{-2})$, there exist $R_0>0$, such that for any $R>R_0$, there exist the following maps and homomorphisms:
\begin{itemize}
\item Up to rechoosing the horotori $T_1$ and $T_2$ in $M$ (by changing their heights), we have maps $i_{\partial,1}:\partial N_0\to T_1$ and $i_{\partial,2}:\partial N_0 \to T_2$,
\item ${\bf i}_1:H_1(N_{\gamma}^{(2)};\mathbb{Z})\to H_1(M;\mathbb{Z})$ and ${\bf i}_2:H_1(N_{\gamma}^{(2)};\mathbb{Z})\to H_1(M;\mathbb{Z})$,
\item ${\bf i}:H_1(N;\mathbb{Z})\to H_1(M;\mathbb{Z})$,
\item $i_1,i_2:N_{\gamma}^{(1)}\to M$ (note that $N^{(1)}=N_{\gamma}^{(1)}$ holds),
\end{itemize}
such that the following properties hold.
\begin{enumerate}

\item For any $s=1,2$, $i_{\partial,s}$ maps each triangle in $\partial N_0$ to a geometric triangle in $T_s$ such that each inner angle is $\epsilon$-close to $\frac{\pi}{3}$, and the length of each edge lies in $[R,(1+\epsilon)R]$.

\item For any $s=1,2$, the following diagram commutes
\begin{diagram}
H_1(\partial N_0;\mathbb{Z}) & \rTo & H_1(N_{\gamma}^{(2)};\mathbb{Z})\\
\dTo^{(i_{\partial,s})_*}                                  &         &                \dTo^{{\bf i}_s}\\
H_1(T_s;\mathbb{Z})             & \rTo  & H_1(M;\mathbb{Z}).
\end{diagram} 

\item The following diagram commutes
\begin{diagram}
H_1(N_{\gamma}^{(2)};\mathbb{Z}) & \rTo^{{\bf i}_1+{\bf i}_2} & H_1(M;\mathbb{Z})\\
\dTo                                                 & \ruTo              _{{\bf i}}                &       \\
H_1(N;\mathbb{Z}).
\end{diagram}

\item For any $s=1,2$, $i_s|_{N^{(0)}}$ is an embedding, $i_s|_{\partial N_0^{(1)}}=i_{\partial,s}|_{\partial N_0^{(1)}}$ and the following diagram commute
\begin{diagram}
H_1(N_{\gamma}^{(1)};\mathbb{Z}) & \rTo^{(i_s)_*} & H_1(M;\mathbb{Z})\\
\dTo                                             & \ruTo_{{\bf i}_s}              &           \\
H_1(N_{\gamma}^{(2)};\mathbb{Z}).   &   &
\end{diagram}
\end{enumerate}
Here all undefined homormophisms are induced by inclusions.
\end{lemma}

\begin{proof}
At first, by Lemma \ref{simplehomology} (3), the inclusion induced homomorphism $H_1(\partial N_0;\mathbb{Z})\to H_1(N_{\gamma};\mathbb{Z})$ is injective. So $H_1(N_{\gamma};\mathbb{Z})$ has a direct summand $H\cong \mathbb{Z}^2$ that contains $H_1(\partial N_0;\mathbb{Z})$ as a finite index subgroup, with $H_1(N_{\gamma};\mathbb{Z})\cong H\oplus H'$ and let $[H:H_1(\partial N_0;\mathbb{Z})]=k$.

Recall that by Lemma \ref{existencealmostequilateral}, $\partial N_0$ is equipped with a triangulation induced by a geometric triangulation of the Euclidean torus $\partial N_c$, such that each triangle of $\partial N_c$ is almost an equilateral triangle with length at most $2\epsilon$. We identify $\partial N_0$ and $\partial N_c$ with the quotient of $\mathbb{R}^2$ by a lattice 
$$\Lambda=A\mathbb{Z}+(B+C\omega_0) \mathbb{Z},$$ 
such that each triangle in $\partial N_0$ and $\partial N_c$ corresponds to an equilateral triangle in $\mathbb{R}^2$ of edge length $1$ (this identification is not an isometry). Here $A,B,C\in \mathbb{Z}$ with $A,C\ne 0$ and $\omega_0=\frac{1+\sqrt{3}i}{2}$.
Moreover, we can assume that the $\mathbb{R}$-coefficient null-homologous curve $l\subset \partial N_0$ corresponds to $A\in \Lambda$.

Similarly, by Lemma \ref{existencealmostequilateral}, for any $s=1,2$, $T_s$ has a geodesic triangulation such that the following hold.
\begin{itemize}
\item Any inner angle of a triangle is $\frac{\epsilon}{10}$-close to $\frac{\pi}{3}$.
\item There exists $r_s\in (0,\frac{\epsilon}{10})$, such that all edges of triangles have length in $[r_s,(1+\frac{\epsilon}{5}) r_s)$.
\item The homology class $\alpha_s\in H_1(T_s;\mathbb{Z})$ is represented by the $d_s$-th power of a simple closed geodesic $l_s$ for some $d_s\in \mathbb{Z}_{>0}$, and $l_s$ is contained in the $1$-skeleton of this triangulation.
\end{itemize}
By the same process as above, we identify $T_s$ with the quotient of $\mathbb{R}^2$ by a lattice 
$$\Lambda_s=A_s\mathbb{Z}+(B_s+C_s\omega_0)\mathbb{Z},$$ 
with $A_s,C_s\ne 0$ and $l_s$ corresponds to $A_s\in \Lambda_s$. By sub-dividing triangles, we can assume that $A_1d_1=A_2d_2$ holds. Moreover, by rechoosing horotori parallel to $T_1$ and $T_2$ respectively, we can assume that $r_1=r_2$ holds, and let $r=r_1=r_2$.

Let $D$ be the least common multiple of $A_1C_1d_1$ and $A_2C_2d_2$. For any $a\in \mathbb{Z}_{>0}$, we construct a map $i_{\partial,s}:\partial N_0\to T_s$ as following. At first, we have
$$\Lambda_s=A_s \mathbb{Z}+(B_s+C_s\omega)\mathbb{Z}>akD\cdot \Lambda=(akDA)\mathbb{Z}+(akD(B+C\omega_0))\mathbb{Z},$$ since 
\begin{align}\label{5.2}
akDA=(akA\frac{D}{A_s})A_s
\end{align} and 
\begin{align}\label{5.3}
akD(B+C\omega_0)=(akB\frac{D}{A_s}-akCB_s\frac{D}{A_sC_s})A_s+(akC\frac{D}{C_s})(B_s+C_s\omega_0).
\end{align}
So the scaling by $akD$ gives a map from $\mathbb{R}^2/\Lambda$  to $\mathbb{R}^2/\Lambda_s$, and it maps each equilateral triangle of length $1$ to an equilateral triangle of length $akD$. 
Since we identified $\partial N_0$ and $T_s$ with $\mathbb{R}^2/\Lambda$ and $\mathbb{R}^2/\Lambda_s$ respectively, the $akD$-scaling map induces $i_{\partial,s}:\partial N_0\to T_s$, such that it maps each triangle in $\partial N_0$ to a triangle in $T_s$ of inner angle $\epsilon$-close to $\frac{\pi}{3}$, and with edge length contained in $[akDr,(1+\frac{\epsilon}{5})akDr)$. If $R>R_0=\frac{2kDr}{\epsilon}$, there exists a positive integer $a$, such that $[akDr,(1+\frac{\epsilon}{5})akDr)\subset [R,(1+\epsilon)R]$. So we can choose $a$ such that item (1) holds for both $i_{\partial,1}$ and $i_{\partial,2}$.

Note that the simple closed curve $l$ in $\partial N_0$ corresponds to $A\in \Lambda$, which is mapped to $akDA\in \Lambda_s$ via the $akD$-scaling map. Since $\alpha_s=d_sl_s$ corresponds to $d_sA_s\in \Lambda_s$, $(i_{\partial,s})_*$ maps $l$ to $(akA\frac{D}{d_sA_s})\alpha_s\in H_1(T_s;\mathbb{Z})$. Since we assumed $A_1d_1=A_2d_2$ and $\alpha_1+\alpha_2=0\in H_1(M;\mathbb{Z})$, we have 
\begin{align}\label{5.4}
(i_{\partial,1})_*(l)+(i_{\partial,2})_*(l)=akA\frac{D}{d_1A_1}(\alpha_1+\alpha_2)=0\in H_1(M;\mathbb{Z}).
\end{align}

Now we define ${\bf i}_s:H_1(N_{\gamma}^{(2)};\mathbb{Z})\to H_1(M;\mathbb{Z})$ for $s=1,2$. At first, by equations (\ref{5.2}) and (\ref{5.3}), $(i_{\partial,s})_*:H_1(\partial N_0;\mathbb{Z})\to H_1(T_s;\mathbb{Z})$ maps each element to a $k$-multiple of an element in $H_1(T_s;\mathbb{Z})$. Since $H_1(N_{\gamma}^{(2)};\mathbb{Z})\cong H_1(N_{\gamma};\mathbb{Z})=H\oplus H'$ for some $H$ containing $H_1(\partial N_0;\mathbb{Z})$ with $[H:H_1(\partial N_0;\mathbb{Z})]=k$, the homomorphism $(i_{\partial,s})_*:H_1(\partial N_0;\mathbb{Z})\to H_1(T_s;\mathbb{Z})$ uniquely extends to a homomorphism ${\bf h}_s:H\to H_1(T_s;\mathbb{Z})$, and we define
${\bf i}_s:H_1(N_{\gamma}^{(2)};\mathbb{Z})\to H_1(M;\mathbb{Z})$ to be 
$$H_1(N_{\gamma}^{(2)};\mathbb{Z})\to H\xrightarrow{{\bf h}_s}H_1(T_s;\mathbb{Z})\to H_1(M;\mathbb{Z}).$$
Here the first homomorphism is the projection to the direct summand $H$, and the third homomorphism is induced by inclusion. It is straight forward to check that the commutative diagram in item (2) holds.

Note that $N$ deformation retracts to $N_0=N_{\gamma}\cup \mathcal{N}(\gamma)$ and $\mathcal{N}(\gamma)$ is a solid torus. Once we prove that the meridian $c$ of $\mathcal{N}(\gamma)$ lies in the kernel of ${\bf i}_1+{\bf i}_2:H_1(N_{\gamma}^{(2)};\mathbb{Z})\to H_1(M;\mathbb{Z})$, then ${\bf i}_1+{\bf i}_2$ induces a homomorphism ${\bf i}:H_1(N;\mathbb{Z})\to H_1(M;\mathbb{Z})$ and the commutative diagram in item (3) holds. Recall that Lemma \ref{simplehomology} (2) implies that $[l]-[c]$ is a torsion element in $H_1(N_{\gamma}^{(2)};\mathbb{Z})=H\oplus H'$. Since $H\cong \mathbb{Z}^2$, we have $[c]-[l]\in H'$. By the definition of ${\bf i}_1$ and ${\bf i}_2$ above, ${\bf i}_1([c]-[l])={\bf i}_2([c]-[l])=0$ holds. So we have 
$$({\bf i}_1+{\bf i}_2)([c])=({\bf i}_1+{\bf i}_2)([l])+({\bf i}_1+{\bf i}_2)([c]-[l])=({\bf i}_1+{\bf i}_2)([l])=(i_{\partial,1})_*([l])+(i_{\partial,2})_*([l])=0.$$
Here the third equation follows from item (2) and the fourth equation follows from equation (\ref{5.4}).

To define $i_s:N_{\gamma}^{(1)}\to M$, we first define $i_s|_{\partial N_0^{(1)}}=i_{\partial,s}|_{\partial N_0^{(1)}}$, and arbitrarily extend $i_s$ to a maximal subcomplex $K$ of $N_{\gamma}^{(1)}$ that deformation retracts to $\partial N_0$. Then since edges in $N_{\gamma}^{(1)}\setminus K$ form a basis of $H_1(N_{\gamma}^{(1)};\mathbb{Z})/H_1(\partial N_0^{(1)};\mathbb{Z})$, we can extend $i_s$ to $N_{\gamma}^{(1)}$ so that the commutative diagram in item (4) holds. Finally, we slightly perturb $i_s$ if necessary, so that $i_s|_{N^{(0)}}$ is an embedding.
\end{proof}

Now we give some notations on the geometry of $N$. Most of the items (except item (3)) are similar to the notations in Notation 4.4 of \cite{Sun5}.
\begin{notation}\label{notation}
\begin{enumerate}
%\item Let the set of vertices of the triangulation of $N$ be $V_N=\{n_1,n_2,\cdots,n_l\}$, and let $V_{N,\partial}=V_N\cap \partial N_0$. If there is an oriented edge in $N^{(1)}$ from $n_i$ to $n_j$, we denote it by $e_{ij}$ and denote its orientation reversal by $e_{ji}$. If there is a triangle in $N^{(2)}$ spanned by (distinct) vertices $(n_i,n_j,n_k)$, we have a marked $2$-complex $\Delta_{ijk}$.

\item For any oriented edge $e_{ij}$ (or $e_{ijk}$) in $N^{(1)}$, let $\vec{v}_{ij}$ ($\vec{v}_{ijk}$) be the unit tangent vector of $e_{ij}$ ($e_{ijk}$) based at $n_i$. By Construction \ref{cellulation}, $\vec{v}_{ijk}$ lies in the plane in $T_{v_i}M$ containing $\vec{v}_{ij}$ and $\vec{v}_{ik}$, and is $\frac{\epsilon}{200}$-close to
$\frac{\vec{v}_{ij}+\vec{v}_{ik}}{|\vec{v}_{ij}+\vec{v}_{ik}|}$. 
For any marked geodesic triangle $\Delta_{ijk}$ in $N^{(2)}$, let 
$$\vec{n}_{ijk}=\frac{\vec{v}_{ij}\times \vec{v}_{ik}}{|\vec{v}_{ij}\times \vec{v}_{ik}|},$$ 
then it is a normal vector of $\Delta_{ijk}$ at $n_i$, and we have a frame 
$${\bf F}_{ijk}=(n_i,\vec{v}_{ij},\vec{n}_{ijk})\in \text{SO}(N)_{n_i}.$$ 
For any marked modified triangle $\Delta_{ijk}$ defined in Construction \ref{cellulation}, with $e_{ij}$ contained in $\partial N_0$ (which is not an edge of $\Delta_{ijk}$), the frames ${\bf F}_{kij}$ and ${\bf F}_{kji}$ are defined as in the previous case. For ${\bf F}_{ijk}$, let 
$$\vec{n}_{ijk}=\frac{\vec{v}_{ijk}\times \vec{v}_{ik}}{|\vec{v}_{ijk}\times \vec{v}_{ik}|}=\frac{\vec{v}_{ij}\times \vec{v}_{ik}}{|\vec{v}_{ij}\times \vec{v}_{ik}|},$$ 
and we get a frame $${\bf F}_{ijk}=(n_i,\vec{v}_{ijk},\vec{n}_{ijk})\in \text{SO}(N)_{n_i}.$$ 
For each frame ${\bf F}_{ijk}$, we denote $-{\bf F}_{ijk}=(n_i,-\vec{v}_{ij},-\vec{n}_{ijk}),$ then we have a finite collection of frames in $N$: 
$$\mathcal{F}_N=\{\pm {\bf F}_{ijk}\ |\ \Delta_{ijk} \text{\ is\ a\ marked\ triangle\ in\ }N^{(2)}\}.$$

\item For any $s=1,2$, let $m_{k,s}=i_s(n_k)$ and let $V_{M,s}=i_s(V_N)$. We take an isometry $t_s:TN|_{V_N}\to TM|_{V_{M,s}}$ that descends to $i_s:V_N\to V_{M,s}$, such that the following holds for any $n_k\in V_{N,\partial}$. At $n_k\in V_{N,\partial}$, there is a frame $(\vec{v}_k,\vec{n}_k)$ such that $\vec{v}_k$ is tangent to the direction of $l\subset \partial N_c$, and $\vec{v}_k\times \vec{n}_k$ points up straightly into the cusp. We require that $t_s$ maps $\vec{v}_k$ and $\vec{n}_k$ to $\vec{v}_k'$ and $\vec{n}_k'$ based at $m_{k,s}$ respectively, such that $\vec{v}_k'$ is tangent to the direction of $l_s\subset T_s$ and $\vec{v}_k'\times \vec{n}_k'$ points up straightly into the cusp. Then $t_s$ induces an $\text{SO}(3)$-equivariant isomorphism $t_s:\text{SO}(N)|_{V_N}\to \text{SO}(M)|_{V_{M,s}}$, denoted by the same notation. We denote ${\bf F}_{ijk,s}^{M}=(m_{i,s},\vec{v}_{ij,s}^{M}, \vec{n}_{ijk,s}^{M})=t_s({\bf F}_{ijk})\in \text{SO}(M)|_{m_{i,s}}$, and let 
$$\mathcal{F}_{M,s}=t_s(\mathcal{F}_N)\subset \text{SO}(M).$$

\item Since $N^{(2)}$ contains finitely many $2$-cells, there exists $\phi_0\in (0,\pi)$, such that all inner angles of bigons and triangles in $N^{(2)}$ and all dihedral angles between adjacent $2$-cells of $N^{(2)}$ lie in $[\phi_0,\pi]$. 
%For any vertex $n_i$ of $N$, the geometric cellulation of $N$ gives a subset $S_i\subset T^1_{n_i}=S^2$, consisting of all unit tangent vectors of bigons and triangles in $N^{(2)}$ adjacent to $n_i$. Then $S_i$ is a union of finitely many geodesic arcs in $S^2$ and has a path metric. Then there exists $\theta_0>0$, such that for any vertex $n_i$ and any two vectors $v_1,v_2\in S_i\subset S^2$, if their distance under the path metric of $S_i$ is at least $\phi_0$, then the angle between them in $S^2$ is at least $\theta_0$.
\end{enumerate}
\end{notation}

\begin{remark}\label{coordinate}
Let $n_i, n_j$ be two vertices of $\partial N_0$ such that $e_{ij}$ is contained in $\partial N_0$, and let $n_k$ be the vertex not contained in $\partial N_0$ such that $n_i,n_j,n_k$ span an triangle in the original triangulation of $N_0$ and it lies to the left of $e_{ij}$, as in Figure \ref{figure2}. We give coordinates of $T^1_{n_i}N$ such that $\vec{v}_{ij}\approx (1,0,0)$ with vanishing second coordinate and the vector pointing to the cusp is $(0,0,1)$. 

Let $\vec{v}_0=(\frac{1}{2},\frac{1}{2\sqrt{3}},\frac{r}{6}-\frac{1-e^{-\frac{2\sqrt{6}}{3}r}}{2r})$. 
Assuming all triangles in $T_c\subset N$ are equilateral triangles of length $r$, an elementary computation gives the following:
$$\vec{v}_{ik}=\vec{v}_1=\frac{\vec{v}_0}{\|\vec{v}_0\|}\approx (\frac{1}{2},\frac{1}{2\sqrt{3}},-\frac{\sqrt{6}}{3}), \vec{v}_{ijk}=\vec{v}_2\approx \frac{\vec{v}_1+(1,0,0)}{\|\vec{v}_1+(1,0,0)\|}\approx (\frac{\sqrt{3}}{2},\frac{1}{6},-\frac{\sqrt{2}}{3}),$$
$$\vec{n}_{ijk}=\vec{v}_3=\frac{\vec{v}_2\times \vec{v}_1}{\|\vec{v}_2\times \vec{v}_1\|}\approx (0,\frac{2\sqrt{2}}{3},\frac{1}{3}).$$

In this remark, the actual vectors are all $\frac{\epsilon}{20}$-close to their numerical approximations above.
\end{remark}

\begin{remark}
Although the frame bundle of a compact orientable $3$-manifold $N$ with connected torus boundary is trivial, there may not be a trivialization of $\text{SO}(N)$ such that its restriction to $\partial N$ has third vector pointing outward. So we do not have a homological instruction as good as Proposition 4.5 of \cite{Sun5}, which reduces the degree of virtual domination by a half.
\end{remark}

\bigskip

\subsection{Construction of the immersion $j:Z\looparrowright M$}\label{constructZ2section}

In this section, we construct the $\pi_1$-injective immersion $j:Z\looparrowright M$. Since $Z$ is a $2$-complex, we will inductively construct the $0$-, $1$-, $2$-skeletons of $Z$ and the restrictions of $j$ on these skeletons. 
Throughout this section, we fix a small number $\epsilon\in (0,10^{-2})$ and a sufficiently large number $R\in( \frac{1}{\epsilon},+\infty)$ such that all (finitely many) constructions below (invoking Theorems \ref{boundingsurface} and \ref{connectionprinciple}) are applicable.

\begin{construction}\label{constructZ0}
We define $Z^{0}$ to be a finite set $\{v_{1,1}, v_{1,2},v_{2,1},v_{2,2},\cdots, v_{l,1},v_{l,2}\}$, whose cardinality doubles the cardinality of $V_N=N^{(0)}$. Then we define $j^{0}:Z^{0}\to M$ by $j^{0}(v_{k,s})=m_{k,s}=i_s(n_k)\in M$ for any $k\in \{1,2,\cdots,l\}$ and $s\in \{1,2\}$.
\end{construction}

Here we take two copies of $N^{(0)}$ since we work with two boundary components $T_1$ and $T_2$ of $M$. Now we construct the $1$-complex $Z^1$ of $Z$. 

\begin{construction}\label{constructZ1}
For any unoriented edge $e_{ij}$ (or $e_{ijk}$) in $N^{(1)}$, it gives two edges $e_{ij,1}^Z,e_{ij,2}^Z$ (or $e_{ijk,1}^Z,e_{ijk,2}^Z$) in $Z^1$, such that $e_{ij,s}^Z$ (or $e_{ijk,s}^Z$) connects $v_{i,s}$ and $v_{j,s}$ for $s=1,2$. So $Z^{1}$ consists of two isomorphic components $Z^{1}_1$ and $Z^{1}_2$, and each of them is isomorphic to $N^{(1)}$. For the vertices and edges of $Z^1$ corresponding to vertices and edges in $N^{(1)}\cap \partial N_0$, they form a subcomplex of $Z^1$ and we denote it by $\partial_p Z^1$.
\end{construction}
A picture of $Z^{1}$ near a vertex of $\partial_p Z^1$ is shown in Figure \ref{figure1}.

The indices of vertices induce a total order on the set of vertices in $N^{(0)}$, and also induce total orders on the vertex set of $Z^{1}_1$ and the vertex set of $Z^{1}_2$. Any edge $e_{ij}$ (or $e_{ijk}$) in $N^{(1)}$ between $n_i$ and $n_j$ with $i<j$ has an orientation that goes from $n_i$ to $n_j$, and we always fix such a preferred orientation. Edges of $Z_1^1$ and $Z_2^1$ have identical orientations.

 The map $j^{1}:Z^{1}\to M$ on $1$-skeleton is given in the following construction, which consists of two maps $j^{1}_1:Z^{1}_1\to M$ and $j^{1}_2:Z^{1}_2\to M$ on the two (identical) components of $Z^{1}$.

\begin{construction}\label{constructj1}
For any $s=1,2$, we do the following construction.
\begin{enumerate}
\item For each oriented edge $e_{ij,s}^{Z}\subset Z^{1}_s$ with $i<j$ contained in $\partial_p Z_s^1$, we map it to the oriented geodesic segment homotopic to $i_s(e_{ij})$ relative to endpoints, via a homeomorphism. Note that these geodesic segments have length contained in $[2\log{R},2\log{R}+4\epsilon]$ (by Lemma \ref{preparation} (1)).

\item For each oriented edge $e_{ij,s}^{Z} \subset Z^{1}_s$ with $i<j$ not contained in $\partial_p Z_s^1$, we apply Theorem \ref{connectionprinciple} to construct a $\partial$-framed segment $\mathfrak{s}_{ij,s}$ in $M$ from $m_{i,s}$ to $m_{j,s}$ such that the following conditions hold, and we map $e_{ij,s}^{Z}$ to the carrier of $\mathfrak{s}_{ij,s}$ via a homeomorphism.

\begin{enumerate}
\item The length and phase of $\mathfrak{s}_{ij,s}$ are $\frac{\epsilon}{10}$-close to $2R$ and $0$ respectively, and the height of $\mathfrak{s}_{ij,s}$ is at most $2\log{R}+2$.
\item Let $k$ be the smallest index so that $n_i,n_j,n_k$ form a triangle in $N^{(2)}$, the initial and terminal frames of $\mathfrak{s}_{ij,s}$ are $\frac{\epsilon}{10}$-close to ${\bf F}_{ijk,s}^{M}$ and $-{\bf F}_{jik,s}^{M}$ respectively.
\item The relative homology class of the carrier of $\mathfrak{s}_{ij,s}$ in $H_1(M,\{m_{i,s},m_{j,s}\};\mathbb{Z})$ equals the relative homology class of $i_s(e_{ij})$.
\end{enumerate}

\item %For any $e_{ij}\subset N^{(1)}\cap (\partial N_0)$ with $i<j$, it is adjacent to two triangles in $N_{\gamma}^{(2)}\setminus (\partial N_0)$, with the third vertices denoted by $n_{k_1}$ and $n_{k_2}$ respectively, so that $n_{k_1}$ lies to the left of $e_{ij}$ and $n_{k_2}$ lies to the right of $e_{ij}$, as shown in Figure 2 or 3.
For any oriented edge $e_{ijk,s}^{Z} \subset Z^{1}_s$ with $i<j$ that corresponds to $e_{ijk}\subset N^{(1)}$, we apply Theorem \ref{connectionprinciple} to construct a $\partial$-framed segment $\mathfrak{s}_{ijk,s}$ in $M$ from $m_{i,s}$ to $m_{j,s}$ such that the following conditions hold, and we map $e_{ijk,s}^{Z}$ to the carrier of $\mathfrak{s}_{ijk,s}$ via a homeomorphism.
\begin{enumerate}
\item The length and phase of $\mathfrak{s}_{ijk,s}$ are $\frac{\epsilon}{10}$-close to $2R$ and $0$ respectively, and the height of $\mathfrak{s}_{ijk,s}$ is at most $2\log{R}+2$.
\item The initial and terminal frames of $\mathfrak{s}_{ijk,s}$ are $\frac{\epsilon}{10}$-close to 
${\bf F}_{ijk,s}^{M}$ and $-{\bf F}_{jik,s}^{M}$ respectively.
\item The relative homology class of the carrier of $\mathfrak{s}_{ijk,s}$ in $H_1(M,\{m_{i,s},m_{j,s}\};\mathbb{Z})$ equals the relative homology class of $i_s(e_{ijk})$.
\end{enumerate}

\end{enumerate}
\end{construction}

Figure \ref{figure1} shows the geometry of $j^{1}(Z^{1}_s)$ near a vertex $v_{i,s}$ corresponding to $n_i\in \partial N_0$. 

\begin{remark}\label{constructj1remark}
\begin{enumerate}
\item In Construction \ref{constructj1} (1), the tangent vector of $j^{1}(e_{ij,s}^{Z})$ at $m_{i,s}$ is almost $(0,0,1)$ (with respect to the preferred coordinate system), while the tangent vector of ${\bf F}_{ijk,s}^{M}$ is almost $(1,0,0)$. This is the crucial reason why we need extra edges ($e_{ijk}$ and $e_{ijk,s}^Z$) in $N^{(1)}$ and $Z^{1}$, which takes care of this difference.
We map $e_{ij,s}^{Z}\subset \partial_p Z_s^1$ to the geodesic segment in the relative homotopy class of $i_s(e_{ij})$, instead of prescribing its tangent vectors at initial and terminal points, since we need to construct proper maps between $3$-manifolds with tori boundary. 

\item In Construction \ref{constructj1} (2), if we take another vertex $n_{k'}$ of $N$ such that $n_i,n_j,n_{k'}$ also form a triangle in $N^{(2)}$, then we can rechoose frames of $\mathfrak{s}_{ij,s}$ so that it still satisfies item (2), with respect to ${\bf F}_{ijk'}^{M,s}$ and $-{\bf F}_{jik'}^{M,s}$ in item (2) (b). The reason is that ${\bf F}_{ijk'}={\bf F}_{ijk}\cdot A$ and $-{\bf F}_{jik'}=(-{\bf F}_{jik})\cdot A$ for the same $A\in \text{SO}(3)$, while $t_s$ is $\text{SO}(3)$-equivariant.

\item In Construction \ref{constructj1} (3), by our construction of the triangulation of $N$ near $\partial N_0$ in Construction \ref{triangulation} (1), the third coordinate of $\vec{v}_{ij,s}^{M}$ is at most $\frac{\epsilon}{100}$.
Up to changing coordinate, we assume that $\vec{v}_{ij,s}^{M}$ has trivial second coordinate, then it is $\frac{\epsilon}{100}$-close to $(1,0,0)$. We suppose that $n_k$ lies to the left of $e_{ij}$, as shown in Figure \ref{figure2}, then $\vec{v}_{ik,s}^{M}$ is
$\frac{\epsilon}{20}$-close to $(\frac{1}{2},\frac{\sqrt{3}}{6},-\frac{\sqrt{6}}{3})$, and the initial frame of $\mathfrak{s}_{ijk,s}$ is $\frac{\epsilon}{5}$-close to $\big((\frac{\sqrt{3}}{2},\frac{1}{6},-\frac{\sqrt{2}}{3}), (0,\frac{2\sqrt{2}}{3},\frac{1}{3})\big)$. See also Remark \ref{coordinate}.

Moreover, the common perpendicular vector of $j_1(e_{ij,s}^Z)$ and $j^1(e_{ijk,s}^Z)$ at $m_{i,s}$
is $\frac{\epsilon}{5}$-close to $(-\frac{1}{2\sqrt{7}},\frac{3\sqrt{3}}{2\sqrt{7}},0)$. If we consider Figure \ref{figure2}
as a picture of $j_s^1(Z^1)$ in $M$,
the dihedral angle between 
the hyperplane determined the above vector and the geodesic triangle homotopic to $T_s$ (relative to vertices) is $\epsilon$-close to 
$$\ \ \ \arccos{\big((-\frac{1}{2\sqrt{7}},\frac{3\sqrt{3}}{2\sqrt{7}},0)\cdot (-\frac{\sqrt{3}}{2},-\frac{1}{2},0)\big)}=\arccos{(-\frac{\sqrt{3}}{2\sqrt{7}})}\approx 0.606\pi. $$

Note that this computation will be crucial for our proof of Theorem \ref{pi1injectivity} in Section \ref{pi1inj2}.

\end{enumerate}
\end{remark}

Before we construct the $2$-complex $Z$, we need the following lemma that proves certain closed curves arised from Construction \ref{constructj1} are good curves.

\begin{lemma}\label{checkgoodcurve}
Under the conditions in Construction \ref{constructj1}, if $R$ is large enough, we have the following good curves.
\begin{enumerate}
\item For any bigon $B_{ijk}$ in $N^{(2)}$, the concatenation of $\overline{j^{1}(e_{ij,s}^{Z})},j^{1}(e_{ijk,s}^{Z})$ is homotopic to a null-homologous $(R_{ij},\epsilon)$-good curve $\gamma_{ijk,s}$ of height at most 
$2\log{R}+3$ in $M$, with $R_{ij}=R+\log{l_{ij}}-\log{\frac{6}{3+\sqrt{2}}}$. Here $l_{ij}$ denotes the length of the Euclidean geodesic segment $i_{\partial,s}(e_{ij})$ and $l_{ij}\in [R, (1+\epsilon)R]$.

\item For any triangle $\Delta_{ijk}$ in $N^{(2)}$ with vertices $n_i,n_j,n_k$, the concatenation of $j^{1}(e_{ij,s}^{Z}), j^{1}(e_{jk,s}^{Z}), j^{1}(e_{ki,s}^{Z})$ is an $(R_{ijk},\epsilon)$-good curve $\gamma_{ijk,s}$ of height at most $2\log{R}+3$ in $M$, with 
$$R_{ijk}=3R-(I(\pi-\theta_{ijk})+I(\pi-\theta_{jki})+I(\pi-\theta_{kij})).$$
(Here $j^{1}(e_{ij,s}^{Z})$ is replaced by $j^{1}(e_{ijk,s}^{Z})$ if $\Delta_{ijk}$ is a modified triangle.) 
Here $\theta_{ijk}$ is the inner angle of the triangle $\Delta_{ijk}$ at vertex $n_i$. Moreover, if $\Delta_{ijk}$ is contained in $N_{\gamma}^{(2)}$, $\gamma_{ijk,s}$ is null-homologous in $M$; if $\Delta_{ijk}$ is not contained in $N_{\gamma}^{(2)}$, then $\gamma_{ijk,1}\cup \gamma_{ijk,2}$ is null-homologous in $M$.
\end{enumerate}
\end{lemma}

Note that if $R$ is large enough, all good curves in this lemma have length contained in $[2R,6R]$.

\begin{proof}
\begin{enumerate}
\item 
Note that $j^{1}(e_{ijk,s}^{Z})$ is the carrier of $\mathfrak{s}_{ijk,s}$, and we assumed that $\vec{v}_{ij,s}^{M}$ has trivial second coordinate as in Remark \ref{constructj1remark} (2). So by Remark \ref{constructj1remark} (3), the initial and terminal frames of $\mathfrak{s}_{ijk,s}$ are $\frac{\epsilon}{5}$-close to 
$$\big((\frac{\sqrt{3}}{2},\frac{1}{6},-\frac{\sqrt{2}}{3}), (0,\frac{2\sqrt{2}}{3},\frac{1}{3})\big)\text{\ and\ }\big((\frac{\sqrt{3}}{2},-\frac{1}{6},\frac{\sqrt{2}}{3}), (0,\frac{2\sqrt{2}}{3},\frac{1}{3})\big)$$ 
respectively (with respect to preferred coordinates). For $\phi=\pi-\arcsin{\frac{1}{\sqrt{7}}}$, the initial and terminal frames of the frame rotation $\mathfrak{s}_{ijk,s}(\phi)$ are 
$\frac{\epsilon}{5}$-close to 
$$\qquad \qquad \big((\frac{\sqrt{3}}{2},\frac{1}{6},-\frac{\sqrt{2}}{3}), (\frac{1}{2\sqrt{7}},-\frac{3\sqrt{3}}{2\sqrt{7}},0)\big)\text{\ and\ }\big((\frac{\sqrt{3}}{2},-\frac{1}{6},\frac{\sqrt{2}}{3}),  (-\frac{1}{2\sqrt{7}},-\frac{3\sqrt{3}}{2\sqrt{7}},0)\big)$$ respectively.	

For $j^{1}(e_{ij,s}^{Z})$, its initial and terminal directions are $\frac{2}{R}$-close to $(0,0,1)$ and $(0,0,-1)$ respectively. Since $j^{1}(e_{ij,s}^{Z})$ is a geodesic segment in a cusp, it parallel transports $(0,1,0)$ to $(0,1,0)$. We obtain a $\partial$-framed segment $\mathfrak{t}$ with phase $0$, by equipping $j^{1}(e_{ij,s}^{Z})$ with initial and terminal framings $(0,1,0)$. Then for $\phi'=\pi+\arcsin{\frac{1}{2\sqrt{7}}}$, its $\phi'$-rotation $\mathfrak{t}(\phi')$ is a $\partial$-framed segment with $0$-phase, with initial and terminal frames $\frac{4}{R}$-close to 
$$((0,0,1), (\frac{1}{2\sqrt{7}},-\frac{3\sqrt{3}}{2\sqrt{7}},0)),\text{\ and \ } ((0,0,-1), (-\frac{1}{2\sqrt{7}},-\frac{3\sqrt{3}}{2\sqrt{7}},0))$$ respectively.

If $R>\frac{20}{\epsilon}$, $\overline{\mathfrak{t}(\phi')},\mathfrak{s}_{ijk,s}(\phi)$ is a $(\frac{2}{5}\epsilon)$-consecutive cycle of $\partial$-framed segments, with both bending angles $(\frac{2}{5}\epsilon)$-close to $\arccos{\frac{\sqrt{2}}{3}}$. By elementary hyperbolic geometry, the length of $\overline{\mathfrak{t}(\phi')}$ equals $2\log{\frac{\sqrt{l_{ij}^2+4}+l_{ij}}{2}}$.

Lemma \ref{lengthphase} (2) implies that the concatenation of $\overline{j^{1}(e_{ij,s}^{Z})},j^{1}(e_{ijk,s}^{Z})$ is homotopic to a closed geodesic $\gamma_{ijk,s}$ with complex length $2\epsilon$-close to $$2R_{ij}=2R+2\log{l_{ij}}-2\log{\frac{6}{3+\sqrt{2}}}.$$ So $\gamma_{ijk,s}$ is an $(R_{ij},\epsilon)$-good curve.

We take large enough $R$, so that heights of $T_1$ and $T_2$ are at most $\log{R}$.
Since the heights of $\mathfrak{s}_{ijk,s}$ and $\mathfrak{t}$ are at most $2\log{R}+2$, Lemma \ref{distance} implies the height of $\gamma_{ijk,s}$ is at most $2\log{R}+3$. By the homological conditions in Construction \ref{constructj1} (1) and (3) (c), $j^{1}(e_{ij,s}^{Z})$ and $j^{1}(e_{ijk,s}^{Z})$ represent the same relative homology class, so $\gamma_{ijk,s}$ is a null-homologous closed geodesic in $M$.

\item We prove this result in the case that $\Delta_{ijk}$ is not a modified triangle, and the case of modified triangles can be proved similarly. 

By our constructions of $j^{1}(e_{ij,s}^{Z}), j^{1}(e_{jk,s}^{Z}), j^{1}(e_{ki,s}^{Z})$ in Construction \ref{constructj1} (2), we equip them with initial and terminal frames as following to get three $\partial$-framed segments.
\begin{itemize}
\item Equip $j^{1}(e_{ij,s}^{Z})$ with initial and terminal frames that are $\frac{\epsilon}{10}$-close to $\vec{n}_{ijk,s}^M$ and $-\vec{n}_{jik,s}^M$ respectively.
\item Equip $j^{1}(e_{jk,s}^{Z})$ with initial and terminal frames that are $\frac{\epsilon}{10}$-close to $\vec{n}_{jki,s}^M$ and $-\vec{n}_{kji,s}^M$ respectively.
\item Equip $j^{1}(e_{ki,s}^{Z})$ with initial and terminal frames that are $\frac{\epsilon}{10}$-close to $\vec{n}_{kij,s}^M$ and $-\vec{n}_{ikj,s}^M$ respectively.
\end{itemize}

By Construction \ref{constructj1} (2) and Remark \ref{constructj1remark} (2), the phases of these $\partial$-framed segments are $\frac{\epsilon}{10}$-close to $0$. Since $\vec{n}_{ijk,s}^M=-\vec{n}_{ikj,s}^M$, these three $\partial$-framed segments form a $\frac{\epsilon}{5}$-consecutive cycle.
 Then Lemma \ref{lengthphase} (2) implies that the concatenation is homotopic to a closed geodesic $\gamma_{ijk,s}$ with complex length $2\epsilon$-close to 
$$2R_{ijk}=6R-2(I(\pi-\theta_{ijk})+I(\pi-\theta_{jki})+I(\pi-\theta_{kij})).$$ So $\gamma_{ijk,s}$ is an $(R_{ijk},\epsilon)$-good curve. The height bound of $\gamma_{ijk,s}$ follows from the argument in (1). 

By Construction \ref{constructj1} (2), $\gamma_{ijk,s}$ is homologous to the concatenation of $i_s(e_{ij}), i_s(e_{jk}), i_s(e_{ki})$. If $\Delta_{ijk}$ is a triangle in $N_{\gamma}^{(2)}$, Lemma \ref{preparation} (4) implies that $\gamma_{ijk,s}$ is null-homologous in $M$.
If $\Delta_{ijk}$ is not contained in $N_{\gamma}^{(2)}$, it is a meridian disc of $\mathcal{N}(\gamma)$, then Lemma \ref{preparation} (4) implies that $\gamma_{ijk,1}\cup \gamma_{ijk,2}$ is homologous to ${\bf i}_1(\partial \Delta_{ijk})+{\bf i}_2(\partial \Delta_{ijk})$, which is null-homologous in $M$ by Lemma \ref{preparation} (3).
\end{enumerate}
\end{proof}

Now we construct the $2$-complex $Z$, by adding surfaces to two copies of $Z^{1}$. Rigorously speaking, two copies of $Z^1$ are not the $1$-skeleton of $Z$ as a CW-complex, but we still call it the $1$-skeleton of $Z$, for our convenience. The map $j^1:Z^1\to M$ and the construction of $Z$ below automatically give the desired immersion $j:Z\looparrowright M$.

\begin{construction}\label{constructZ2}
 We take $R'$ to be an integer greater than all the $R_{ij}$ and $R_{ijk}$ in Lemma \ref{checkgoodcurve}.
\begin{enumerate}
\item Recall that $Z^1$ has two components: $Z^1_1$ and $Z^1_2$, and each of them is isomorphic to $N^{(1)}$. The $1$-skeleton $Z^{(1)}$ of $Z$ consists of two copies of $Z^1$, so we have $Z^{(1)}=Z^{1,1}_1\cup Z^{1,1}_2\cup  Z^{1,2}_1\cup Z^{1,2}_2$. For any $s=1,2$, the restriction of $j$ to $Z^{1,1}_s$ and $Z^{1,2}_s$ equals $j^1|_{Z^1_s}$. We denote the two copies of $\partial_p Z^1$ in $Z^{(1)}$ by $\partial_p Z^{(1)}$.

\item For any triangle $\Delta_{ijk}$ in $\partial N_0$ with vertices $n_i,n_j,n_k$, and any component $Z^{1,t}_s$ of $Z^{(1)}$ with $s,t\in \{1,2\}$, we paste a triangle $\Delta_{ijk,s}^{Z,t}$ to $Z^{1,t}_s$ along the concatenation of edges $e^{Z}_{ij,s}, e^{Z}_{jk,s},e^{Z}_{ki,s}$ in $Z^{1,t}_s$. Since $e^{Z}_{ij,s}$ is mapped to a path homotopic to $i_s(e_{ij})$ (Construction \ref{constructj1} (1)) and $i_s|_{\partial N_0^{(1)}}$ extends to $i_{\partial,s}:\partial N_0\to T_s$ (Lemma \ref{preparation} (4)), the $j^{1}$-image of this concatenation is null-homotopic in $M$, so we map the triangle $\Delta_{ijk,s}^{Z,t}$ to the corresponding totally geodesic triangle in $M$.

\item For any bigon $B_{ijk}$ in $N_{\gamma}^{(2)}$ containing an edge $e_{ij}\subset \partial N_0$ and any $s=1,2$, we do the following construction.

By Lemma \ref{checkgoodcurve} (1), the concatenation of $\overline{j^{1}(e_{ij,s}^{Z})},j^{1}(e_{ijk,s}^{Z})$ is homotopic to a null-homologous $(R_{ij},\epsilon)$-good curve $\gamma_{ijk,s}$ in $M$, via a nearly geodesic two-cornered annulus $A_{ijk,s}$ (see Figure \ref{figure3} (a)). Let $\vec{w}_{ijk,s}$ be the tangent vector of the shortest geodesic in $A_{ijk,s}$ from $\gamma_{ijk,s}$ to $m_{i,s}$, and let $\vec{v}_{ijk,s}=\vec{w}_{ijk,s}+(1+\pi i)\in N^1(\sqrt{\gamma_{ijk,s}})$.
By Proposition \ref{boundingsurface} and Remark \ref{combinatorialdistance}, two copies of $\gamma_{ijk,s}$ bound an $(R_{ij},R',\epsilon)$-nearly geodesic subsurface $S_{ijk,s}\looparrowright M$, such that the following hold.
\begin{enumerate}
\item The two feet of $S_{ijk,s}$ on two copies of $\gamma_{ijk,s}$ are both $\frac{\epsilon}{R}$-close to $\vec{v}_{ijk,s}$.
\item Any essential path in $S_{ijk,s}$ with end points in $\partial S_{ijk,s}$ must have combinatorial length (with respect to the decomposition of $S_{ijk,s}$ to pants and hamster wheels) at least $R'e^{\frac{R'}{2}}$.
\end{enumerate}

We identify the two boundary components of $S_{ijk,s}$ with the two copies of concatenations of $e_{ij,s}^{Z},e_{jk,s}^{Z}, e_{ki,s}^{Z}$ in $Z^{1,1}_s$ and $Z^{1,2}_s$ respectively. The restriction of $j$ on $S_{ijk,s}$ is naturally defined by two copies of the nearly geodesic $2$-cornered annulus $A_{ijk,s}$ and the above $(R_{ij},R',\epsilon)$-nearly geodesic subsurface $S_{ijk,s}\looparrowright M$.

\item For any triangle $\Delta_{ijk}$ in $N_{\gamma}^{(2)}$ not contained in $\partial N_0$ and any $s=1,2$, we do the following construction.

By Lemma \ref{checkgoodcurve} (2), the concatenation of $j^{1}(e_{ij,s}^{Z}),j^{1}(e_{jk,s}^{Z}), j^{1}(e_{ki,s}^{Z})$ is homotopic to a null-homologous $(R_{ijk},\epsilon)$-good curve $\gamma_{ijk,s}$ in $M$, via a nearly geodesic three-cornered annulus $A_{ijk,s}$ (see Figure \ref{figure3} (b)). Let $\vec{w}_{ijk,s}$ be the tangent vector of the shortest geodesic in $A_{ijk,s}$ from $\gamma_{ijk,s}$ to $m_{i,s}$, and let $\vec{v}_{ijk,s}=\vec{w}_{ijk,s}+(1+\pi i)\in N^1(\sqrt{\gamma_{ijk,s}})$.
By Proposition \ref{boundingsurface} and Remark \ref{combinatorialdistance}, two copies of $\gamma_{ijk,s}$ bound an $(R_{ijk},R',\epsilon)$-nearly geodesic subsurface $S_{ijk,s}\looparrowright M$, such that conditions (a) (b) in item (3) hold.

We identify the two boundary components of $S_{ijk,s}$ with the two copies of concatenations $e_{ij,s}^{Z},e_{jk,s}^{Z}, e_{ki,s}^{Z}$ in $Z^{1,1}_s$ and $Z^{1,2}_s$ respectively. The restriction of $j$ on $S_{ijk,s}$ is naturally defined by two copies of the nearly geodesic $3$-cornered annulus $A_{ijk,s}$ and the above $(R_{ijk},R',\epsilon)$-nearly geodesic subsurface $S_{ijk,s}\looparrowright M$.

\item Up until now, the $2$-complex we have constructed has (at least) two components, one containing $Z^{1,1}_1\cup Z^{1,2}_1$ and one containing $Z^{1,1}_2\cup Z^{1,2}_2$. For any triangle $\Delta_{ijk}$ in $N^{(2)}$ not contained in $N_{\gamma}^{(2)}$, we do the following construction.

By Lemma \ref{checkgoodcurve} (3), for any $s=1,2$, the concatenation $j^{1}(e_{ij,s}^{Z}),j^{1}(e_{jk,s}^{Z}),j^{1}(e_{ki,s}^{Z})$ is homotopic to an $(R_{ijk},\epsilon)$-good curve $\gamma_{ijk,s}$ via a three-cornered annulus $A_{ijk,s}$ in $M$, and $\gamma_{ijk,1}\cup \gamma_{ijk,2}$ is null-homologous in $M$. 
Let $\vec{w}_{ijk,s}$ be the tangent vector of the shortest geodesic in $A_{ijk,s}$ from $\gamma_{ijk,s}$ to $m_{i,s}$, and let $\vec{v}_{ijk,s}=\vec{w}_{ijk,s}+(1+\pi i)\in N^1(\sqrt{\gamma_{ijk,s}})$.
By Proposition \ref{boundingsurface} and Remark \ref{combinatorialdistance}, two copies of $\gamma_{ijk,1}\cup \gamma_{ijk,2}$ bound an $(R_{ijk},R',\epsilon)$-nearly geodesic surface $S_{ijk}\looparrowright M$, such that conditions (a) (b) in item (3) hold, with $S_{ijk,s}$ replaced by $S_{ijk}$.

We identify the four boundary components of $S_{ijk}$ with two copies of concatenations of $e_{ij,1}^{Z},e_{jk,1}^{Z}, e_{ki,1}^{Z}$ and $e_{ij,2}^{Z},e_{jk,2}^{Z}, e_{ki,2}^{Z}$ in $Z^{1,1}_1, Z^{1,2}_1$ and $Z^{1,1}_2, Z^{1,2}_2$ respectively. The restriction of $j$ on $S_{ijk}$ is naturally defined by two copies of the nearly geodesic $3$-cornered annuli $A_{ijk,1}\cup A_{ijk,2}$, and the above $(R_{ijk},R',\epsilon)$-nearly geodesic subsurface $S_{ijk}\looparrowright M$.

\end{enumerate}
\end{construction}

\begin{figure}
\centering
\label{new picture}
\def\svgwidth{.7\textwidth}
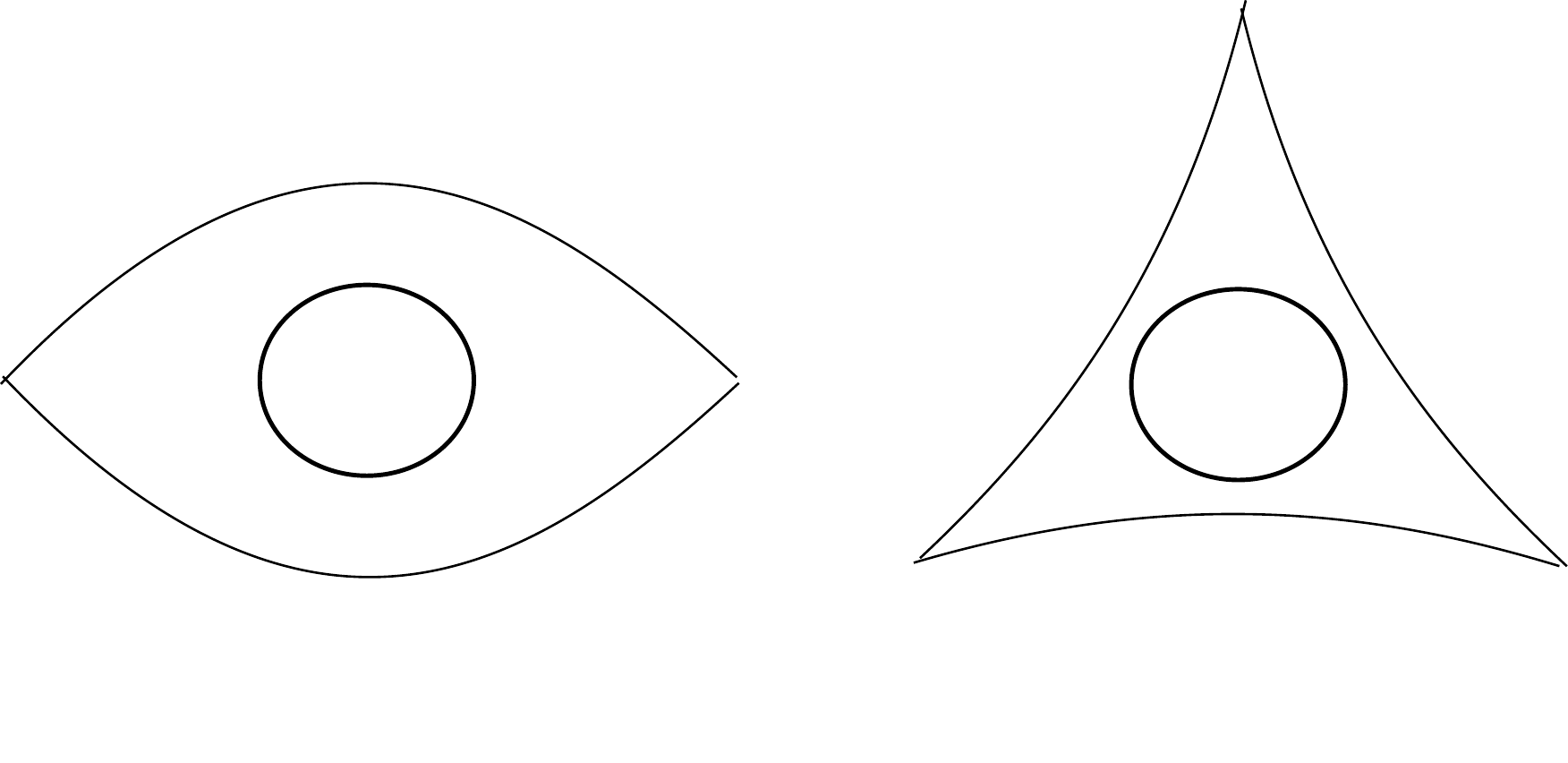
\caption{The two-cornered annulus and three-cornered annulus in Construction \ref{constructZ2}.}
\label{figure3}
\end{figure}

Note that the surfaces $S_{ijk,s}$ and $S_{ijk}$ in Construction \ref{constructZ2} (3), (4), (5) may not be connected, so $Z$ may not be connected and has at most four connected components. One can actually work harder as in \cite{LM} to make sure these surfaces are connected, but we choose to save some work here.

Alternatively, $Z$ is obtained from four copies of $N^{(2)}$, denoted by $N^{(2),1}_1, N^{(2),1}_2, N^{(2),2}_1, N^{(2),2}_2$ respectively, by making the following modifications.
\begin{enumerate}
\item Each triangle in $N^{(2)}\cap \partial N_0$ is not modified. 

\item For any bigon $B_{ijk}$ or triangle $\Delta_{ijk}$ in $N_{\gamma}^{(2)}\setminus \partial N_0$ and any $s=1,2$, the two copies of $B_{ijk}$ or $\Delta_{ijk}$ in $N^{(2),1}_s, N^{(2),2}_s$ are replaced by a compact orientable surface $S_{ijk,s}$ with two boundary components.

%\item  For each triangle $\Delta_{ijk}$ in $N_{\gamma}^{(2)}$ and any $s=1,2$, the two copies of $\Delta_{ijk}$ in $N^{(2),1}_s, N^{(2),2}_s$ are replaced by a compact orientable surface $S_{ijk,s}$ with two boundary components.

\item For each triangle $\Delta_{ijk}$ in $N^{(2)}$ not contained in $N_{\gamma}^{(2)}$, its four copies are replaced by a compact orientable surface $S_{ijk}$ with four boundary components.
\end{enumerate}

\bigskip

\subsection{Contruction of virtual proper domination}\label{constructdomination2}

In this section, we construct the desired finite cover $M'$ of $M$ and the proper non-zero map $f:M'\to N$, modulo a $\pi_1$-injectivity result (Theorem \ref{pi1injectivity}).

Recall that the $1$-skeleton $Z^{(1)}$ of $Z$ has four identical components. So $Z$ has at most four components, and we take a component of $Z$ that contains the least number (one, two or four) of components of the $1$-skeleton. We abuse notation and still denote this component by $Z$, and we still denote the restriction of $j$ on this component by $j$.

The next section is devoted to prove the $\pi_1$-injectivity of $j:Z\looparrowright M$ and some further refinements. To state this result, we first need to define a compact $3$-manifold $\mathcal{Z}$ with boundary.

Recall that the cellulation of $N_0$ in Construction \ref{cellulation} induces a handle structure on a compact neighborhood $\mathcal{N}(N_0)$ of $N_0$. Let $\mathcal{N}^{(1)}$ be the union of $0$- and $1$-handles of this handle structure. Then we define a compact oriented $3$-manifold $\mathcal{Z}$ as following. 

\begin{construction}\label{constructmathcalZ}
\begin{enumerate}
\item We start with four copies of $\mathcal{N}^{(1)}$, denoted by $\mathcal{N}^{(1),1}_1, \mathcal{N}^{(1),1}_2, \mathcal{N}^{(1),2}_1, \mathcal{N}^{(1),2}_2$ respectively.

\item For any triangle $\Delta_{ijk}$ contained in $N^{(2)}\cap \partial N_0$, the corresponding $2$-handle is pasted to $\mathcal{N}^{(1)}$ along an annulus $L_{ijk}$. Then for each copy of $\mathcal{N}^{(1)}$, we attach a $2$-handle along the same annulus $L_{ijk}$.

\item For any bigon $B_{ijk}$ or triangle $\Delta_{ijk}$ in $N_{\gamma}^{(2)}\setminus \partial N_0$, the corresponding $2$-handle is pasted to $\mathcal{N}^{(2)}$ along an annulus $L_{ijk}$. For any $s=1,2$, we take the surface $S_{ijk,s}$ constructed in Construction \ref{constructZ2} (3) or (4), which has two boundary components. Then we attach $S_{ijk,s}\times I$ to $\mathcal{N}^{(1),1}_s\cup \mathcal{N}^{(1),2}_s$ via an orientation reversing homeomorphism from $\partial S_{ijk,s}\times I$ to copies of $L_{ijk}$ in $\mathcal{N}^{(1),1}_s\cup \mathcal{N}^{(1),2}_s$. 

%\item For any triangle $\Delta_{ijk}$ in $N^{(2)}_{\gamma}$ not contained in $\partial N_0$, the corresponding $2$-handle is pasted to $\mathcal{N}^{(1)}$ along an annulus $L_{ijk}$. For any $s=1,2$,  we take the surface $S_{ijk,s}$ constructed in Construction \ref{constructZ2} (4), which has two boundary components. Then  we paste $S_{ijk,s}\times I$ to $\mathcal{N}^{(1),1}_s\cup \mathcal{N}^{(1),2}_s$ via an orientation reversing homeomorphism from $\partial (S_{ijk,s}\times I)$ to copies of $L_{ijk}$ in $\mathcal{N}^{(1),1}_s\cup \mathcal{N}^{(1),2}_s$.

\item For any triangle $\Delta_{ijk}$ in $N^{(2)}$ not contained in $N^{(2)}_{\gamma}$, the corresponding $2$-handle is pasted to $\mathcal{N}^{(1)}$ along an annulus $L_{ijk}$. We take the surface $S_{ijk}$ constructed in Construction \ref{constructZ2} (5), which has four boundary components. Then we paste $S_{ijk}\times I$ to the four copies of $\mathcal{N}^{(1)}$ via an orientation reversing homeomorphism from $\partial S_{ijk}\times I$ to the four copies of $L_{ijk}$.
\end{enumerate}

Then we take $\mathcal{Z}$ to be the component of the resulting manifold containing the least number ($1, 2$ or $4$) of components of $\mathcal{N}^{(1)}$.
\end{construction}

In the construction of $\mathcal{Z}$ (Construction \ref{constructmathcalZ}), after the second step, we obtain four copies of the same $3$-manifold, which is homeomorphic to the union of $\mathcal{N}^{(1)}$ and a neighborhood of $\partial N_0$. Here each copy has a unique boundary component homeomorphic to the torus. Since further constructions in Construction \ref{constructmathcalZ} do not affect these four tori, we obtain at most four tori components of $\partial \mathcal{Z}$ and we denote their union by $\partial_p\mathcal{Z}$. 

Now we can state the result to be proved in Section \ref{pi1inj2}.

\begin{theorem}\label{pi1injectivity}
For any small $\epsilon\in (0,10^{-2})$, there exists $R_0>0$ depending on $M$ and $\epsilon$, such that the following statement holds for any $R>R_0$. If the construction of $j:Z\looparrowright M$ satisfies all conditions in Constructions \ref{constructZ1}, \ref{constructj1}, \ref{constructZ2} (involving $\epsilon$ and $R$), then $j:Z\looparrowright M$ is $\pi_1$-injective and the $\pi_1$-image $j_*(\pi_1(Z))<\pi_1(M)$ is a geometrically finite subgroup. 

Moreover, the convex core of the covering space $\tilde{M}$ of $M$ corresponding to $j_*(\pi_1(Z))<\pi_1(M)$ is homeomorphic to the $3$-manifold $\mathcal{Z}\setminus \partial_p\mathcal{Z}$ as oriented manifolds, and the cusp ends of $\tilde{M}$ corresponding to $\partial_p \mathcal{Z}$ are mapped to $T_1\cup T_2\subset \partial M$ via the covering map.
\end{theorem}

The proof of Theorem \ref{pi1injectivity} is more complicated than proofs of corresponding $\pi_1$-injectivity results in \cite{Sun2} and \cite{Sun5}, since the construction of the mapped-in $2$-complex $j:Z\looparrowright M$ is more complicated. The proof of Theorem \ref{pi1injectivity} is more geometric, which is in different flavor from other constructions in this section, so we defer its proof to Section \ref{pi1inj2}.

For $\mathcal{N}^{(2)}$, it has a unique boundary component homeomorphic to the torus, and we denote it by $\partial_p \mathcal{N}^{(2)}$.
The following elementary property of $\mathcal{Z}$ is important for the construction of virtual domination.

\begin{lemma}\label{elementary}
There exists a proper map $g:(\mathcal{Z},\partial \mathcal{Z})\to (\mathcal{N}^{(2)},\partial \mathcal{N}^{(2)})$ of $\text{deg}(g)\in \{1,2,4\}$, such that 
$g^{-1}(\partial_p \mathcal{N}^{(2)})=\partial_p \mathcal{Z}$, and the restriction of $g$ on each component of $\partial_p \mathcal{Z}$ is an orientation preserving homeomorphism to $\partial_p \mathcal{N}^{(2)}$ .
\end{lemma}

\begin{proof}
We construct $g$ by following the steps in Construction \ref{constructmathcalZ} that constructs $\mathcal{Z}$. 

At first, $g$ maps one, two or four copies of $\mathcal{N}^{(1)}$ in $\mathcal{Z}$ to $\mathcal{N}^{(1)}\subset \mathcal{N}^{(2)}$ by identity. Then $g$ maps each $2$-handle in $\mathcal{Z}$ corresponding to $\Delta_{ijk}\subset \partial N_0$ to the corresponding $2$-handle in $\mathcal{N}^{(2)}$ by homeomorphism. So $g$ maps $\partial_p \mathcal{Z}$ to $\partial_p \mathcal{N}^{(2)}$, and the restriction of $g$ on each component of $\partial_p \mathcal{Z}$  is an orientation preserving homeomorphism to $\partial_p \mathcal{N}^{(2)}$.

In steps (3), (4) of Construction \ref{constructmathcalZ}, each $2$-handle in $\mathcal{N}^{(2)}$ not contained in $\partial N_0$ corresponds to a bigon $B_{ijk}$ or triangle $\Delta_{ijk}$, and it gives $(S_{ijk,1}\times I) \cup (S_{ijk,2}\times I)$ or $S_{ijk}\times I$ in $\mathcal{Z}$. Then each component of $(S_{ijk,1}\times I) \cup (S_{ijk,2}\times I)$ or $S_{ijk}\times I$ in $\mathcal{Z}$ is mapped to the corresponding $2$-handle $B_{ijk}\times I$ or $\Delta_{ijk}\times I$ in $\mathcal{N}^{(2)}$, via the product of a pinching map on surfaces and the identity on $I$. Then we get a proper map $g:(\mathcal{Z},\partial \mathcal{Z})\to (N^{(2)},\partial N^{(2)})$ that maps $\partial \mathcal{Z}\setminus \partial_p \mathcal{Z}$ to $\partial \mathcal{N}^{(2)}\setminus \partial_p \mathcal{N}^{(2)}$. 
%Moreover, each component of $(\partial \mathcal{N}^{(2)})\setminus (\partial_p \mathcal{N}^{(2)})$ is the boundary of a $3$-cell, which is homeomorphic to $S^2$.

We have $\text{deg}(g)\in \{1,2,4\}$ since $g^{-1}(\mathcal{N}^{(1)})$ consists of $1$, $2$ or $4$ copies of $\mathcal{N}^{(1)}$, and the restriction of $g$ on each such component is identity. 

\end{proof}

Now we are ready to prove Theorem \ref{main3}. 

\begin{proof}[Proof of Theorem \ref{main3}]

At first, we consider $M$ as a noncompact $3$-manifold with cusp ends. We take small $\epsilon>0$ and large enough $R>0$ such that Constructions \ref{constructj1}, \ref{constructZ2} and Theorem \ref{pi1injectivity} hold, and we construct a mapped in $2$-complex $j:Z\looparrowright M$.
By Theorem \ref{pi1injectivity}, $j_*:\pi_1(Z)\to \pi_1(M)$ is injective and the convex core of $j_*(\pi_1(Z))<\text{Isom}_+(\mathbb{H}^3)$ is homeomorphic to $\mathcal{Z}\setminus \partial_p \mathcal{Z}$. Let $\tilde{M}$ be the covering space of $M$ corresponding to $j_*(\pi_1(Z))<\pi_1(M)$, then it contains a submanifold homeomorphic to $\mathcal{Z}\setminus \partial_p \mathcal{Z}$, such that ends of $\mathcal{Z}\setminus \partial_p \mathcal{Z}$ correspond to cusp ends of $\tilde{M}$. 

Now we chop off cusp ends of $M$ and consider it as a compact $3$-manifold with boundary. As a manifold with boundary, $\tilde{M}$ contains a compact submanifold homeomorphic to $\mathcal{Z}$ such that $\mathcal{Z}\cap \partial \tilde{M}=\partial_p \mathcal{Z}$.

By Agol's celebrated result that hyperbolic $3$-manifold groups are LERF (\cite{Agol2}), the covering map $\tilde{M}\to M$ factors through a finite cover $M'$ of $M$, such that $\mathcal{Z}$ is mapped into $M'$ via embedding, we have $\mathcal{Z}\cap \partial M'=\partial_p \mathcal{Z}$. 

Recall that we have a handle decomposition of a neighborhood $\mathcal{N}(N_0)$ of $N_0$, and we can identify $\mathcal{N}(N_0)$ with $N$. By Lemma \ref{elementary}, there is a proper map $g:(\mathcal{Z},\partial \mathcal{Z})\to (\mathcal{N}^{(2)},\partial \mathcal{N}^{(2)})$ such that $g^{-1}(\partial_p \mathcal{N}^{(2)})=\partial_p \mathcal{Z}$ and $\text{deg}(g)\in \{1,2,4\}$. Note that each component of $\partial N^{(2)}\setminus \partial_p N^{(2)}$ is the boundary of a $3$-cell, which is homeomorphic to $S^2$. Then we extend $g$ to a proper map $f:M'\to N$ as following. 

Let $K$ be a component of $M'\setminus \mathcal{Z}$, then $K$ is disjoint from $\partial_p \mathcal{Z}$ and $\partial K$ is the union of $\partial_p K=K\cap \partial M'$ and $\partial_i K=K\cap \mathcal{Z}$. Then $\partial_i K$ has a neighborhood $\partial_i K\times I$ in $K$. Since $g$ maps each component of $\partial_i K\subset \partial \mathcal{Z}$ to an $S^2$-component of $\partial \mathcal{N}^{(2)}$ that bounds a $3$-cell, we first map $\partial_i K\times I\subset K$ to a union of $3$-cells in $N_0$ that maps $\partial_i K\times 0=\partial_i K$ via $g$ and maps $\partial_i K\times 1$ to centers of $3$-cells. Then we map $K\setminus (\partial_i K\times I)$ to a graph in $\mathcal{N}(N_0)$ such that each component of $\partial_i K\times 1$ is mapped to the corresponding center of $3$-cell and each component of $\partial_p K$ is mapped to a point in $\partial (\mathcal{N}(N_0))$. Moreover, we can assume that this graph misses $\mathcal{N}^{(1)}$. By the above definition of $f$ on components of $M'\setminus \mathcal{Z}$, we get a proper map $f:(M',\partial M')\to (\mathcal{N}(N_0),\partial \mathcal{N}(N_0))=(N,\partial N)$.

For any point $p\in \mathcal{N}^{(1)}$, we have $f^{-1}(p)=g^{-1}(p)$, while $f$ and $g$ have the same local mapping degree at $f^{-1}(p)$. So $\text{deg}(f)=\text{deg}(g)\in \{1,2,4\}$ holds.
\end{proof}

\bigskip

\subsection{Proof of Theorem \ref{mixed}}\label{proof1.3}

The proof of Theorem \ref{mixed} is similar to the proof of Theorem \ref{main2}, and we sketch the proof in the following. 

We start with a compact oriented mixed $3$-manifold $M$ with tori boundary such that its boundary intersects with a hyperbolic JSJ piece, and a compact oriented $3$-manifold $N$ with tori boundary.

At first, Propositions \ref{homology2boundarymixed} and \ref{reduction1cusp} imply the following hold.
\begin{itemize}
\item $M$ has a finite cover $M'$ with two boundary components $T_1,T_2$ contained in the same hyperbolic JSJ piece $M_0'\subset M'$, such that the kernel of $H_1(T_1\cup T_2;\mathbb{Z})\to H_1(M_0';\mathbb{Z})$ induced by inclusion  contains an element that has nontrivial components in both $H_1(T_1;\mathbb{Z})$ and $H_1(T_2;\mathbb{Z})$.
\item $N$ is virtually $2$-dominated by a one-cusped oriented hyperbolic $3$-manifold $N'$.
\end{itemize}
By the argument at the beginning of Section \ref{topo2} (Theorem \ref{main3} implies Theorem \ref{main2}), it suffices to prove that $M'$ virtually dominates $N'$ with virtual mapping degree in $\{1,2,4\}$. By abusing notation, we still use $M$ and $N$ to denote $M'$ and $N'$ respectively, and use $M_0$ to denote the hyperbolic piece of $M$ containing $T_1,T_2\subset \partial M$.

Then we take a geometric cellulation of $N$ as in Constructions \ref{triangulation} and \ref{cellulation}, and construct instructional maps and homomorphisms 
as in Lemma \ref{preparation}, with $M$ replaced by $M_0$. Then we choose small enough $\epsilon>0$ and large enough $R>0$, and follow Constructions \ref{constructZ0}, \ref{constructZ1}, \ref{constructj1} and \ref{constructZ2} to construct a map $j:Z\looparrowright M_0$. By Theorem \ref{pi1injectivity} (to be proved in Section \ref{pi1inj2}), if $R$ is large enough, $j:Z\looparrowright M_0$ is $\pi_1$-injective, and the convex core of $j_*(\pi_1(Z))<\pi_1(M_0)<\text{Isom}_+(\mathbb{H}^3)$ is homeomorphic to $\mathcal{Z}\setminus \partial_p \mathcal{Z}$. 

So the covering space $\tilde{M_0}$ of $M_0$ corresponding to $j_*(\pi_1(Z))<\pi_1(M_0)$ contains a submanifold homeomorphic $\mathcal{Z}\setminus \partial_p \mathcal{Z}$, and $\partial_p \mathcal{Z}$ corresponds to cusp ends of $\tilde{M_0}$. Now we chop off cusp ends of $M_0$ and consider it as a compact $3$-manifold with tori boundary, the corresponding $\tilde{M_0}$ contains a compact submanifold homeomorphic to $\mathcal{Z}$ such that $\partial \tilde{M_0}\cap \mathcal{Z}=\partial_p \mathcal{Z}$, and $\partial \tilde{M_0}\cap \mathcal{Z}$ is mapped to $T_1\cup T_2\subset \partial M_0\cap \partial M$.

Let $\tilde{M}$ be the covering space of $M$ corresponding to $j_*(\pi_1(Z))<\pi_1(M_0)<\pi_1(M)$, then it contains a submanifold homeomorphic to $\tilde{M_0}$.
Since $\partial M$ contains the tori boundary components $T_1,T_2$ of $M_0$, $\mathcal{Z}\subset \tilde{M_0}$ is also a compact submanifold of $\tilde{M}$ such that $\partial \tilde{M}\cap \mathcal{Z}=\partial_p \mathcal{Z}$. Since $j:Z\to M$ maps into a JSJ hyperbolic piece $M_0\subset M$, by \cite{Sun3} (which heavily relies on \cite{Agol2}), $j_*(\pi_1(Z))$ is a separable subgroup of $\pi_1(M)$. Then by \cite{Sco}, there is an intermediate finite cover $M'$ of $\tilde{M}\to M$, such that $\mathcal{Z}$ is mapped into $M'$ via embedding, and $\mathcal{Z}\cap \partial M'=\partial_p \mathcal{Z}$ holds.

Again, we identify $N$ with a neighborhood $\mathcal{N}(N_0)$ of $N_0$, which has an induced handle structure. By Lemma \ref{elementary}, there is a proper map $g:(\mathcal{Z},\partial \mathcal{Z})\to (\mathcal{N}^{(2)},\partial \mathcal{N}^{(2)})$ such that $g^{-1}(\partial_p \mathcal{N}^{(2)})=\partial_p \mathcal{Z}$ and $\text{deg}(g)\in \{1,2,4\}$. Then we extend $g$ to a proper map $f:(M',\partial M')\to (\mathcal{N}(N_0), \partial \mathcal{N}(N_0))$. For each component $K$ of $M'\setminus \mathcal{Z}$, the map $f$ can be defined exactly as in the proof of Theorem \ref{main3}, which maps each component of $K\cap \partial M'$ to a point in $\partial \mathcal{N}(N_0)$. Then we have $\text{deg}(f)=\text{deg}(g)\in \{1,2,4\}.$

\bigskip
\bigskip

\section{Proof of $\pi_1$-injectivity}\label{pi1inj2}

In this section, we devote to prove the $\pi_1$-injectivity result Theorem \ref{pi1injectivity}, and the proof is more difficult than the $\pi_1$-injectivity results in \cite{Sun2} and \cite{Sun5}. One reason is that our constructions of $Z$ and $j:Z\looparrowright M$ are more complicated. The more important reason is that the induced map $\tilde{j}:\tilde{Z}\to \tilde{M}=\mathbb{H}^3$ on universal covers is not a quasi-isometric embedding.

We sketch the structure of this section in the following. In Section \ref{technical}, we define a (ideal) $3$-complex $Z^{3}$ that contains $Z$ as a deformation retract, then we define a family of representations $\rho_t:\pi_1(Z^{3})\to \text{Isom}(\mathbb{H}^3)$ and a family of maps $\tilde{j}_t:\tilde{Z^{3}}\to \mathbb{H}^3$ that is $\rho_t$-equivariant, such that $\tilde{j}_1|_{\tilde{Z}}=\tilde{j}$ and $\rho_1=j_*$. In Section \ref{estimate2}, we prove that $\tilde{j}_0:\tilde{Z^{3}}\to \mathbb{H}^3$ is a quasi-isometric embedding and some further details, via a more complicated definition of {\it modified sequence} than in \cite{Sun2}. In Section \ref{qi2}, we prove that any $\tilde{j}_t$ is a quasi-isometric embedding, and finish the proof of Theorem \ref{pi1injectivity}.

\subsection{An extension of $Z$ and a family of maps}\label{technical}

By Construction \ref{constructZ2} (2), for any $s,t\in \{1,2\}$, each triangle $\Delta_{ijk}\subset \partial N_0$ gives a triangle $\Delta_{ijk,s}^{Z,t}$ in $Z$ with $s,t\in \{1,2\}$. If we fix $s$ and $t$, then the union of all such $\Delta_{ijk,s}^{Z,t}$ gives a torus $T_s^t\subset Z$. Each $T_s^t\subset Z$ is combinatorially isomorphic to $\partial N_0$, and the restriction of $j$ on $T_s^t$ is homotopic to a covering map to $T_s\subset M$. Since both $T_1$ and $T_2$ are peripheral tori in $M$, the lifting of $j|:T_s^t\to M$ to the universal cover is not a quasi-isometric embedding, so neither is $\tilde{j}:\tilde{Z}\to \tilde{M}=\mathbb{H}^3$. To treat this undesired situation, we extend $Z$ to an ideal $3$-complex $Z^{3}$ as following. 

\begin{definition}\label{Z3}
For each triangulated torus $T_s^t\subset Z$ with $s,t=1,2$, we take a cone over $T_s^t$ and delete the cone point, and denote the resulting ideal $3$-complex by $C(T_s^t)$. We define $Z^{3}$ to be the union of $Z$ and all these $C(T_s^t)$.
\end{definition}

It is clear that $Z^3$ deformation retracts to $Z$. For any vertex $v_{i,s}^t$, edge $e_{ij,s}^{Z,t}$ and triangle $\Delta_{ijk,s}^{Z,t}$ (copies of $v_i$, $e_{ij}^Z$, $\Delta_{ijk}$ in $T_s^t$) contained in $T_s^t$, their cones give an edge, a triangle and a tetrahedron in $Z^3$ respectively, with the cone point deleted. We denote these cones by $e_{i\infty,s}^{Z,t}$, $\Delta_{ij\infty,s}^{Z,t}$ and $T_{ijk\infty,s}^{Z,t}$ respectively.

We define $j_1:Z^3\looparrowright M$ as following, which is an extension of $j:Z\looparrowright M$. We require that $j_1$ maps $e_{i\infty,s}^{Z,t}$ to the geodesic ray from $j(v_{i,s}^t)$ to the ideal point corresponding to the cusp end $T_s$ of $\partial M$, maps $\Delta_{ij\infty,s}^{Z,t}$ to the ideal triangle given by $j(e_{ij,s}^{Z,t})$ and the ideal point,
and maps $T_{ijk\infty,s}^{Z,t}$ to the ideal tetrahedron given by $j(\Delta_{ijk,s}^{Z,t})$ and the ideal point. To prove Theorem \ref{pi1injectivity}, the main step is to prove that the lifting $\tilde{j_1}:\tilde{Z^3}\to \tilde{M}=\mathbb{H}^3$ of $j_1:Z^3\looparrowright M$ to the universal cover is a quasi-isometric embedding. 

Now we define a family of maps $\{\tilde{j_t}:\tilde{Z^3}\to \mathbb{H}^3\ |\ t\in [0,1]\}$ (the map given by $t=1$ is the above $\tilde{j_1}$) and a family of representations $\{\rho_t:\pi_1(Z^3)\to \text{Isom}_+(\mathbb{H}^3) \ |\ t\in [0,1]\}$, such that $\tilde{j_t}$ is $\rho_t$-equivariant. We also have the property that $\tilde{j_0}$ maps each component of the preimage of $Z\setminus Z^{(1)}$ in $\tilde{Z^3}$ to a totally geodesic subsurface of $\mathbb{H}^3$.

The map $j_1:Z^3\looparrowright M$ gives us the following parameters. All parameters are similar to the ones given in Parameter 5.4 of \cite{Sun5}, except items (3) and (4).
\begin{parameter}\label{parameterj1}
\begin{enumerate}
\item For each vertex $v_{i,s}^t\in Z^{(0)}$ and each edge $e_{ij,s}^{Z,t}\subset Z^{(1)}\setminus \partial_p Z^{(1)}$ (or $e_{ijk,s}^{Z,t}$) adjacent to $v_{i,s}$,
the initial frame of $\mathfrak{s}_{ij,s}$ (or $\mathfrak{s}_{ijk,s}$) equals $F_{ijk}^{M}\cdot A_{ij,s}$ (or $F_{ijk,s}^{M}\cdot A_{ijk,s}$) for some  $A_{ij,s}\in \text{SO}(3)$ (or $A_{ijk,s}$) that is $\frac{\epsilon}{10}$-close to $id\in \text{SO}(3)$ (see Construction \ref{constructj1} (2) (b) and (3) (b)).

\item For each edge $e_{ij,s}^{Z,t}\in Z^{(1)}\setminus \partial_p Z^{(1)}$ (or $e_{ijk,s}^{Z,t}$), the complex length of the associated $\partial$-framed segment $\mathfrak{s}_{ij,s}$ (or $\mathfrak{s}_{ijk,s}$) equals $2R+\lambda_{ij,s}$ (or $2R+\lambda_{ijk,s}$) for some complex number $\lambda_{ij,s}$ (or $\lambda_{ijk,s}$) with modulus at most $\frac{\epsilon}{5}$ (see Construction \ref{constructj1} (2) (a) and (3) (a)).

\item For each edge $e_{ij,s}^{Z,t}\subset \partial_p Z^{(1)}$, $j(e_{ij,s}^{Z,t})$ is homotopic to an Euclidean geodesic segment $g_{ij,s}$ in the horotorus $T_s\subset M$ of length $(1+\tau_{ij,s})R$ with $|\tau_{ij,s}|<\epsilon$, (see Construction \ref{constructj1} (1) and Lemma \ref{preparation} (1).)

\item For each vertex $v_{i,s}^t\in \partial_p Z^{(1)}$, we take a preferred edge $e_{ij,s}^{Z,t}\subset \partial_p Z^{(1)}$, and let $n_k$ be the vertex of $N$ such that $n_i,n_j,n_k$ form a triangle in $N^{(2)}$ not contained in $\partial N_0$, and $n_k$ lies to the left of the edge $e_{ij}$ (as in Figure \ref{figure2}). Let ${\bf F}_{i,s}=(j(v_{i,s}^t),\vec{v}_{ij,s},\vec{n}_{i,s})$ be the frame based at $j(v_{i,s}^t)$ such that $\vec{v}_{ij,s}$ is tangent to $g_{ij,s}$ (in item (3)) and $\vec{n}_{i,s}$ is tangent to $j_1(e_{i\infty, s}^{Z,t})$. Then we have $${\bf F}_{ijk,s}^{M}={\bf F}_{i,s}\cdot X\cdot B_{i,s},$$ where $X\in \text{SO}(3)$ satisfies 
$\begin{pmatrix}
\vec{v}_2^T & \vec{v}_3^T & (\vec{v}_2\times \vec{v}_3)^T
\end{pmatrix}=\begin{pmatrix}
1 & 0 & 0\\
0 & 0& -1\\
0 & 1 & 0
\end{pmatrix}\cdot X$ for vectors $\vec{v}_2,\vec{v}_3$ in Remark \ref{coordinate} and $B_{i,s}\in \text{SO}(3)$ is $\epsilon$-close to $\text{id}\in \text{SO}(3)$.

\item For each decomposition curve $C$ of some surface $S_{ijk,s}$ (or $S_{ijk}$) and is not an inner cuff of any hamster wheel, the corresponding good curve has complex length $2R'+\xi_C$, for some complex number $\xi_C$ with $|\xi_C|<2\epsilon$ (the condition of $(R',\epsilon)$-good curves).

\item For each hamster wheel $H$ in some $S_{ijk,s}$ or $S_{ijk}$, it has $R'$ rungs (common perpendicular segments of its two outer cuffs) $r_{H,1},\cdots, r_{H,R'}$, and these rungs divide both outer cuffs $c,c'$ to $R'$ geodesic segments $s_{H,1},\cdots,s_{H,R'}$ and $s_{H,1}',\cdots,s_{H,R'}'$ respectively. Then for any $i=1,\cdots,R'$, the complex distance between $c$ and $c'$ along $r_{H,i}$ is $R'-2\log{\sinh{1}}+\mu_{H,i}$ with $|\mu_{H,i}|<\frac{\epsilon}{R'}$ (equation (2.9.1) of \cite{KW}); and for any $i=1,\cdots,R'-1$, the complex distance between $r_{H,i}$ and $r_{H,i+1}$ along $s_{H,i}$ and $s_{H,i}'$ are $2+\nu_{H,i}$ and $2+\nu_{H,i}'$ respectively, with $|\nu_{H,i}|,|\nu_{H,i}'|<\frac{\epsilon}{R'}$ (equation (2.9.3) of \cite{KW}).

\item For each decomposition curve $C$ of some surface $S_{ijk,s}$ (or $S_{ijk}$), the feet of its two adjacent good components differ by $1+\pi i+\eta_C$, for some complex number $\eta_C$ with $|\eta_C|<100\epsilon$ (the $(R,\epsilon)$-well-matched condition in Section 2.10 of \cite{KW}) and $|\eta_C|<\frac{\epsilon}{R'}$ if formal feet are defined on both sides of $C$. Here if $C$ is contained in $\partial S_{ijk,s}$, the foot from the three-cornered annulus (or two-cornered annulus) $A_{ijk,s}$ is the foot of the shortest geodesic segment from $\gamma_{ijk,s}$ to a preferred vertex $v_{i,s}$, as in Construction \ref{constructZ2} (3) (4) (5).
\end{enumerate}
\end{parameter}

So we have parameters 
$$A_{ij,s},A_{ijk,s},B_{i,s}\in \text{SO}(3),$$ 
$$\lambda_{ij,s},\lambda_{ijk,s},\xi_C,\eta_C, \mu_{H,i}, \nu_{H,i}, \nu_{H,i}'\in \mathbb{C},\ \ \tau_{ij,s}\in \mathbb{R}$$ associated to $j_1:Z^3\looparrowright M$, and these parameters are very small with respect to metrics of $\text{SO}(3)$, $\mathbb{C}$ and $\mathbb{R}$, respectively.

Note that the data in Parameter \ref{parameterj1} (5) and (6) determine shapes of all $(R',\epsilon)$-good components, as in the discussion after Parameter 5.4 of \cite{Sun5}.

For any $t\in[0,1]$, we take parameters $$tA_{ij,s}, tA_{ijk,s}, tB_{i,s}\in \text{SO}(3),$$ 
$$t\lambda_{ij,s},t\lambda_{ijk,s},t\xi_C,t\eta_C, t\mu_{H,i}, t\nu_{H,i}, t\nu_{H,i}'\in \mathbb{C},\ \ t\tau_{ij,s}\in \mathbb{R}.$$ Here for any $A\in\text{SO}(3)$ close to $\text{id}$, $tA$ denotes the image of $t\in [0,1]$ under the shortest geodesic $[0,1]\to \text{SO}(3)$ from $\text{id}$ to $A$. These parameters give rise to a map $\tilde{j}_t:\tilde{Z}^3\to \mathbb{H}^3$ that is equivariant with respect to a representation $\rho_t:\pi_1(Z^3)\to \text{Isom}_+(\mathbb{H}^3)$, and $\tilde{j}_t$ can be defined by a developing argument as following.

We use $Z'$ to denote the subcomplex of $Z^3$ consisting of following pieces:
\begin{itemize}
\item $Z^{(1)}\subset Z$ (as in Construction \ref{constructZ1}),
\item all triangles $\Delta_{ijk,s}^t\subset Z$ corresponding to triangles $\Delta_{ijk}\subset \partial N_0$, 
\item all ideal edges $e_{i\infty,s}^{Z,t}$, ideal triangles $\Delta_{ij,\infty,s}^{Z,t}$ and ideal tetrahedra $T_{ijk\infty,s}^{Z,t}$.
\end{itemize}
We define $Z''$ to be the union of $Z'$ and all decomposition curves and boundary components of $S_{ijk,s}$ and $S_{ijk}$. The inclusion $Z''\hookrightarrow Z^3$ is $\pi_1$-injective on each component. 
We further define $Z'''$ to be the union of $Z''$ and following pieces:
\begin{itemize}
\item For each two-cornered or three-cornered annulus $A_{ijk,s}^t$ in $Z\subset Z^3$, take an edge from a vertex $v_{i,s}^t$ to the corresponding good curve $\gamma_{ijk,s}^t$
\item For each pair of pants in $Z\subset Z^3$, take its three seams.
\item For each hamster wheel in $Z\subset Z^3$, take all of its short seams between adjacent inner cuffs, and $2R'$ seams from two outer cuffs to all inner cuffs.
\end{itemize}

Let $\pi:\tilde{Z}^3\to Z^3$ be the universal cover of $Z^3$, then each component of $\pi^{-1}(Z')=\tilde{Z}'\subset \tilde{Z}^3$ is the universal cover of a component of $Z'$. We also use $\pi_M:\mathbb{H}^3\to M$ to denote the universal cover of $M$. The construction of $\tilde{j}_t:\tilde{Z}^3\to \mathbb{H}^3$ is given in the following.

\begin{construction}\label{constructjt}
We will first define $\tilde{j}_t:\tilde{Z}^3\to \mathbb{H}^3$ on $\tilde{Z}''=\pi^{-1}(Z'')\subset \tilde{Z}^3$, by the following steps.
\begin{enumerate}
\item We start with a vertex $\tilde{v}_{i,1}^1\in \tilde{Z}^3$ such that $\pi(\tilde{v}_{i,1}^1)=v_{i,1}^1$, and a point $p\in \mathbb{H}^3$ such that $\pi_M(p)=j(v_{i,1}^1)\in M$. Then we have an isometry of tangent spaces $(d\pi_M)_p:T_p\mathbb{H}^3\to T_{\pi_M(p)}(M)$, and we define $\tilde{j}_t(\tilde{v}_{i,1}^1)=p$. 

\item Let $\tilde{e}_{ij,1}^{Z,1}$ be an edge from $\tilde{v}_{i,1}^1$ to another vertex $\tilde{v}_{j,1}^1$ such that it projects to $e_{ij,1}^{Z,1}\subset Z^1\setminus \partial_p Z$. We map $\tilde{e}_{ij,1}^{Z,1}$ to a geodesic segment in $\mathbb{H}^3$ of length $R+t\text{Re}(\lambda_{ij,1})$ from $p$ to some $q\in \mathbb{H}^3$, such that its tangent vector at $p$ is the tangent vector of 
$$\tilde{{\bf F}}_{ijk,1}^M(t):=(d\pi_M)_p^{-1}({\bf F}_{ijk,1}^M\cdot tA_{ij,1}).$$ 
We parallel transport $-\tilde{{\bf F}}_{ijk,1}^M(t)(t\text{Im}\lambda_{ij,1})\cdot(tA_{ji,1})^{-1}$ along this geodesic segment to get a frame $\tilde{{\bf F}}_{jik,1}^M(t)\in \text{SO}_q(\mathbb{H}^3)$. Here $\tilde{{\bf F}}_{ijk,1}^M(t)(t\text{Im}\lambda_{ij,1})$ denotes the frame rotation of $\tilde{{\bf F}}_{ijk,1}^M(t)$ by angle $t\text{Im}\lambda_{ij,1}$. Then we take an isometry $T_q\mathbb{H}^3\to T_{j(v_{j,1}^1)}M$ that identifies $\tilde{{\bf F}}_{jik,1}^M(t)$ with ${\bf F}_{jik,1}^M\in \text{SO}_{j(v_{j,1}^1)}(M)$. Under this identification, we can further define the map $\tilde{j}_t$ on edges adjacent to $\tilde{v}_{j,1}^1$ that do not project to $\partial_p Z$, as in item (1).

\item If $\tilde{v}_{i,1}^1$ corresponds to a vertex $n_i\in \partial N_0$, then $v_{i,1}^1$ is contained in the torus $T_1^1\subset \partial_p Z$ and let $\tilde{T}_1^1$ be the component of $\pi^{-1}(T_1^1)\subset \tilde{Z}^3$ containing $\tilde{v}_{i,1}^1$, then we do the following construction. Let $\tilde{e}_{ij,1}^{Z,1}\subset \tilde{Z}^{(1)}$ be the edge of $\tilde{Z}^{(1)}$ adjacent to $\tilde{v}_{i,1}^1$ corresponding to the preferred edge $e_{ij}\in \partial N_0$ as in Parameter \ref{parameterj1} (4), and let $n_k$ be the vertex of $N^{(0)}$ as in Figure \ref{figure2}. The geodesic ray starting from $p=\tilde{j}_t(\tilde{v}_{i,1}^1)$ and tangent with the normal vector of 
$$(d\pi_M)_p^{-1}\big({\bf F}_{ijk,1}^M\cdot (X\cdot tB_{i,1})^{-1}\big)$$ 
gives an ideal point $b\in \partial \mathbb{H}^3$, and it determines a horoplane $P$ going through $p$. Then we take an Euclidean geodesic segment in $P$ tangent to the tangent vector of $(d\pi_M)_p^{-1}\big({\bf F}_{ijk,1}^M\cdot (X\cdot tB_{i,1})^{-1}\big)$ with length $(1+t\tau_{ij,1})R$ and map $\tilde{v}_{j,1}^1$ to its endpoint. The parameters $\tau_{st,1}$ (in Parameter \ref{parameterj1} (3)) inductively give triangles in $P$ with edge length $(1+t\tau_{st,1})R$. In this way, we get a map $\tilde{i}_t:\tilde{T}_1^1\to P$. 

\item For $\tilde{i}_t:\tilde{T}_1^1\to P$ defined in (3), we homotopy each edge in $\tilde{T}_1^1$ to a geodesic segment in $\mathbb{H}^3$ (relative to endpoints) and each triangle in $\tilde{T}_1^1$ to the corresponding totally geodesic triangle, to get the desired map $\tilde{j}_t|:\tilde{T}_1^1\to \mathbb{H}^3$.
For each ideal edge $e_{i\infty,1}^{Z,1}$, ideal triangle $\Delta_{ij\infty,1}^{Z,1}$, ideal tetrahedron $T_{ij\infty,1}^{Z,1}$ in $\tilde{Z}^3$ adjacent to $\tilde{T}_1^1$, we map them to the geodesic ray, the ideal geodesic triangle and the ideal hyperbolic tetrahedron determined by the ideal point $b\in \partial \mathbb{H}^3$ and $\tilde{j}_t:\tilde{T}_1^1\to \mathbb{H}^3$, respectively.

\item For any vertex $\tilde{v}_{i',1}^1\in \tilde{T}_1^1$ that corresponds to $v_{i',1}^1\in T_1^1$, there is an preferred edge $\tilde{e}_{i'j',1}^{Z,1}$ adjacent to $\tilde{v}_{i',1}^1\in \tilde{T}_1^1$. Then we get a frame $\tilde{\bf F}_{i',1}$ based at $\tilde{j}_t(\tilde{v}_{i',1}^1)$ such that its tangent vector is tangent to $\tilde{i}_t(\tilde{e}_{i'j',1}^{Z,1})$ and its normal vector points to $b\in \partial \mathbb{H}^3$. We take the isometry $T_{\tilde{j}_t(\tilde{v}_{i',1}^1)}\mathbb{H}^3\to T_{j(v_{i',1}^1)}M$ that maps $\tilde{\bf F}_{i',1}$ to ${\bf F}_{i'j'k',1}^M\cdot X\cdot (tB_{i',1})$. 
Then we construct $\tilde{j}_t$ on edges of $\tilde{Z}^{(1)}$ adjacent to $\tilde{v}_{i',1}^1\in \tilde{T}_1^1$ as in item (2).

\item By applying the construction in items (2), (3), (4), (5) inductively, we define $\tilde{j}_t$ on the component $\tilde{W}$ of $\tilde{Z}'\subset \tilde{Z}^3$ containing $\tilde{v}_{i,1}^1$, and let $W$ be the image of $\tilde{W}$ in $Z'$. Note that when $t=0$, for any triangle $\Delta_{ijk}$ (or bigon $B_{ijk}$) in $N_0\setminus \partial N_0$, the $\tilde{j}_0$-image of any component of $\pi^{-1}(e_{ij,1}^{Z,1}\cup e_{jk,1}^{Z,1}\cup e_{ki,1}^{Z,1})$ (or $\pi^{-1}(e_{ijk,1}^{Z,1}\cup e_{ji,1}^{Z,1})$) lies in a hyperbolic plane in $\mathbb{H}^3$. For $\Delta_{ijk}$, this claim follows from item (2) since the normal vector of $\tilde{F}_{ijk,1}^M(0)$ is parallel transported to the normal vector of $-\tilde{F}_{jik,1}^M(0)$, and the same holds if $i,j,k$ are permuted. For $B_{ijk}$, this claim follows from Remark \ref{coordinate} that the $t=0$ case is modeled by an equilateral tessellation of the horoplane.

When $t=1$, $\tilde{j}_1$ is exactly the restriction of $\tilde{j}:\tilde{Z}\to \mathbb{H}^3$ on $\tilde{W}$. Since $\tilde{j}_t|_{\tilde{W}}$ is defined by geometric parameters, it induces a representation $\rho_t^{\tilde{W}}:\pi_1(W)\to \text{Isom}_+(\mathbb{H}^3)$, such that $\tilde{j}_t|_{\tilde{W}}$ is $\rho_t^{\tilde{W}}$-equivariant.

\item Now we work on $\tilde{Z}'''$. For any line component $l\subset \tilde{Z}'''$ of the preimage of some $C\subset \partial S_{ijk,1}$ (or $\partial S_{ijk}$) that is adjacent to $\tilde{W}$, it corresponds to a bi-infinite concatenation of edges in $\tilde{W}$, and $\tilde{j}_t$ maps this concatenation to a bi-infinite quasi-geodesic in $\mathbb{H}^3$. We map $l$ to the bi-infinite geodesic that share end points with the above quasi-geodesic, equivariant under the $\pi_1(C)$-action (via $\rho_t^{\tilde{W}}$). For each edge $e$ of $\tilde{Z}'''$ adjacent to $l$ that is mapped to an edge contained in a three-cornered annulus (or two-cornered annulus) in $Z$, we map it to the shortest geodesic segment between the corresponding vertex in $\tilde{W}$ and the bi-infinite geodesic $\tilde{j}_t(l)$ in $\mathbb{H}^3$.

\item For each line component $l'\subset \tilde{Z}''$ in the preimage of an $(R',\epsilon)$-good curve $C'$ in $S_{ijk}$ and adjacent to $l$ in $\tilde{Z}''$, we map the seam $s$ between $l$ and $l'$ to the geodesic segment whose feet is the $(1+\pi i +t\eta_C)$-shift of the closest feet on $\tilde{j}_t(l)$ (arised from an edge in a three-cornered or a two-cornered annulus in item (7)). The complex length of $\tilde{j}_t(s)$ is determined by the parameters of this good component, which determines the $\tilde{j}_t$-image of $l'$, a bi-infinite geodesic in $\mathbb{H}^3$. Let $\pi_1(C')$ acts on $\mathbb{H}^3$ by translating along $\tilde{j}_t(l')$, with complex translation length $2R'+\xi_{C'}$, and we define $\tilde{j}_t$ on $l'$ to be $\pi_1(C')$-equivarient. 

Then we apply this process inductively to define $\tilde{j}_t$ on one component of $\tilde{Z}^3\setminus \tilde{Z}'$. For a hamster wheel, the complex lengths of its seams are determined by parameters $\mu_{H,i},\nu_{H,i},\nu_{H,i}'$ in Parameter \ref{parameterj1}  (6) and complex lengths of its outer cuffs.

\item Then we define $\tilde{j}_t$ on $\tilde{Z}'''$ inductively by the process in items (2) - (8). Again, since $\tilde{j}_t:\tilde{Z}'''\to \mathbb{H}^3$ is defined by geometric parameters, it induces a representation $\rho_t:\pi_1(Z^3)\to \text{Isom}(\mathbb{H}^3)$ and $\tilde{j}_t$ is $\rho_t$-equivariant. 
\end{enumerate}

At the end, since each component of $Z^3\setminus Z'''$ is topologically a disc, we can further triangulate $Z^3$ and map each new triangle in $Z^3$ to a geodesic triangle in $\mathbb{H}^3$. Then the map $\tilde{j}_t:\tilde{Z}'''\to \mathbb{H}^3$ extends to a $\rho_t$-equivariant map $\tilde{j}_t:\tilde{Z}^3\to \mathbb{H}^3$. Here $\tilde{j}_1:\tilde{Z}\to \mathbb{H}^3$ is the lifting of $j:Z^3\looparrowright M$ to universal covers, and $\tilde{j}_0$ maps each component of $\tilde{Z}^3\setminus \tilde{Z}'$ into a hyperbolic plane in $\mathbb{H}^3$.

\end{construction}

Note that $\tilde{j}_0:\tilde{Z}^3\to \mathbb{H}^3$ maps each component of $\tilde{Z}^3\setminus \tilde{Z}'$ to a totally geodesic subsurface of $\mathbb{H}^3$, maps each ideal tetrahedron in $\tilde{Z}^3$ to an ideal tetrahedron in $\mathbb{H}^3$, and the union of these pieces is $\tilde{Z}^3$. Then $\tilde{j}_0$ pulls back the hyperbolic metric on $\mathbb{H}^3$ to metrics on aforementioned pieces of $\tilde{Z}^3$, and further induces a path metric on $\tilde{Z}^3$. Each component of $\tilde{Z}^3\setminus \tilde{Z}^{(1)}$ is called a {\it $2$-dim piece} or a {\it $3$-dim piece} of $\tilde{Z}^3$, according to its dimension.

Now we prove a lemma that describes the coarse geometry of $\tilde{Z}^3$, with respect to the above metric. These estimates may not be optimal.

\begin{lemma}\label{new}
If $R>0$ is large enough, the following estimates hold  for $\tilde{Z}^3$, with respect to the above metric.
\begin{enumerate}
\item For any vertex $v\in \tilde{Z}^3$ and any edge $e\in \tilde{Z}^{(1)}$ not containing $v$, $d_{\tilde{Z}^3}(v,e)\geq\log{R}$. 

\item For any two distinct $3$-dim pieces of $\tilde{Z}^3$, their distance is at least $\frac{9}{10}\log{R}$.

\item For any two distinct edges $e_1,e_2\in \tilde{Z}^{(1)}$ that do not share a vertex and do not lie in the same $3$-dim piece of $\tilde{Z}^3$, we have $d_{\tilde{Z}^3}(e_1,e_2)\geq\frac{1}{3}\log{R}$.

\item For any two distinct vertices $v_1,v_2\in \tilde{Z}^3$, we have $d_{\tilde{Z}^3}(v_1,v_2)\geq \log{R}$.

\end{enumerate}
\end{lemma}

\begin{proof}

We take $R>0$ large enough such that 
$$\log{R}\geq \max{\{4, 2I(\pi-\phi_0)+10, -10\log{(\sin{\phi_0})}\}}.$$
At first, note that all edges of $\tilde{Z}^{(1)}$ that project to $\partial_p Z$ have length at least $2\log{R}$ (Construction \ref{constructj1} (1)), and all other edges have length at least $2R-1>2\log{R}$ (Construction \ref{constructj1} (2)(3)).

\begin{enumerate}
\item Let $e$ be an edge of $\tilde{Z}^{(1)}$, let $v$ be a vertex of $\tilde{Z}^3$ not contained in $e$, and let $\gamma$ be the shortest oriented path in $\tilde{Z}^3$ from $v$ to $e$.

If $v$ and $e$ do not lie in the closure of the same component of $\tilde{Z}^3\setminus \tilde{Z}^{(1)}$, then the interior of $\gamma$ intersects with $\tilde{Z}^{(1)}$ at finitely many edges, and let $e'$ be the first such edge. Then $e'$ and $v$ lie in the same component of the closure of $\tilde{Z}^3\setminus \tilde{Z}^{(1)}$ and $v$ is not a vertex of $e'$. So it suffices to assume that $v$ and $e$ lie in the closure $C$ of the same component of $\tilde{Z}^3\setminus \tilde{Z}^{(1)}$. 

\begin{enumerate}
\item We first suppose that $C$ is a $3$-dim piece. Let the ideal point corresponding to $C$ be $\infty$, then the projection of $C$ gives a tessellation of $\mathbb{R}^2$ consisting of equilateral triangles of length $R$. A computation in hyperbolic geometry gives $d_{\tilde{Z}^3}(v,e)>\log{R}$. So we can assume that $C$ is a $2$-dim piece.

\item If $\gamma$ intersects with the preimage of some surface piece $\tilde{S}_{ijk,s}$ or $\tilde{S}_{ijk}$, Remark \ref{combinatorialdistance} implies that $d_{\tilde{Z}^3}(v,e)=l(\gamma)\geq R>\log{R}$. 
So $\gamma$ does not intersect with any $\tilde{S}_{ijk,s}$ or $\tilde{S}_{ijk}$, and is homotopic to a concatenation of subsegments of edges in $\tilde{Z}^{(1)}$ relative to endpoints, as shown in Figure \ref{figure4}.

                       \begin{figure}
\centering
\label{new picture}
\def\svgwidth{0.6\textwidth}
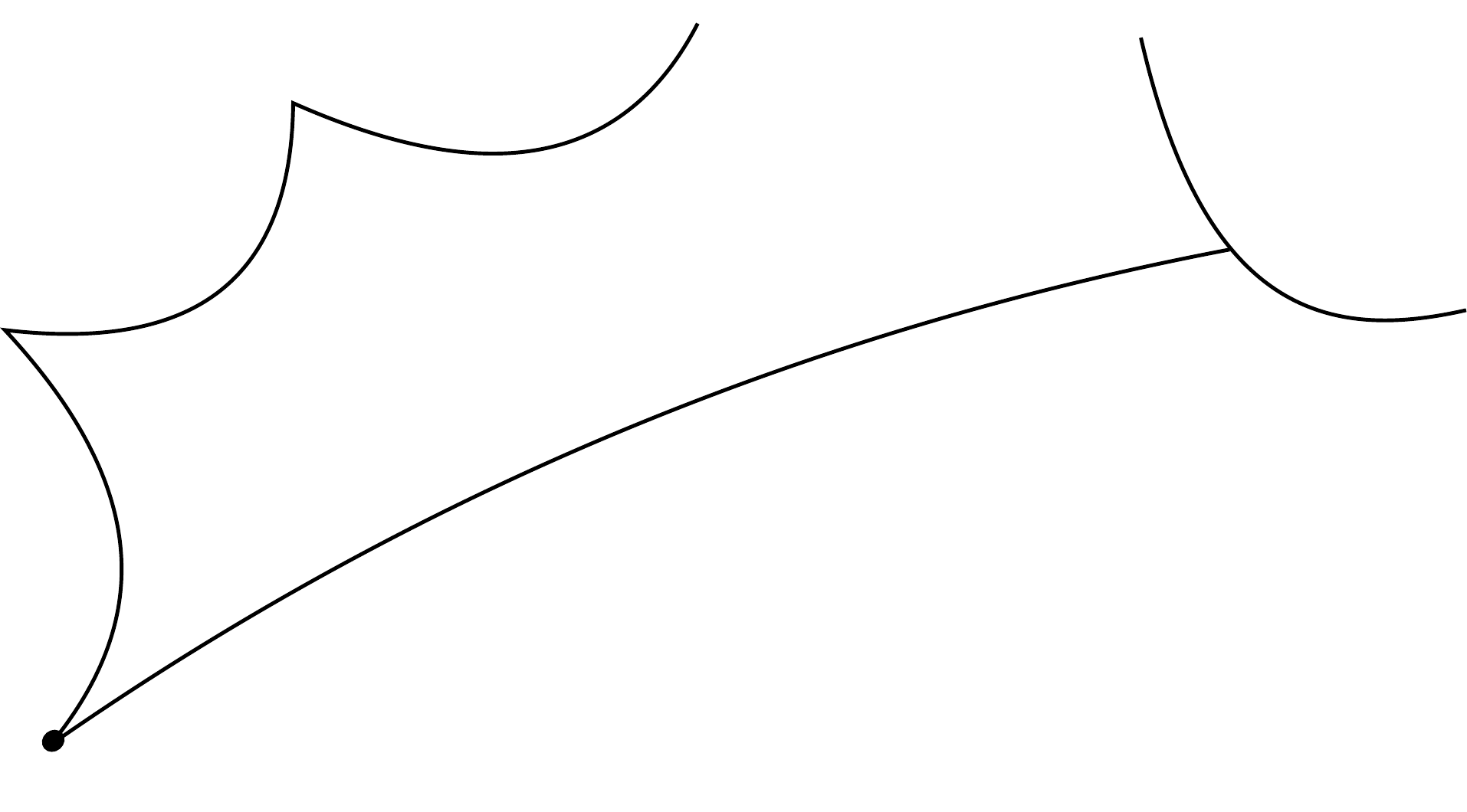
\caption{\ The shortest path in a piece of $\tilde{Z}^3\setminus \tilde{Z}^{(1)}$ from $v$ to $e$, without intersecting $\tilde{S}_{ijk,s}$ of $\tilde{S}_{ijk}$.}
\label{figure4}
\end{figure} 

\item There must be $k\geq 2$ edges in Figure \ref{figure4}, and all edges have length at least $2\log{R}$ except the last one. Since $\gamma$ lies in a $2$-dim piece, each inner angle between edges is at least $\phi_0$. If the last edge has length at least $\frac{1}{4}\log{R}$, then Lemma \ref{lengthphase} (1) implies 
\begin{align*}
&d_{\tilde{Z}^3}(v,e)=l(\gamma) \geq (k-1)\cdot 2\log{R}+\frac{1}{4}\log{R}-(k-1)(I(\pi-\phi_0)+1)\\
\geq &\ \log{R}+(k-1)(\frac{1}{2}\log{R}-I(\pi-\phi_0)-1) \geq \log{R}.
\end{align*}
If the last edge has length less than $\frac{1}{4}\log{R}$, then Lemma \ref{lengthphase} (1) implies
\begin{align*}
&d_{\tilde{Z}^3}(v,e)=l(\gamma) \geq (k-1)\cdot 2\log{R}-(k-2)(I(\pi-\phi_0)+1)-\frac{1}{4}\log{R}\\
\geq &\ \log{R}+(k-2)(2\log{R}-I(\pi-\phi_0)-1) \geq \log{R}.
\end{align*}
\end{enumerate}

\item Let $C_1,C_2$ be two distinct $3$-dim pieces of $\tilde{Z}^3$, then $C_1\cap C_2=\emptyset$ holds. Let $\gamma$ be the shortest path in $\tilde{Z}^3$ between $C_1$ and $C_2$. We call intersections of $\gamma$ with $3$-dim and $2$-dim pieces of $\tilde{Z}^3\setminus \tilde{Z}^{(1)}$ as $3$-dim subsegments and $2$-dim subsegments of $\gamma$.
We can assume that $\gamma$ does not have any $3$-dim subsegment, otherwise the proof can be reduced to a subsegment of $\gamma$. So we can assume that $\gamma$ only has $2$-dim subsegments.

\begin{enumerate}
\item If $\gamma$ intersects with some component of $\tilde{S}_{ijk,s}$ or $\tilde{S}_{ijk}$, then $d_{\tilde{Z}^3}(v,e)=l(\gamma)\geq R>\log{R}$ as in (1) (b). So we assume that $\gamma$ does not intersect with any such $\tilde{S}_{ijk,s}$ or $\tilde{S}_{ijk}$, and the intersection of $\gamma$ with $2$-dim pieces of $\tilde{Z}^3\setminus \tilde{Z}^{(1)}$ are as shown in Figure \ref{figure5} (a) and (b). 

\item If a $2$-dim subsegment of $\gamma$ is homotopic to a concatenation of $k\geq 3$ subsegments of edges of $\tilde{Z}^{(1)}$ relative to endpoints, as in Figure \ref{figure5} (a), then all such subsegments have lengths at least $2\log{R}$ except the first and last one. We divide into three cases. If both the first and last subsegments of edges have length at least $\frac{1}{4}\log{R}$, by Lemma \ref{lengthphase} (1), we have 
\begin{align*}
& l(\gamma)\geq 2\times \frac{1}{4}\log{R}+(k-2)2\log{R}-(k-1)(I(\pi-\phi_0)+1)\\
= &\ \frac{3}{2}\log{R}+(2k-5)(\log{R}-2I(\pi-\phi_0)-2)+(3k-9)(I(\pi-\phi_0)+1)\\
\geq &\ \log{R}.
\end{align*}
If exactly one of the first and last subsegment of edges has length at least $\frac{1}{4}\log{R}$, we have
\begin{align*}
& l(\gamma)\geq \frac{1}{4}\log{R}+(k-2)2\log{R}-(k-2)(I(\pi-\phi_0)+1)-\frac{1}{4}\log{R}\\
= &\ \frac{3}{2}\log{R}+(2k-\frac{11}{2})(\log{R}-2I(\pi-\phi_0)-2)+(3k-9)(I(\pi-\phi_0)+1)\\
\geq & \ \log{R}.
\end{align*}
If both the first and last subsegments of edges have length at most $\frac{1}{4}\log{R}$, we have
\begin{align*}
& l(\gamma)\geq (k-2)2\log{R}-(k-3)(I(\pi-\phi_0)+1)-2\cdot \frac{1}{4}\log{R}\\
= &\ \frac{5}{4}\log{R}+(2k-\frac{23}{4})(\log{R}-2I(\pi-\phi_0)-2)+(3k-\frac{17}{2})(I(\pi-\phi_0)+1)\\
\geq & \ \log{R}.
\end{align*}

\item By (2) (b) above, we can assume that only Figure \ref{figure5} (b) shows up. If all subsegments of edges in Figure \ref{figure5} (b) have length less than $\log{R}$, since all edges of $\tilde{Z}^{(1)}$ have length at least $2\log{R}$, the vertices in each occurrence of Figure \ref{figure5} (b) must be the same vertex, contradicting with the fact that $\gamma$ connects two distinct $3$-dim pieces. So some subsegment of edge in Figure \ref{figure5} (b) must have length at least $\log{R}$. Since the angle in Figure \ref{figure5} (b) is at least $\phi_0$, a computation in hyperbolic geometry gives 
$\sinh(l(\gamma))\geq \sinh(\log{R})\cdot \sin{\phi_0}$, thus $$l(\gamma)\geq \log{R}+\log{(\sin{\phi_0})}>\frac{9}{10}\log{R}.$$
\end{enumerate}

\item Let $e_1,e_2$ be two distinct edges of $\tilde{Z}^{(1)}$ that do not share a vertex and do not lie in the same $3$-dim piece of $\tilde{Z}^3$. Let $\gamma$ be the shortest path in $\tilde{Z}^3$ between $e_1$ and $e_2$.
If $\gamma$ intersects with a component of $\tilde{S}_{ijk,s}$ or $\tilde{S}_{ijk}$, then $l(\gamma)\geq \log{R}$ as in (1) (b) above. So we assume that $\gamma$ does not intersect with any such $\tilde{S}_{ijk,s}$ or $\tilde{S}_{ijk}$, and we consider  the intersection of $\gamma$ with components of $\tilde{Z}^3\setminus \tilde{Z}^{(1)}$, as the following cases.

\begin{enumerate}
\item If $\gamma$ has a $2$-dim subsegment as shown in Figure \ref{figure5} (a), with $k\geq 3$ edges in $\tilde{Z}^{(1)}$ show up in the picture. Then (2) (b) implies $l(\gamma)\geq \log{R}$ holds. So we can assume that all $2$-dim subsegments of $\gamma$ are as shown in Figure \ref{figure5} (b).

\item If some subsegment of edge in Figure \ref{figure5} (b) has length at least $\frac{2}{3}\log{R}$, then the argument as in (2) (c) implies $l(\gamma)>\frac{2}{3}\log{R}+\log{(\sin{\phi_0})}>\frac{1}{2}\log{R}$. So we can assume that, for any $2$-dim subsegment of $\gamma$ as in Figure \ref{figure5} (b), all subsegments of edges have length smaller than $\frac{2}{3}\log{R}$.

\item We call concatenations of adjacent $2$-dim subsegments of $\gamma$ as $2$-dim pieces of $\gamma$, and we also call $3$-dim subsegments of $\gamma$ as $3$-dim pieces. Note that by (3) (b), each $2$-dim piece of $\gamma$ lies in the link of a vertex of $\tilde{Z}^3$. In particular, $\gamma$ can not be a single $2$-dim piece, since $e_1$ and $e_2$ do not share a vertex of $\tilde{Z}^3$, thus $\gamma$ must contain at least one $3$-dim piece.

\item If $\gamma$ contains two $3$-dim pieces, since each $3$-dim piece of $\tilde{Z}^3$ is convex (as a subset of $\mathbb{H}^3$), these two $3$-dim pieces of $\gamma$ must be contained in two distinct $3$-dim pieces of $\tilde{Z}^3$. Then item (2) implies that $l(\gamma)\geq \frac{9}{10}\log{R}$ holds. So we can assume that $\gamma$ has at most one $3$-dim piece. 

\item So $\gamma$ is a concatenation of one $2$-dim piece, one $3$-dim piece and one $2$-dim piece.
It is possible that one $2$-dim piece of $\gamma$ may degenerate, but these $2$-dim pieces can not both degenerate, since $e_1$ and $e_2$ do not lie in the same $3$-dim piece of $\tilde{Z}^3$. 

If exactly one of the two $2$-dim pieces of $\gamma$ degenerates, we assume that the degenerated edge contains the terminal point of $\gamma$. The unique $2$-dim piece of $\gamma$ lies in the the $(\frac{2}{3}\log{R})$-neighborhood of a vertex $v$. Then $v$ is a vertex of $e_1$, and is contained in a $3$-dim piece $C$ of $\tilde{Z}^3\setminus \tilde{Z}^{(1)}$. Since $e_1$ and $e_2$ do not share a vertex, $e_2$ is contained in $C$ and does not have $v$ as its vertex. By (1), $d_{\tilde{Z}^3}(e_2,v)\geq \log{R}$ holds. So we have $l(\gamma)\geq \log{R}-\frac{2}{3}\log{R}=\frac{1}{3}\log{R}$.

If neither two $2$-dim pieces of $\gamma$ degenerate, then there are two vertices $v_1$ and $v_2$ of the same $3$-dim piece $C$ of $\tilde{Z}^3$, such that the two $2$-dim pieces of $\gamma$ lie in $(\frac{2}{3}\log{R})$-neighborhoods of $v_1$ and $v_2$ respectively. Since $e_1$ and $e_2$ do not share a common vertex, $v_1$ and $v_2$ are two distinct vertices of $C$. Since $d_{\tilde{Z}^3}(v_1,v_2)=d_C(v_1,v_2)\geq 2\log{R}$ holds, we have $l(\gamma)\geq 2\log{R}-2\cdot \frac{2}{3}\log{R}=\frac{2}{3}\log{R}$.
\end{enumerate}

\item Let $v_1,v_2$ be two distinct vertices of $\tilde{Z}^3$, and let $\gamma$ be the shortest path in $\tilde{Z}^3$ between $v_1$ and $v_2$. If $\gamma$ intersects with some component of $\tilde{S}_{ijk,s}$ or $\tilde{S}_{ijk}$, or if $\gamma$ intersects with some $2$-dim piece of $\tilde{Z}^3$ as in Figure \ref{figure5} (a) with $k\geq 3$, then (1)(b) and (2)(b) imply $l(\gamma)\geq \log{R}$ hold, respectively. So we assume that each $2$-dim subsegment of $\gamma$ is as shown in Figure \ref{figure5} (b). 

\begin{enumerate}
\item If the initial point $v_1$ of $\gamma$ is contained in a $2$-dim subsegment of $\gamma$, as in Figure \ref{figure5} (b), then the subsegment of edge containing $v_1$ has length at least $2\log{R}$. So as in (2)(c), we have $l(\gamma)\geq \log{R}$.

\item If the initial point $v_1$ of $\gamma$ is contained in a $3$-dim subsegment of $\gamma$, then the other end point of this $3$-dim subsegment is contained in an edge not containing $v_1$. Then by (1), we have $l(\gamma)\geq \log{R}$.
\end{enumerate}

\end{enumerate}

                                          \begin{figure}
\centering
\label{new picture}
\def\svgwidth{0.6\textwidth}
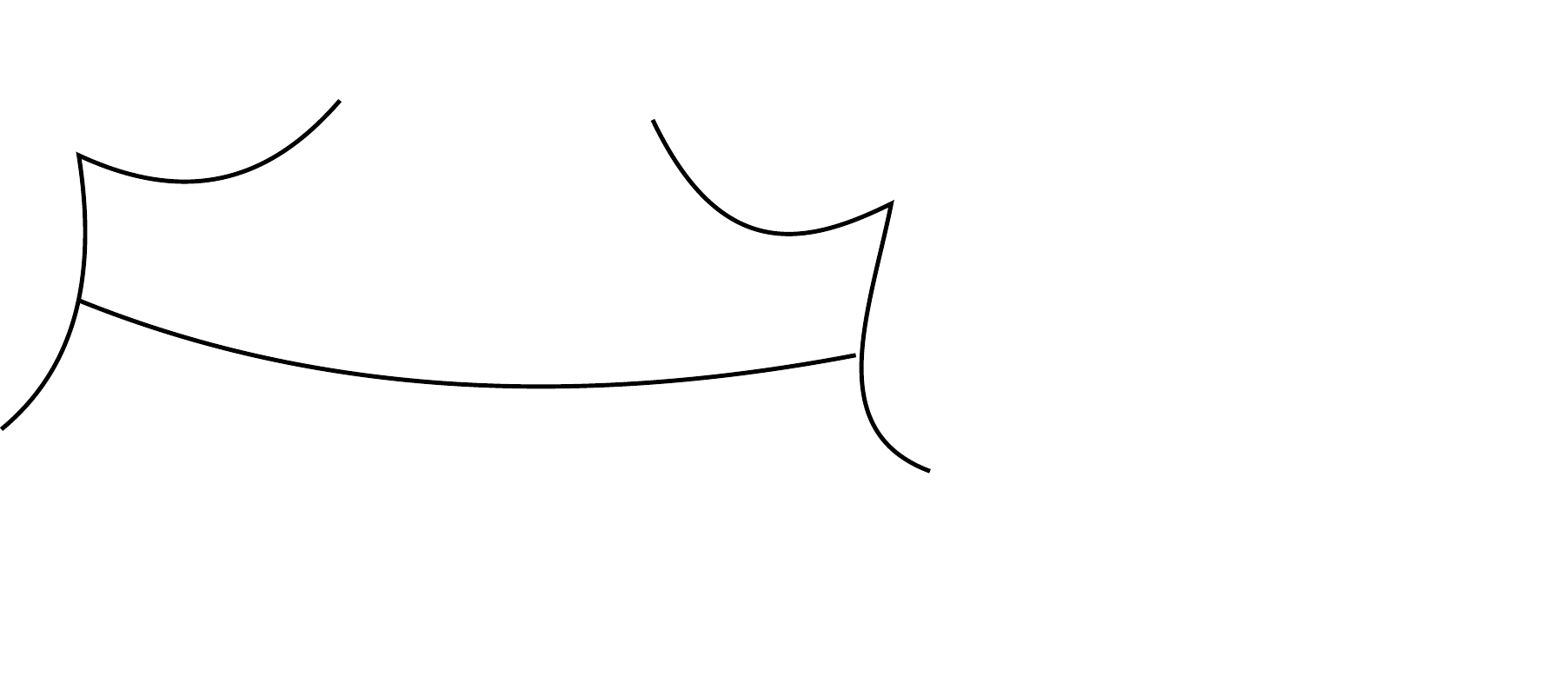
\caption{\ Intersection of $\gamma$ and components of $\tilde{Z}^3\setminus \tilde{Z}^{(1)}$.}
\label{figure5}
\end{figure}

\end{proof}

\bigskip

\subsection{Estimation on the ideal model of $Z$}\label{estimate2}

In this section, we prove the following result on the map $\tilde{j}_0:\tilde{Z}^3\to \mathbb{H}^3$ defined in the last section.

\begin{proposition}\label{model}
Given the metric on $\tilde{Z}^3$, for large enough $R>0$, the following statements hold.
\begin{enumerate} 
\item The map $\tilde{j}_0:\tilde{Z}^3\to \mathbb{H}^3$ is a quasi-isometric embedding. 
\item The representation $\rho_0:\pi_1(Z^3)\to \text{Isom}_+(\mathbb{H}^3)$ is injective.
\item The map $\tilde{j}_0:\tilde{Z}^3\to \mathbb{H}^3$ is an embedding.
\item The convex core of $\mathbb{H}^3/\rho_0(\pi_1(Z^3))$ is homeomorphic to the $3$-manifold $\mathcal{Z}\setminus \partial_p\mathcal{Z}$ in Theorem \ref{pi1injectivity}, as oriented manifolds.
\end{enumerate}
\end{proposition}

The idea of the proof of Proposition \ref{model} is similar to proofs of corresponding results in \cite{Sun2} and \cite{Sun5}, but the actual proof is more complicated, since the construction of the $3$-complex $Z^3$ is more involved. In particular, we will have a more complicated definition of the {\it modified sequence}.

We take large enough $R$ so that there is an $L$ satisfying the following inequality: 
\begin{align}\label{6.0}
\max{\{1000,2I(\pi-\phi_0)\}}\leq L\leq \frac{1}{320}\log{R}.
\end{align} 
Here $\phi_0$ is defined in Notation \ref{notation} (3).

For any two points $x,y\in \tilde{Z}''\subset \tilde{Z}^3$, we will estimate $d_{\mathbb{H}^3}(\tilde{j}_0(x),\tilde{j}_0(y))$. %Here we assume that $x$ is either a vertex of $\tilde{Z}^3$ or does not lie in the $L+\frac{1}{160}\log{R}$ of any vertex of $\tilde{Z}^3$, and the same holds for $y$.
Let $\gamma$ be the shortest path in $\tilde{Z}^3$ from $x$ to $y$, and we will assume that $\gamma$ intersects with $\tilde{Z}^{(1)}$ nontrivially in its interior. Let $x_1,x_2,\cdots,x_n$ be the intersection points of $\gamma\cap \tilde{Z}^{(1)}$ that follow the orientation of $\gamma$. If $\gamma$ contains a subsegment of an edge in $\tilde{Z}^{(1)}$, we only record the endpoints of this subsegment. This sequence $x_1,x_2,\cdots,x_n$ is called the {\it intersection sequence} of $\gamma$, and let $x_0=x$ and $x_{n+1}=y$. We use $\gamma_i=\overline{x_ix_{i+1}}$ to denote the subsegment of $\gamma$ (and the geodesic segment) from $x_i$ to $x_{i+1}$. Such a $\gamma_i$ is called a {\it $3$-dim piece} of $\gamma$ if it is contained in the union of ideal tetrahedra of $\tilde{Z}^3$.

Now we make the following assumption on $\gamma$.
\begin{assumption}\label{assumption}
For the segment $\overline{xx_1}$, we assume the following hold.
\begin{enumerate}
\item Either $x$ is a vertex of $\tilde{Z}^3$, or $x$ lies on an edge of $\tilde{Z}^{(1)}$ and its distance to any vertex of $\tilde{Z}^3$ is at least $L$, or the distance between $x$ and any vertex of $\tilde{Z}^3$ is at least $L+\frac{1}{160}\log{R}$.
\item If $\overline{xx_1}$ is not a $3$-dim piece, then $d(x,x_1)\geq \frac{1}{160}\log{R}$.
\end{enumerate}
We also assume the same condition holds for $\overline{x_ny}$.
\end{assumption}

If $\gamma$ satisfies Assumption \ref{assumption}, we construction the {\it modified sequence} of $\gamma$ as following.
 \begin{construction}\label{modifiedsequence} 
 For any $i=1,\cdots,n$, we do the following modification. 
\begin{enumerate}
\item If neither $\gamma_{i-1}$ nor $\gamma_i$ are $3$-dim pieces and both of following hold:
\begin{itemize}
\item $d(x_i,v_i)<L$ holds for some vertex $v_i\ne x_i$ (then $x_i$ and $v_i$ lie on the same edge of $\tilde{Z}^3$ and $v_i$ is unique, by Lemma \ref{new} (1)(4)).
\item $d(x_{i-1},x_i)<\frac{1}{160}\log{R}$ or $d(x_i,x_{i+1})<\frac{1}{160}\log{R}$.
\end{itemize}
Then we replace $x_i$ by $v_i$.

\item If $\gamma_{i-1}$ or $\gamma_i$ is a $3$-dim piece, then exactly one of them is. Without loss of generality, we assume $\gamma_i$ (from $x_i$ to $x_{i+1}$) is a $3$-dim piece, and $\gamma_{i-1}$ is not. Then we do the following two steps.
\begin{enumerate}
\item If $d(x_i,v_i)<L$ for some vertex $v_i\ne x_i$ and $d(x_{i-1},x_i)<\frac{1}{160}\log{R}$, we replace $x_i$ by $v_i$. Similarly, if $d(x_{i+1},v_{i+1})<L$ for some vertex $v_{i+1}\ne x_{i+1}$, then $x_{i+1}\ne y$ and $x_{i+2}$ exists by Assumption \ref{assumption} (1). In this case, if $d(x_{i+1},x_{i+2})<\frac{1}{160}\log{R}$, we replace $x_{i+1}$ by $v_{i+1}$.
\item If the modification in Step I is done for $x_i$ but not for $x_{i+1}$ and $d(x_{i+1},v_i)<L$, then by Lemma \ref{new} (1), $v_i$ is contained in an edge containing $x_{i+1}$, and we replace $x_{i+1}$ by $v_i$. We do a similar process if the modification in Step I is done for $x_{i+1}$ but not for $x_i$.
\end{enumerate}
\end{enumerate}
\end{construction}

%Note that all the modification process replaces some $x_i$ by some vertex within distance $L$.
Note that during the modification process, if some $x_i$ is replaced by $v_i$, then they lie on the same edge of $\tilde{Z}^{(1)}$ and we must have $d(x_i,v_i)< L$.

After doing the above process, by replacing certain $x_i$ by corresponding vertices of $\tilde{Z}^3$, we get the {\it modified sequence} $x=y_0,y_1,\cdots,y_m,y_{m+1}=y$ of $\gamma$. Note that a few points in the intersection sequence might be replaced by the same point $y_i$. Then $y_i$ and $y_{i+1}$ lie in the same component of $\tilde{Z}^3\setminus \tilde{Z}^{(1)}$. Each component of $\tilde{Z}^3\setminus \tilde{Z}^{(1)}$ is either a simply connected convex hyperbolic surface with piecewise geodesic boundary, or a convex hyperbolic $3$-manifold obtained by an infinite union of ideal tetrahedra. Let $\gamma_i'=\overline{y_iy_{i+1}}$ be the shortest path in the piece of $\tilde{Z}^3\setminus \tilde{Z}^{(1)}$ from $y_i$ to $y_{i+1}$, and we say that $\gamma_i'$ is a $2$-dim or a $3$-dim piece if it lies in a $2$-dim or a $3$-dim piece of $\tilde{Z}^3\setminus \tilde{Z}^{(1)}$, respectively.
We define the {\it modified path} $\gamma'$ of $\gamma$ to be the concatenation of $\gamma_0',\gamma_1',\cdots,\gamma_m'$. 

Any $y_i$ that also lies in the intersection sequence of $\gamma$ is called {\it an unmodified point}, otherwise it is called {\it a modified point}. Any $\gamma_i'$ that is also a piece of $\gamma$ is called {\it an unmodified piece} of $\gamma'$.

We will prove a few properties of the modified path in the following.

\begin{lemma}\label{modifiedsequencelength}
If $\gamma$ satisfies Assumption \ref{assumption}, then for any $i=0, 1,\cdots,m$, $\gamma_i'$ is either an unmodified $3$-dim piece or satisfies $l(\gamma_i')\geq L$.
\end{lemma}

\begin{proof}

{\bf Case I}. If both $y_i$ and $y_{i+1}$ are modified points, by Lemma \ref{new} (4), we have $l(\gamma_i')=d(y_i,y_{i+1})\geq \log{R}\geq L$.

{\bf Case II}. If both $y_i$ and $y_{i+1}$ are unmodifed points, we divide into following sub-cases.
\begin{itemize}
\item If $\gamma_i'$ is a $3$-dim piece, then it is unmodified and the result trivially holds. So we assume that $\gamma_i'$ is not a $3$-dim piece in the following.

\item If $y_i=x$ of $y_{i+1}=y$, we assume $y_i=x$ without loss of generality. Then Assumption \ref{assumption} (2) implies $l(\gamma_i')=d(y_i,y_{i+1})\geq\frac{1}{160}\log{R}>L$. So we assume that $y_i$ and $y_{i+1}$ are not $x$ and $y$ respectively.

\item If $\gamma_i'$ intersects with any component of $\tilde{S}_{ijk,s}$ or $\tilde{S}_{ijk}$, then $l(\gamma_i')\geq R>L$, by Remark \ref{combinatorialdistance}.

\item If $\gamma_i'$ does not intersect with $\tilde{S}_{ijk,s}$ or $\tilde{S}_{ijk}$, we have two cases. If $l(\gamma_i')\geq \frac{1}{160}\log{R}$, then the result trivially holds. If $l(\gamma_i')<\frac{1}{160}\log{R}$, then by Lemma \ref{new} (3), $y_i$ and $y_{i+1}$ must lie on two adjacent edges of $\tilde{Z}^{(1)}$, with a common vertex $v$. Then we must have $d(y_i,v),d(y_{i+1},v)\geq L$, otherwise $y_i$ or $y_{i+1}$ should be modified in Construction \ref{modifiedsequence} (1) or (2) (a). Since we have $\angle y_ivy_{i+1}\geq \phi_0$ (Notation \ref{notation} (3)), we get $$d(y_i,y_{i+1})\geq d(y_i,v)+d(y_{i+1},v)-I(\pi-\phi_0)\geq 2L-I(\pi-\phi_0)\geq L.$$
\end{itemize}

{\bf Case III}. If exactly one of $y_i$ and $y_{i+1}$ is modified, we assume that $y_i$ is unmodified and $y_{i+1}$ is modified. Then there are $x_i$ and $x_{i+1}$ in the intersection sequence of $\gamma$, such that $x_i=y_i$ and $x_{i+1}$ is modifed to $y_{i+1}$ (in Construction \ref{modifiedsequence}).  If $y_i=x_i=x$, then since $y_{i+1}$ is a vertex of $\tilde{Z}^3$ and distinct from $y_i$, Assumption \ref{assumption} (1) implies $d(y_i,y_{i+1})\geq L$. So we assume that $y_i\ne x$ in the following, thus it lies on an edge of $\tilde{Z}^{(1)}$.

If $y_i$ and $y_{i+1}$ do not lie on the same edge of $\tilde{Z}^{(1)}$, since $y_{i+1}$ is a vertex of $\tilde{Z}^{(1)}$, Lemma \ref{new} (1) implies $d(y_i,y_{i+1})\geq \log{R}>L$. 

So $y_i$ and $y_{i+1}$ lie on the same edge, and the picture is shown in Figure \ref{figure6}. Suppose that $d(y_i,y_{i+1})<L$ holds. Since $x_{i+1}$ is modified to $y_{i+1}$, we have $d(x_{i+1},y_{i+1})<L$. Since $x_i=y_i$, we have $d(x_i,x_{i+1})<2L<\frac{1}{160}\log{R}$. If $\overline{x_ix_{i+1}}$ is not a $3$-dim piece, then $x_i$ should be modified, according to Construction \ref{modifiedsequence} (1) or (2)(a), and we get a contradiction. If $\overline{x_ix_{i+1}}$ is a $3$-dim piece, then the fact that $x_{i+1}$ is modified implies that $x_i$ should be modified, according to Construction \ref{modifiedsequence} (2)(b), which is impossible. So we must have $d(y_i,y_{i+1})\geq L$.

The proof is done.
\end{proof}

\begin{figure}
\centering
\label{new picture}
\def\svgwidth{.4\textwidth}
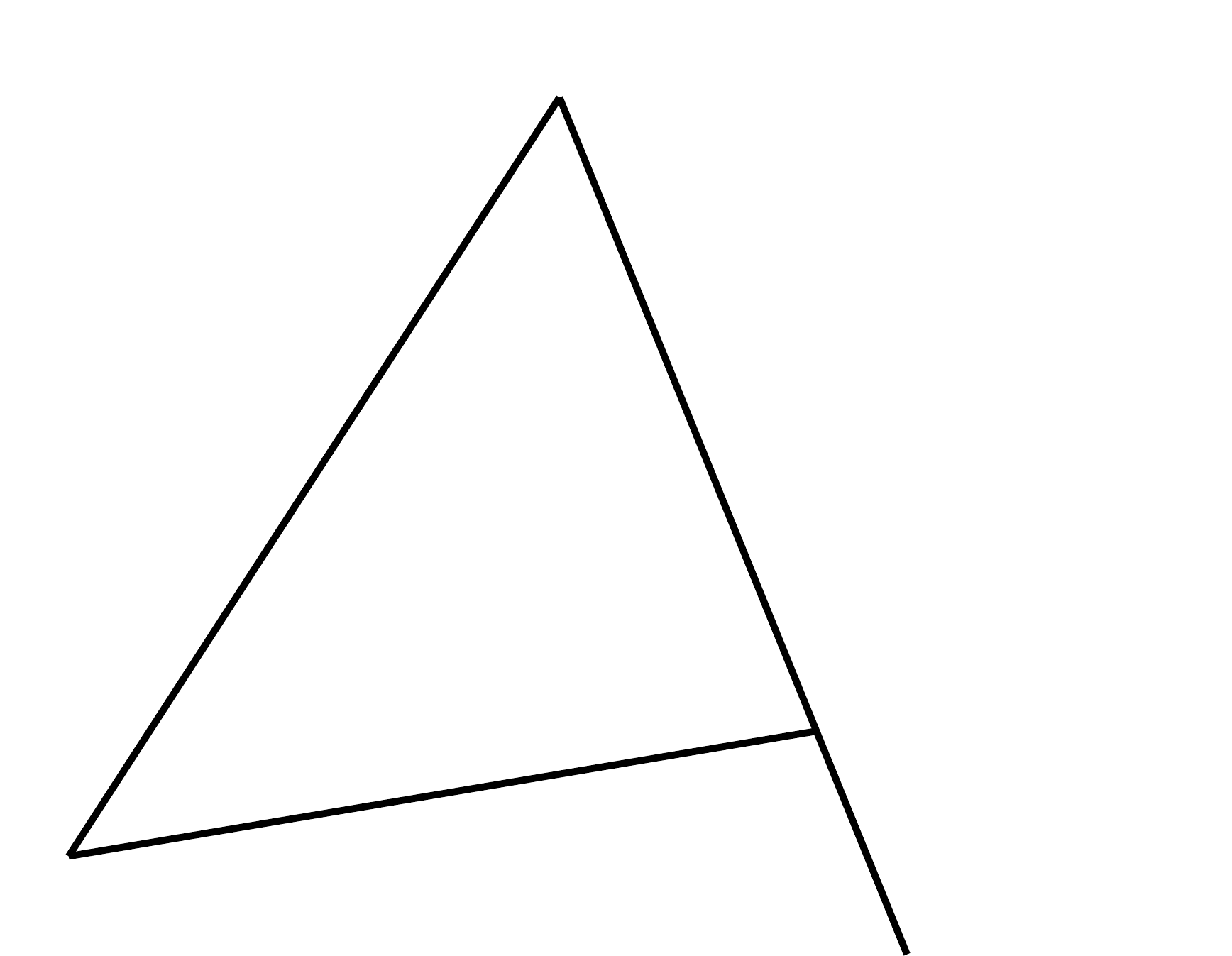
\caption{\ The position of $x_i=y_i, x_{i+1}$ and $y_{i+1}$.}
\label{figure6}
\end{figure}

The next job is to estimate the angle $\angle{y_{i-1}y_iy_{i+1}}$. To obtain this estimate, we first prove the following lemma.

\begin{lemma}\label{onemodification}
We suppose that $\gamma$ satisfies Assumption \ref{assumption}. Let $x_i,x_{i+1}$ be two consecutive points in the intersection sequence of $\gamma$ with $x_i\ne x$, such that when producing the modified sequence, $x_i$ is not modified (thus $y_i=x_i$) and $x_{i+1}$ is replaced by $y_{i+1}$, then the following hold.
\begin{enumerate}
\item If $\overline{x_ix_{i+1}}$ is not a $3$-dim piece, then $\angle x_{i+1}y_iy_{i+1}\leq 2e^{-\frac{L}{2}}$.
\item If $\overline{x_ix_{i+1}}$ is a $3$-dim piece, while $x_i$ and $y_{i+1}$ do not lie in the same edge of $\tilde{Z}^{(1)}$, then $\angle x_{i+1}y_iy_{i+1}\leq 2e^{-\frac{L}{2}}$.
\item If $\overline{x_ix_{i+1}}$ is a $3$-dim piece, while $x_i$ and $y_{i+1}$ lie in the same edge of $\tilde{Z}^{(1)}$, then $\angle x_{i+1}y_iy_{i+1}< \frac{\pi}{2}$.
\end{enumerate}
\end{lemma}

\begin{proof}
Since $x_{i+1}$ is replaced by $y_{i+1}$, by Construction \ref{modifiedsequence}, we have $d(x_{i+1},y_{i+1})<L$. Since $\overline{y_iy_{i+1}}$ is a piece of the modified sequence and is not an unmodified $3$-dim piece, Lemma \ref{modifiedsequencelength} implies $d(x_i,y_{i+1})=d(y_i,y_{i+1})\geq L$. 

{\bf Case I}. Suppose that $y_i=x_i$ does not lie in any edge of $\tilde{Z}^{(1)}$ containing $y_{i+1}$. Since $y_{i+1}$ is a vertex of $\tilde{Z}^{(1)}$, by Lemma \ref{new} (1),  we have $d(y_i,y_{i+1})\geq \log{R}$. Then we get 
 \begin{align}\label{6.1}
\angle{x_{i+1}y_iy_{i+1}}\leq 2 \sin{\angle{x_{i+1}y_iy_{i+1}}}\leq 2\frac{\sinh{d(y_{i+1},x_{i+1})}}{\sinh{d(y_i,y_{i+1})}}\leq 2e^{-\frac{L}{2}}.
 \end{align} So the proof of (2) is done.
 
{\bf Case II}. Suppose that $y_i=x_i$ lies in an edge of $\tilde{Z}^{(1)}$ containing $y_{i+1}$, the picture is shown in Figure \ref{figure6}. 
\begin{enumerate}
\item[(a)] Suppose that $\overline{x_ix_{i+1}}$ is not a $3$-dim piece. Then the triangle in Figure \ref{figure6} does not lie in a $3$-dim piece of $\tilde{Z}^3\setminus \tilde{Z}^{(1)}$, we have $\angle y_iy_{i+1}x_{i+1}\geq \phi_0$, and 
\begin{align*}
&d(y_i,x_{i+1})\geq d(y_i,y_{i+1})+d(y_{i+1},x_{i+1})-I(\pi-\phi_0)\\
\geq &\ d(y_{i+1},x_{i+1})+(L-I(\pi-\phi_0))\geq d(y_{i+1},x_{i+1})+\frac{L}{2}.
\end{align*}
So we have 
\begin{align}\label{6.2}
\angle{x_{i+1}y_iy_{i+1}}\leq 2 \sin{\angle{x_{i+1}y_iy_{i+1}}}\leq 2\frac{\sinh{d(x_{i+1},y_{i+1})}}{\sinh{d(y_i,x_{i+1})}}\leq 2e^{-\frac{L}{2}}.
\end{align}
Then equations (\ref{6.1}) and (\ref{6.2}) together imply (1).

\item[(b)] Suppose that $\overline{x_ix_{i+1}}$ is a $3$-dim piece. Then since $x_i$ is not modified, by Construction \ref{modifiedsequence} (2)(b), we have $d(y_i,y_{i+1})\geq L >d(x_{i+1},y_{i+1})$. So $\angle x_{i+1}x_iy_{i+1}<\frac{\pi}{2}$ holds, and the proof of (3) is done.

\end{enumerate}

\end{proof}

The next technical lemma estimates the angle $\angle{y_{i-1}y_iy_{i+1}}$ in the modified path.

\begin{lemma}\label{modifiedsequenceangle}
There exists $\eta_0>0$ (only depend on the geometry of the triangulation of $N$) such that for large enough $L>0$ (depending on $\eta_0$) and large enough $R>0$ (depending on $L$), the following statements hold for any $\gamma$ satisfying Assumption \ref{assumption}.
\begin{enumerate}
\item If neither $\gamma_{i-1}'$ nor $\gamma_i'$ are unmodified $3$-dim pieces, then $\angle y_{i-1}y_iy_{i+1}\geq \eta_0$.
\item If either $\gamma_{i-1}'$ or $\gamma_i'$ is an unmodified $3$-dim piece, then $\angle y_{i-1}y_iy_{i+1}\geq \frac{\pi}{2}+\eta_0$.
\end{enumerate}
\end{lemma}

Note that it is impossible that both $\gamma_{i-1}'$ and $\gamma_i'$ are unmodified $3$-dim pieces.

\begin{proof}

Recall that the constant $\phi_0>0$ defined in Notation \ref{notation} (3) is a lower bound of all inner angles of triangles in $N$ and all dihedral angles between intersecting totally geodesic triangles in $N$. We take a smaller $\phi_0$ if necessary, such that $\phi_0\in (0,\frac{\pi}{20})$.

%For any vertex $v^Z\in \tilde{Z}^3$ that corresponds to a vertex $v$ in $N\setminus \partial N$, the geometry of $\tilde{j}_0$ near $v^Z$ is same with the geometry near $v$. For any vertex $v^Z\in \tilde{Z}^3$ corresponding to a vertex $v$ in $\partial N$, the geometry of $\tilde{j}_0$ near $v^Z$ can be computed explicitly: For any two edges of $\tilde{Z}^{(1)}$ at $v^Z$ such that at least one of them is not a $3$-dim piece, the angle between their tangent vectors is at least $\frac{\pi}{10}$. Any two totally geodesic pieces of $\tilde{Z}^3\setminus \tilde{Z}^{(1)}$ adjacent to $v^Z$ has dihedral at least $\frac{\pi}{20}$. Now we retake a smaller $\phi_0$ with $\phi_0\in (0,\frac{\pi}{20})$.

For any vertex $v^Z$ of $\tilde{Z}^3$, we have a subspace  $S_{v^Z}\subset T^1_{\tilde{j}_0(v^Z)}=S^2$, consisting of all unit tangent vectors at $v^Z$ tangent to $\tilde{j}_0$-images of pieces of $\tilde{Z}^3\setminus \tilde{Z}^{(1)}$ that are adjacent to $v^Z$. If $v^Z$ corresponds to a vertex $v$ in $N_0\setminus \partial N_0$, $S_{v^Z}$ is a union of finitely many geodesic arcs in $S^2$, and it is determined by the geometry of $N$ near $v$.  If $v^Z$ corresponds to a vertex in $\partial N_0$, $S_{v^Z}$ is a union of finitely many geodesic arcs and one hexagon in $S^2$ (corresponding to a $3$-dim component of $\tilde{Z}^3\setminus \tilde{Z}^{(1)}$, see Figure \ref{figure1}). Moreover, in the second case, $S_{v^Z}$ also depends on $R$. The hexagon degenerates to a point when $R$ goes to infinity, and the limit geometry only depends on the geometry of $N$. The metric on $S^2$ induces a path metric on each $S_{v^Z}$. For fixed $R$, we only have finitely many isometric classes of $S_{v^Z}$, since the vertex set of $\tilde{Z}^{(1)}$ is a finite union of $\pi_1(Z)$-orbits.
So there exists $R_0>0$ and a constant $\theta_0>0$, such that for any $R>R_0$, any vertex $v^Z$ of $\tilde{Z}^3$ and any two vectors $v_1,v_2\in S_{v^Z}\subset S^2$, if their distance under the path metric of $S_{v^Z}$ is at least $\phi_0$, then the angle between them is at least $\theta_0$.

We take $\eta_0=\min{\{\frac{\phi_0}{2},\frac{\theta_0}{2}\}}$. We also take large $L$ and $R$ such that $L>2\log{\frac{8}{\eta_0}}, R>(\frac{8}{\eta_0})^{640}$ and equation (\ref{6.0}) holds.
Since the definition of the modified sequence is complicated, we need to run a case-by-case argument.

{\bf Case I}. We first assume that $y_i$ is an unmodified point.
       \begin{enumerate}
       \item Both $y_{i-1}$ and $y_{i+1}$ are unmodified points. Then the concatenation of $\overline{y_{i-1}y_i},     \overline{y_iy_{i+1}}$ is the shortest path in $\tilde{Z}^3$ from $y_{i-1}$ to $y_{i+1}$. 

               \begin{enumerate}
               \item We first suppose that neither of $\overline{y_{i-1}y_i}, \overline{y_iy_{i+1}}$ are $3$-dim pieces. Since the dihedral angle between corresponding totally geodesic subsurfaces in $\mathbb{H}^3$ is at least $\phi_0$ (Notation \ref{notation} (3)), we have 
               \begin{align}\label{6.3}
               \angle x_{i-1}x_ix_{i+1}=\angle{y_{i-1}y_iy_{i+1}}\geq \phi_0> \eta_0.
               \end{align}

                \item We suppose that one of $\overline{y_{i-1}y_i}, \overline{y_iy_{i+1}}$ is a $3$-dim piece. By the explicit vectors in Remark \ref{constructj1remark} (3), the dihedral angle between the boundary of an ideal tetrahedron and an adjacent geodesic subsurface (not the boundary of an ideal tetrahedron) in $\tilde{Z}^3$ is at least $\arccos{(-\frac{\sqrt{3}}{2\sqrt{7}})}>\frac{3}{5}\pi$. So we have 
                \begin{align}\label{6.4}
                \angle x_{i-1}x_ix_{i+1}=\angle y_{i-1}y_iy_{i+1}>\frac{3}{5}\pi>\frac{\pi}{2}+\eta_0.
                \end{align}
                Note that equations (\ref{6.3}) and (\ref{6.4}) will be repeatedly used in the remainder of this proof.
                \end{enumerate}

        \item Exactly one of $y_{i-1}$ and $y_{i+1}$ is a modified point, and we assume that $y_{i-1}$ is unmodified and $y_{i+1}$ is modified. Then we have $y_{i-1}=x_{i-1}$, $y_i=x_i$, and $y_{i+1}$ is replaced by $x_{i+1}$ as in Construction \ref{modifiedsequence}. In particular, we have $d(x_{i+1},y_{i+1})<L$ holds.

                \begin{enumerate}
                \item We first suppose that $\overline{x_ix_{i+1}}$ is not a $3$-dim piece. Then Lemma \ref{onemodification} (1) implies that $\angle{x_{i+1}y_iy_{i+1}}\leq 2e^{-\frac{L}{2}}<\frac{\phi_0}{2}$. By equations (\ref{6.3}) and (\ref{6.4}) respectively, $\angle{y_{i-1}y_i x_{i+1}}$ is at least $\phi_0$ or $\frac{3\pi}{5}$ in the two cases of this lemma. So we get that $\angle{y_{i-1}y_i y_{i+1}}$ is at least $\eta_0$ or $\frac{\pi}{2}+\eta_0$ in these two cases.
 
                \item Now we suppose that $\overline{x_ix_{i+1}}$ is a $3$-dim piece, then $\overline{x_{i-1}x_i}$ is not a $3$-dim piece. So neither $\overline{y_{i-1}y_i}$ nor $\overline{y_iy_{i+1}}$ are unmodified $3$-dim pieces, and $d(y_{i-1},y_i),d(y_i,y_{i+1})\geq L$ holds by Lemma \ref{modifiedsequencelength}.
                
                          \begin{enumerate}
                          \item If $y_i$ and $y_{i+1}$ do not lie on the same edge of $\tilde{Z}^{(1)}$, then Lemma \ref{onemodification} (2) implies $\angle{x_{i+1}y_iy_{i+1}}\leq 2e^{-\frac{L}{2}}$, and equation (\ref{6.3}) implies $\angle{y_{i-1}y_i y_{i+1}}\geq \eta_0$.
                          
                           \item If $y_i$ and $y_{i+1}$ lie on the same edge of $\tilde{Z}^{(1)}$, then Lemma \ref{onemodification} (3) implies $\angle{x_{i+1}y_iy_{i+1}}\leq \frac{\pi}{2}$. Since $\gamma$ is the shortest path in $\tilde{Z}^3$, equations (\ref{6.4}) implies 
                           $$\angle y_{i-1}y_i y_{i+1}=\angle x_{i-1}x_i y_{i+1}\geq \angle x_{i-1}x_ix_{i+1}-\angle x_{i+1}y_iy_{i+1}\geq \frac{3}{5}\pi- \frac{\pi}{2}>\eta_0.$$
                           \end{enumerate}
                 \end{enumerate}

         \item Both $y_{i-1}$ and $y_{i+1}$ are modified points. So $y_i=x_i$, while $y_{i-1}$ and $y_{i+1}$ are modified from $x_{i-1}$ and $x_{i+1}$ respectively.
In this case, neither $\overline{y_{i-1}y_i}$ nor $\overline{y_iy_{i+1}}$ are unmodified $3$-dim pieces, and $d(y_{i-1},y_i),d(y_i,y_{i+1})\geq L$ hold by Lemma \ref{modifiedsequencelength}.

                  \begin{enumerate}
                  \item Neither $\overline{x_{i-1}x_i}$ nor $\overline{x_ix_{i+1}}$ are $3$-dim pieces. By Lemma \ref{onemodification} (1), we have $\angle x_{i-1}y_iy_{i-1}, \angle x_{i+1}y_iy_{i+1}\leq 2e^{-\frac{L}{2}}$. Then by equation (\ref{6.3}), we have $\angle x_{i-1}y_ix_{i+1}\geq \phi_0$, so 
$$\angle y_{i-1}y_iy_{i+1}\geq \phi_0-4e^{-\frac{L}{2}}\geq \frac{\phi_0}{2}\geq \eta_0.$$

                  \item Both $\overline{x_{i-1}x_i}$ and $\overline{x_ix_{i+1}}$ are $3$-dim pieces. It is impossible.

                  \item Exactly one of $\overline{x_{i-1}x_i}$ and $\overline{x_ix_{i+1}}$ is a $3$-dim piece. We assume that $\overline{x_{i-1}x_i}$ is not a $3$-dim piece and $\overline{x_ix_{i+1}}$ is a $3$-dim piece. By Lemma \ref{onemodification} (1), we have $\angle x_{i-1}y_iy_{i-1}\leq 2e^{-\frac{L}{2}}$. 

                             \begin{itemize}
                             \item If $y_i$ and $y_{i+1}$ do not lie on the same edge of $\tilde{Z}^{(1)}$, then Lemma \ref{onemodification} (2) implies $\angle x_{i+1}y_iy_{i+1}\leq 2e^{-\frac{L}{2}}$. Since $\angle x_{i-1}y_ix_{i+1}\geq \phi_0$ (equation (\ref{6.3})), we have 
$\angle y_{i-1}y_iy_{i+1}\geq \phi_0-4e^{-\frac{L}{2}}\geq \frac{\phi_0}{2}\geq \eta_0.$

                              \item If $y_i$ and $y_{i+1}$ lie on the same edge of $\tilde{Z}^{(1)}$, then Lemma \ref{onemodification} (3) implies $\angle x_{i+1}y_iy_{i+1}\leq \frac{\pi}{2}$. Since $\angle x_{i-1}y_ix_{i+1}\geq \frac{3}{5}\pi$ (equation (\ref{6.4})), we get 
                              $$\angle y_{i-1}y_iy_{i+1}\geq \angle x_{i-1}y_ix_{i+1}-\angle x_{i-1}y_iy_{i-1}-\angle x_{i+1}y_iy_{i+1}\geq \frac{3}{5}\pi-2e^{-\frac{L}{2}}-\frac{1}{2}\pi\geq \eta_0.$$
                              % so $\angle x_{i-1}y_iy_{i+1}\geq \frac{\pi}{2}$. Then we get $\angle y_{i-1}y_iy_{i+1}\geq \angle x_{i-1}y_iy_{i+1}-2e^{-\frac{L}{2}}\geq \eta_0$. 
                              \end{itemize}

                    \end{enumerate}

         \end{enumerate}

So the proof of Case I is done.

\bigskip

{\bf Case II}. Now we assume that $y_i$ is a modified point. In this case, there might be several consecutive points $x_i,\cdots,x_{i+k}$ in the intersection sequence of $\gamma$ that are modified to $y_i$, then we have $d(y_i,x_i),\cdots,d(y_i,x_{i+k})<L$. Note that neither $\overline{y_{i-1}y_i}$ nor $\overline{y_iy_{i+1}}$ are unmodified $3$-dim pieces in this case.

          \begin{enumerate}
          \item Both $y_{i-1}$ and $y_{i+1}$ are unmodified points. 

                   \begin{enumerate}
                   \item We assume that $k=0$ holds, then the picture is shown as in Figure \ref{figure7} (a). This figure shows a flattened picture of $\tilde{Z}^3$ in a hyperbolic plane, while the actual picture is bended in $\mathbb{H}^3$. By Construction \ref{modifiedsequence}, at least one of $d(x_{i-1},x_i)$ and $d(x_i,x_{i+1})$ is smaller than $\frac{1}{160}\log{R}$ and we assume $d(x_{i-1},x_i)<\frac{1}{160}\log{R}$. Then we have $d(x_{i-1},y_i)\leq d(x_{i-1},x_i)+d(x_i,y_i)<\frac{1}{160}\log{R}+L$. By Assumption \ref{assumption} (1), even if $x_{i-1}=x$ holds, we know that $x_{i-1}$ lies on an edge of $\tilde{Z}^{(1)}$. Since $d(x_{i-1},y_i)<\frac{1}{160}\log{R}+L<\frac{1}{80}\log{R}$, by Lemma \ref{new} (1), $x_{i-1}$ and $y_i$ must lie on the same edge of $\tilde{Z}^{(1)}$.

                            \begin{itemize}
                            \item  If $\overline{x_{i-1}x_i}$ is not a $3$-dim piece, then $\angle x_{i-1}y_ix_i\geq \phi_0$ holds by the definition of $\phi_0$ in Notation \ref{notation} (3).

                            \item If $\overline{x_{i-1}x_i}$ is a $3$-dim piece, then $\overline{x_ix_{i+1}}$ is not a $3$-dim piece. Since $x_i$ is a modified point and $x_{i-1}$ is not, by Construction \ref{modifiedsequence} (2), we have $d(x_i,x_{i+1})<\frac{1}{160}\log{R}$. The same argument as above implies $\angle x_iy_ix_{i+1}\geq \phi_0$.
                            \end{itemize}

                             Let $S$ be the subset of $S^2$ corresponding to vertex $y_i$ given in the beginning of this proof, and let $\vec{v}_{i-1}$ and $\vec{v}_{i+1}$ be points in $S$ given by tangent vectors of $\overline{y_ix_{i-1}}$ and $\overline{y_ix_{i+1}}$ respectively. The above inequalities on $\angle x_{i-1}y_ix_i$ and $\angle x_iy_ix_{i+1}$ imply $d_S(\vec{v}_{i-1},\vec{v}_{i+1})\geq \phi_0$, otherwise the concatenation $\overline{x_{i-1}x_i}\cdot \overline{x_ix_{i+1}}$ is not the shortest path in $\tilde{Z}^3$ from $x_{i-1}$ to $x_{i+1}$. The choice of $\theta_0$ implies $\angle y_{i-1}y_iy_{i+1}=\angle x_{i-1}y_ix_{i+1}\geq \theta_0\geq \eta_0$ holds. This argument will be used repeatedly in the following part of this proof, referred as ``the argument in Case II (1)(a)''.
 
                      \item We assume that $k= 1$ holds, then the picture is shown in Figure \ref{figure7} (b). 

                               \begin{itemize}
                               \item If $\overline{x_ix_{i+1}}$ is not a $3$-dim piece, then we have $\angle x_iy_ix_{i+1}\geq \phi_0$. The argument in Case II (1)(a) implies $\angle y_{i-1}y_iy_{i+1}=\angle x_{i-1}y_ix_{i+1}\geq \theta_0\geq\eta_0$.

                                \item If $\overline{x_ix_{i+1}}$ is a $3$-dim piece, then by Construction \ref{modifiedsequence} (2), either $d(x_{i-1},x_i)<\frac{1}{160}\log R$ or $d(x_{i+1},x_{i+2})<\frac{1}{160}\log R$ holds. We assume that $d(x_{i-1},x_i)<\frac{1}{160}\log R$ holds. By Assumption \ref{assumption} (1) and Lemma \ref{new} (1) again, even if $x_{i-1}= x$ holds,  $x_{i-1}$ lies on an edge of $\tilde{Z}^{(1)}$ containing $y_i$. Since $\overline{x_{i-1}x_i}$ is not a $3$-dim piece, we have $\angle x_{i-1}y_ix_i\geq \phi_0$. The argument in Case II (1)(a) implies $\angle y_{i-1}y_iy_{i+1}=\angle x_{i-1}y_ix_{i+2}\geq \theta_0\geq\eta_0$.

                                 \end{itemize}

                        \item We assume that $k\geq 2$ holds, then the picture is shown in Figure \ref{figure7} (c). Here either $\overline{x_ix_{i+1}}$ or $\overline{x_{i+1}x_{i+2}}$ is not a $3$-dim piece, thus either $\angle x_iy_ix_{i+1}\geq \phi_0$ or $\angle x_{i+1}y_ix_{i+2}\geq \phi_0$ holds. Again, the argument in Case II (1)(a) implies $\angle y_{i-1}y_iy_{i+1}=\angle x_{i-1}y_ix_{i+k+1}\geq \theta_0\geq\eta_0$.
                        \end{enumerate}
                        
                        \begin{figure}
\centering
\label{new picture}
\def\svgwidth{1.0\textwidth}
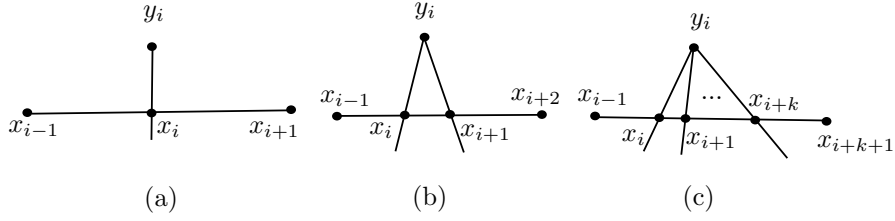
\caption{\ $y_i$ is a modified point, while $y_{i-1}=x_{i-1}$ and $y_{i+1}=x_{i+k+1}$ are not.}
\label{figure7}
\end{figure}

                \item Exactly one of $y_{i-1}$ and $y_{i+1}$ is a modified point, and we assume $y_{i-1}$ is modified (from $x_{i-1}$) and $y_{i+1}$ is unmodified (equals $x_{i+k+1}$). Since $y_{i-1}$ and $y_i$ are distinct vertices of $\tilde{Z}^{(1)}$, by Lemma \ref{new} (4), we have $d(y_{i-1},y_i)\geq \log{R}$. Since $x_{i-1}$ and $x_i$ are modified to $y_{i-1}$ and $y_i$ respectively, we have $d(x_{i-1},y_{i-1}),d(x_i,y_i)<L$. So we have 
$$\angle x_{i-1}y_iy_{i-1}\leq 2\sin{\angle x_{i-1}y_iy_{-1}}\leq 2 \frac{\sinh{d(x_{i-1},y_{i-1})}}{\sinh{d(y_{i-1},y_i)}}\leq 2e^{-\frac{R}{2}}.$$
Moreover, we have
$$\qquad \qquad d(x_{i-1},x_i)\geq d(y_{i-1},y_i)-d(x_{i-1},y_{i-1})-d(x_i,y_i)\geq \log R-2L\geq \frac{1}{2}\log{R}.$$
Now we claim that $\angle x_{i-1}y_ix_{i+k+1}\geq \theta_0$ holds, and the proof divides into following cases.

                         \begin{enumerate}
                         \item We first assume that $k=0$ holds, then the picture is shown in Figure \ref{figure8} (a). 
Since $x_i$ is modified to $y_i$, by Construction \ref{modifiedsequence}, $\overline{x_ix_{i+1}}$ can not be a $3$-dim piece, and we must have $d(x_i,x_{i+1})<\frac{1}{160}\log{R}$. By Assumption \ref{assumption} (1) and Lemma \ref{new} (1), even if $x_{i+1}= y$ holds, $x_{i+1}$ lies on an edge of $\tilde{Z}^{(1)}$ containing $y_i$, and we have $\angle x_iy_ix_{i+1}\geq \phi_0$. By the argument in Case II (1)(a), we have $\angle x_{i-1}y_ix_{i+1}\geq \theta_0$.

                          \item We assume that $k=1$ holds, then the picture is shown in Figure \ref{figure8} (b).
                                  \begin{itemize}
                                  \item If $\overline{x_ix_{i+1}}$ is not a $3$-dim piece, then we have $\angle x_iy_ix_{i+1}\geq \phi_0$. The argument  in Case II (1)(a) implies $\angle x_{i-1}y_ix_{i+2}\geq \theta_0$.

                                  \item If $\overline{x_ix_{i+1}}$ is a $3$-dim piece, then $\overline{x_{i+1}x_{i+2}}$ is not a $3$-dim piece. Since $d(x_{i-1},x_i)\geq\frac{1}{2}\log{R}$, by Construction \ref{modifiedsequence} (2), we must have $d(x_{i+1},x_{i+2})<\frac{1}{160}\log{R}$. By Assumption \ref{assumption} (1) and Lemma \ref{new} (1), even if $x_{i+2}= y$ holds, $x_{i+2}$  lies on an edge of $\tilde{Z}^{(1)}$ containing $y_i$, and $\angle x_{i+1}y_ix_{i+2}\geq \phi_0$ holds. Again, the argument  in Case II (1)(a) implies $\angle x_{i-1}y_ix_{i+2}\geq \theta_0$.
                                  \end{itemize}

                          \item We assume that $k\geq 2$ holds, then the picture is shown in Figure \ref{figure8} (c). Then either $\overline{x_ix_{i+1}}$ or $\overline{x_{i+1}x_{i+2}}$ is not a $3$-dim piece, thus either $\angle x_iy_ix_{i+1}\geq \phi_0$ or $\angle x_{i+1}y_ix_{i+2}\geq \phi_0$ holds. Again, the argument in Case II (1)(a) implies $\angle x_{i-1}y_ix_{i+k+1}\geq \theta_0$.
                          \end{enumerate}
                          
                          So the claim is established, and we have 
$$\qquad \quad \angle y_{i-1}y_iy_{i+1}=\angle y_{i-1}y_ix_{i+k+1}\geq \angle x_{i-1}y_ix_{i+k+1}-\angle x_{i-1}y_iy_{i-1}\geq \theta_0-2e^{-\frac{R}{2}}\geq \eta_0.$$ 

                        \begin{figure}
\centering
\label{new picture}
\def\svgwidth{1.0\textwidth}
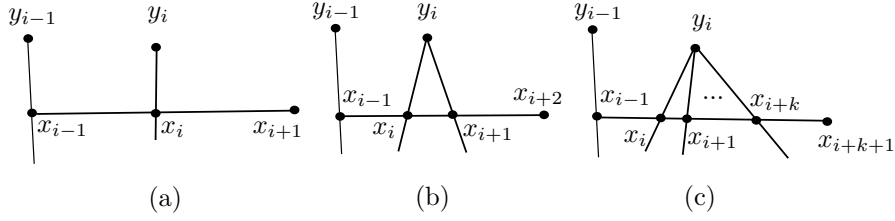
\caption{\ $y_{i-1}$ and $y_i$ are modified points, while $y_{i+1}=x_{i+k+1}$ is not.}
\label{figure8}
\end{figure}

                  \item Both $y_{i-1}$ and $y_{i+1}$ are modified points. Then $y_{i-1}$ and $y_{i+1}$ are obtained by modifying $x_{i-1}$ and $x_{i+k+1}$ respectively. Since $y_{i-1},y_i, y_{i+1}$ are distinct vertices of $\tilde{Z}$, by Lemma \ref{new} (4), we have $d(y_{i-1},y_i),d(y_i,y_{i+1})\geq \log{R}$. By the modification process in Construction \ref{modifiedsequence}, we have 
                  $$d(x_{i-1},y_{i-1}), d(x_i,y_i), \cdots, d(x_{i+k},y_i),d(x_{i+k+1},y_{i+1})<L.$$
By the computation at the beginning of Case II (2), we have 
$$\qquad \quad \angle x_{i-1}y_iy_{i-1},\angle x_{i+k+1}y_iy_{i+1}\leq 2e^{-\frac{R}{2}};\ \  d(x_{i-1},x_i),d(x_{i+k},x_{i+k+1})>\frac{1}{2}\log{R}.$$
As in Case II (2), we claim that $\angle x_{i-1}y_ix_{i+k+1}\geq \theta_0$, and the proof divides into following cases.

                           \begin{enumerate}
                           \item We first assume that $k=0$ holds, then the picture is shown in Figure \ref{figure9} (a). Since $d(x_{i-1},x_i),d(x_{i},x_{i+1})>\frac{1}{2}\log{R}$, Construction \ref{modifiedsequence} implies that $x_i$ should not be modified. This case is impossible.

                           \item We assume that $k=1$ holds, then the picture is shown in Figure \ref{figure9} (b).
                                   \begin{itemize}
                                   \item If $\overline{x_ix_{i+1}}$ is not a $3$-dim piece, then we have $\angle x_iy_ix_{i+1}\geq \phi_0$. The argument  in Case II (1)(a) implies $\angle x_{i-1}y_ix_{i+2}\geq \theta_0$.

                                   \item If $\overline{x_ix_{i+1}}$ is a $3$-dim piece, since $d(x_{i-1},x_i),d(x_{i+k},x_{i+k+1})>\frac{1}{2}\log{R}$, Construction \ref{modifiedsequence} (2) implies that $x_i$ and $x_{i+1}$ should not be modified. This case is impossible.
                                   \end{itemize}

                            \item We assume that $k\geq 2$ holds, then the picture is shown in Figure \ref{figure9} (c). Then either $\overline{x_ix_{i+1}}$ or $\overline{x_{i+1}x_{i+2}}$ is not a $3$-dim piece, thus either $\angle x_iy_ix_{i+1}\geq \phi_0$ or $\angle x_{i+1}y_ix_{i+2}\geq \phi_0$ holds. Again, the argument in Case II (1)(a) implies $\angle x_{i-1}y_ix_{i+k+1}\geq \theta_0$.
                            \end{enumerate}
                            So the claim is established, and we have 
$$\qquad \quad \angle y_{i-1}y_iy_{i+1}\geq \angle x_{i-1}y_ix_{i+k+1}-\angle x_{i-1}y_iy_{i-1}-\angle x_{i+k+1}y_iy_{i+1}\geq \theta_0-4e^{-\frac{R}{2}} \geq \eta_0.$$
                  \end{enumerate}
                  
                                          \begin{figure}
\centering
\label{new picture}
\def\svgwidth{1.0\textwidth}
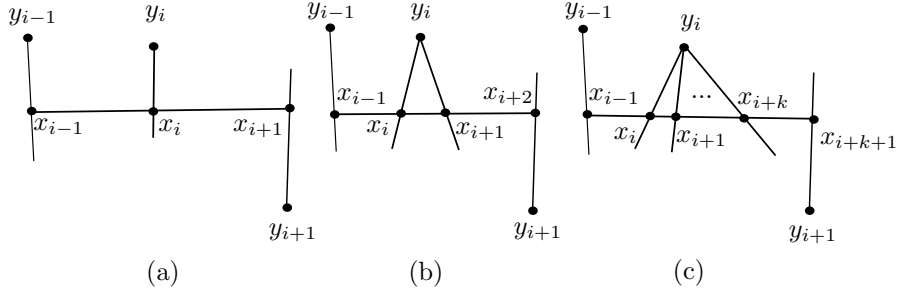
\caption{\ All of $y_{i-1}, y_i,y_{i+1}$ are modified points.}
\label{figure9}
\end{figure}

The proof of Case II is done and the proof of this lemma is finished.
\end{proof}

Now we are ready to prove Proposition \ref{model}.

\begin{proof}[Proof of Proposition \ref{model}]

We first take $\eta_0>0$ in Lemma \ref{modifiedsequenceangle}, and take $L>0$ such that $\frac{L}{2}$ satisfies the assumption of Proposition \ref{lengthshort} with respect to $\frac{\eta_0}{2}$. Then we enlarge $L$ and take large $R>0$ so that equation (\ref{6.0}) and Lemma \ref{modifiedsequenceangle} hold.

{\bf (1) $\tilde{j}_0:\tilde{Z}^3\to \mathbb{H}^3$ is a quasi-isometric embedding.} Since $\tilde{Z}''$ (defined before Construction \ref{constructjt}) is $2$-dense in $\tilde{Z}^3$, we only need to prove that the restriction $\tilde{j}_0|:\tilde{Z}''\to \mathbb{H}^3$ is a quasi-isometric embedding.
More precisely, for any $x,y\in \tilde{Z}''$, we will prove that
\begin{align}\label{6.5}
\frac{1}{2}d_{\tilde{Z}^3}(x,y)-\frac{3}{40}\log{R}-4L \leq d_{\mathbb{H}^3}(\tilde{j}_0(x),\tilde{j}_0(y)) \leq d_{\tilde{Z}^3}(x,y)+\frac{3}{40}\log{R}+4L.
\end{align}
We first do the following two-step modification on $x$ and $y$.

{\bf Modification I.} If $x$ or $y$ lie in the $(L+\frac{1}{80}\log{R})$-neighborhood of some vertex of $\tilde{Z}^3$ (which is unique by Lemma \ref{new} (4)), we replace it by the corresponding vertex. So we can assume that $x$ and $y$ are either vertices of $\tilde{Z}^3$ or do not lie in the $(L+\frac{1}{80}\log{R})$-neighborhood of any vertex of $\tilde{Z}^3$. Under this assumption, we only need to prove the following estimate, which implies (\ref{6.5}).
\begin{align}\label{6.6}
\frac{1}{2}d_{\tilde{Z}^3}(x,y)-\frac{1}{40}\log{R}\leq d_{\mathbb{H}^3}(\tilde{j}_0(x),\tilde{j}_0(y)) \leq d_{\tilde{Z}^3}(x,y)+\frac{1}{40}\log{R}.
\end{align}

Let $\gamma$ be the shortest path in $\tilde{Z}^3$ from $x$ to $y$. If the interior of $\gamma$ is contained in $\tilde{Z}^3\setminus \tilde{Z}^{(1)}$, then $\tilde{j}_0(\gamma)$ is a geodesic segment in $\mathbb{H}^3$, thus $d_{\tilde{Z}^3}(x,y)=d_{\mathbb{H}^3}(\tilde{j}_0(x),\tilde{j}_0(y))$ holds and (\ref{6.6}) holds. So we can assume that $\gamma$ is not contained  in $\tilde{Z}^3\setminus \tilde{Z}^{(1)}$.

{\bf Modification II.} We take the intersection sequence $x_1,\cdots,x_n$. If $d(x,x_1)<\frac{1}{160}\log{R}$ or $d(x_n,y)<\frac{1}{160}\log{R}$, we replace $x$ or $y$ by $x_1$ or $x_n$ respectively. We still denote the new initial and terminal points by $x$ and $y$, and still denote the new shortest path between $x$ and $y$ by $\gamma$. Then we only need to prove the following estimate, which implies (\ref{6.6}).
\begin{align}\label{6.7}
\frac{1}{2}d_{\tilde{Z}^3}(x,y) \leq d_{\mathbb{H}^3}(\tilde{j}_0(x),\tilde{j}_0(y)) \leq d_{\tilde{Z}^3}(x,y).
\end{align}

We claim that the path $\gamma$ obtained by Modifications I and II satisfies Assumption \ref{assumption}. We will only argue for the initial point $x$, and the proof for the terminal point $y$ is the same. 
\begin{itemize}
\item If we did Modification I for $x$, then the new $x$ is a vertex of $\tilde{Z}^3$ and its distance to any edge of $\tilde{Z}^{(1)}$ not containing $x$ is at least $\log{R}$ (by Lemma \ref{new} (1)). So Modification II is not applied to $x$, and Assumption \ref{assumption} holds. 

\item If we did not do Modification I but did Modification II for $x$, then the new $x$ lies on an edge of $\tilde{Z}^{(1)}$ and its distance to any vertex of $\tilde{Z}^3$ is at least $(L+\frac{1}{80}\log{R})-\frac{1}{160}\log{R}=L+\frac{1}{160}\log{R}$. After Modification II, if $\overline{xx_1}$ is not a $3$-dim piece and $d(x,x_1)<\frac{1}{160}\log{R}$, then the edges of $\tilde{Z}^{(1)}$ containing $x$ and $x_1$ share a vertex $v$ (by Lemma \ref{new} (3)) and $\angle{xvx_1}\geq \phi_0$ holds. So we have 
$$\qquad \quad d(x,x_1)\geq d(x,v)+d(v,x_1)-I(\pi-\phi_0)\geq (L+\frac{1}{160}\log{R})-I(\pi-\phi_0)>\frac{1}{160}\log{R},$$
which is impossible. So Assumption \ref{assumption} holds in this case. If $\overline{xx_1}$ is a $3$-dim piece, Assumption \ref{assumption} trivially holds.

\item If we neither do Modification I nor Modification II for $x$, then the distance between $x$ and any vertex of $\tilde{Z}^3$ is at least $L+\frac{1}{80}\log{R}>L+\frac{1}{160}\log{R}$, and we have $d(x,x_1)\geq \frac{1}{160}\log{R}$. So Assumption \ref{assumption} holds.
\end{itemize}

Now we take the modified sequence $y_1,\cdots,y_m$ of $\gamma$, and let $\gamma_i'$ be the shortest path (in a piece of $\tilde{Z}^3\setminus \tilde{Z}^{(1)}$) from $y_i$ to $y_{i+1}$. The modified path $\gamma'$ is the concatenation of $\gamma_0',\gamma_1',\cdots,\gamma_m'$.

For each consecutive $\gamma_i'$ and $\gamma_{i+1}'$ in the modified path $\gamma$, we have the following possibilities. 
\begin{itemize}
\item If neither of them are unmodified $3$-dim pieces, then Lemma \ref{modifiedsequencelength} implies $l(\gamma_i'),l(\gamma_{i+1}')\geq L$, and Lemma \ref{modifiedsequenceangle} (1) implies the bending angle is at most $\pi-\eta_0$. 
\item If either $\gamma_i'$ or $\gamma_{i+1}'$ is an unmodified $3$-dim piece, exactly one of them is. Then Lemma \ref{modifiedsequencelength} implies that one of $l(\gamma_i'),l(\gamma_{i+1}')$ is at least $L$, and Lemma \ref{modifiedsequenceangle} (2) implies the bending angle is at most $\frac{\pi}{2}-\eta_0$.
\end{itemize}

Then Proposition \ref{lengthshort} implies 
\begin{align}\label{6.8}
d_{\mathbb{H}^3}(\tilde{j}_0(x),\tilde{j}_0(y))=l(\gamma_0'\gamma_1'\cdots\gamma_m')\geq \frac{1}{2}\sum_{i=0}^ml(\gamma_i')=\frac{1}{2}\sum_{i=0}^md_{\tilde{Z}^3}(y_i,y_{i+1})\geq \frac{1}{2}d_{\tilde{Z}^3}(x,y).
\end{align}
On the other hand, since the metric on $\tilde{Z}^3$ is a path metric induced by the metric of $\mathbb{H}^3$, we always have $d_{\mathbb{H}^3}(\tilde{j}_0(x),\tilde{j}_0(y))\leq d_{\tilde{Z}^3}(x,y)$. So equation (\ref{6.7}) holds for the path $\gamma$ obtained after Modifications I and II, thus equation (\ref{6.5}) holds for any $x,y\in \tilde{Z}''$. This implies that
\begin{align}\label{6.9}
\frac{1}{2}d_{\tilde{Z}^3}(x,y)-\frac{3}{40}\log{R}-4L-8 \leq d_{\mathbb{H}^3}(\tilde{j}_0(x),\tilde{j}_0(y)) \leq d_{\tilde{Z}^3}(x,y)+\frac{3}{40}\log{R}+4L+8
\end{align}
holds for any $x,y\in \tilde{Z}^3$, by the $2$-denseness of $\tilde{Z}''\subset \tilde{Z}^3$. So $\tilde{j}_0:\tilde{Z}^3\to \mathbb{H}^3$ is a quasi-isometric embedding.

{\bf (2) $\pi_1$-injectivity of $\rho_0$.} Since $\pi_1(Z)\cong \pi_1(Z^3)$ is torsion-free and $\tilde{j}_0:\tilde{Z}^3\to \mathbb{H}^3$ is $\rho_0$-equivariant, the fact that $\tilde{j}_0$ is a quasi-isometric embedding implies that $\rho_0$ is injective.

Moreover, since $\pi_1(Z^3)$ is neither a surface group nor a free group, the covering theorem (\cite{Can}) implies that $\rho_0(\pi_1(Z^3))<\text{Isom}_+(\mathbb{H}^3)$ is a geometrically finite subgroup.

{\bf (3) Injectivity of $\tilde{j}_0$.} Now we prove that  $\tilde{j}_0:\tilde{Z}^3\to\mathbb{H}^3$ is injective.
For $x,y\in \tilde{Z}^3$ such that $d_{\tilde{Z}^3}(x,y)>\frac{1}{5}\log{R}>\frac{3}{20}\log{R}+8L+16$, the left hand side of equation (\ref{6.9}) implies $\tilde{j}_0(x)\ne \tilde{j}_0(y)$. Moreover, if $x$ and $y$ lie in the same component of $\tilde{Z}^3\setminus \tilde{Z}^{(1)}$, then since $\tilde{j}_0$ restricts to an embedding on this component, we have $\tilde{j}_0(x)\ne \tilde{j}_0(y)$.

So we can assume that $d_{\tilde{Z}^3}(x,y)\leq\frac{1}{5}\log{R}$ holds, while $x$ and $y$ lie in different components of $\tilde{Z}^3\setminus \tilde{Z}^{(1)}$. We take the shortest path $\gamma$ in $\tilde{Z}^3$ from $x$ to $y$ and take the intersection sequence $x_1,\cdots,x_n$. Let $x=x_0$ and $y=x_{n+1}$, and we denote the subpath of $\gamma$ from $x_i$ to $x_{i+1}$ by $\gamma_i$. Then we have $l(\gamma_i)\leq \frac{1}{5}\log{R}$ for all $i=0,\cdots,n$. So for any $\gamma_i$ with $i=1,\cdots,n-1$, one of the following hold.
\begin{enumerate}
\item [(i)] Either $\gamma_i$ is a $3$-dim piece.
\item [(ii)] Or $\gamma_i$ is a $2$-dim piece, and by Lemma \ref{new} (3), the two edges of $\tilde{Z}^{(1)}$ containing $x_i$ and $x_{i+1}$ share a vertex $v_i$. Moreover, since $\angle x_iv_ix_{i+1}\geq \phi_0$, we have $d(x_i,v_i),d(x_{i+1},v_i)<\frac{2}{5}\log{R}$. 
\end{enumerate}
Moreover, $\gamma$ contains at most one $3$-dim piece, since by Lemma \ref{new} (2), any two different $3$-dim components of $\tilde{Z}^3\setminus \tilde{Z}^{(1)}$ have distance at least $\frac{9}{10}\log{R}$. 

Now we prove $\tilde{j}_0(x)\ne \tilde{j}_0(y)$ by dividing into following cases.
\begin{enumerate}
\item [(a)] If $\gamma$ contains no $3$-dim pieces, then all vertices $v_i$ in item (ii) above must be the same vertex, thus $x$ and $y$ lie in two $2$-dim pieces of $\tilde{Z}^3\setminus \tilde{Z}^{(1)}$ that share a vertex. Since $\tilde{j}_0$ maps these two pieces to two totally geodesic subsurfaces in $\mathbb{H}^3$ that are disjoint except at the common edge or vertex, we have $\tilde{j}_0(x)\ne \tilde{j}_0(y)$. So we can assume that $\gamma$ contains exactly one $3$-dim piece in the following.

\item [(b)] If $\gamma_0$ or $\gamma_n$ is a $3$-dim piece, we assume that $\gamma_0$ is. Then by item (ii) again, the $3$-dim piece of $\tilde{Z}^3\setminus \tilde{Z}^{(1)}$ containing $x$ and the $2$-dim piece of $\tilde{Z}^3\setminus \tilde{Z}^{(1)}$ containing $y$ share a vertex. Then as in case (a), the geometry of $\tilde{j}_0$ implies $\tilde{j}_0(x)\ne \tilde{j}_0(y)$. So we further assume that neither $\gamma_0$ nor $\gamma_n$ are $3$-dim pieces, and $n\geq 2$ holds in the following.

\item [(c)] If $n=2$, then $\gamma_1$ is a $3$-dim piece. By the proof of Proposition \ref{modifiedsequenceangle} Case I (1)(b), the bending angle at $y_1$ and $y_2$ are both at most $\frac{2}{5}\pi<\frac{\pi}{2}$, so we have $\tilde{j}_0(x)\ne \tilde{j}_0(y)$. So we can further assume that $n\geq 3$ in the following.

%\item [(d)] If $n\geq 3$ and either $\gamma_1$ or $\gamma_{n-1}$ is a $3$-dim piece, and we assume that $\gamma_1$ is. Then the endpoint $y_2$ of $\gamma_1$ lies on an edge of $\tilde{Z}^{(1)}$, and $y_2$ is $\frac{2}{5}\log{R}$-close to a vertex $v$ of this edge. Since $d_{\tilde{Z}^3}(y_1,y_2)\leq \frac{1}{5}\log{R}$, we have $d_{\tilde{Z}^3}(y_1,v)\leq \frac{3}{5}\log{R}$ and Lemma \ref{new} (1) implies that the edge of $\tilde{Z}^{(1)}$ containing $y_1$ must also contain $v$ as a vertex. By item (ii) again, the $2$-dim pieces of $\tilde{Z}^3\setminus \tilde{Z}^{(1)}$ containing $x$ and $y$ share the vertex $v$, and $\tilde{j}_0(x)\ne \tilde{j}_0(y)$ as in case (a) above.
%So we further assume that neither $\gamma_1$ nor $\gamma_{n-1}$ are $3$-dim pieces, and $n\geq 4$ in the following.

\item [(d)] So $n\geq 3$ for some $1\leq i\leq n-1$, $\gamma_i$ is the unique $3$-dim piece of $\gamma$. Then $x_i$ is $(\frac{2}{5}\log{R})$-close to a vertex $v_i$ of $\tilde{Z}^{(1)}$, and $x_{i+1}$ is $(\frac{2}{5}\log{R})$-close to a vertex $v_{i+1}$. We must have $v_i=v_{i+1}$, otherwise $d_{\tilde{Z}^3}(v_i,v_{i+1})\geq \log{R}$ by Lemma \ref{new} (4) and $d_{\tilde{Z}^3}(x_i,x_{i+1})\geq \frac{1}{5}\log{R}$, which is impossible. So both $x_i$ and $x_{i+1}$ lie on edges of $\tilde{Z}^{(1)}$ containing $v_i$, and by item (ii), the $2$-dim pieces of $\tilde{Z}^3\setminus \tilde{Z}^{(1)}$ containing $x$ and $y$ share the vertex $v_i$. So we obtain $\tilde{j}_0(x)\ne \tilde{j}_0(y)$ as in case (a).
\end{enumerate}

The proof of the injectivity of $\tilde{j}_0:\tilde{Z}^3\to \mathbb{H}^3$ is finished.

{\bf (4) Homeomorphic type of the convex core.} Since $\tilde{j}_0:\tilde{Z}^3\to \mathbb{H}^3$ is injective and is $\rho_0$-equivariant, the image $\tilde{j}_0(\tilde{Z}^3)$ has a closed $\rho_0(\pi_1(Z^3))$-equivariant neighborhood $\mathcal{N}(\tilde{Z}^3)$ in $\mathbb{H}^3$. By the construction of $\mathcal{Z}$, we can see that $\mathcal{N}(\tilde{Z}^3)/\rho_0(\pi_1(Z^3))$ is homeomorphic to $\mathcal{Z}\setminus \partial_p \mathcal{Z}$, as oriented manifolds.

Also note that $\mathcal{N}(\tilde{Z}^3)/\rho_0(\pi_1(Z^3))$ is a finite volume submanifold of $\mathbb{H}^3/\rho_0(\pi_1(Z^3))$ such that the inclusion induces an isomorphism on $\pi_1$. Since all boundary components of $\mathcal{Z}$ are incompressible, by tameness of open hyperbolic $3$-manifolds (\cite{Agol1, CG}), each component of 
$$\big(\mathbb{H}^3/\rho_0(\pi_1(Z^3))\big)\setminus \big(\mathcal{N}(\tilde{Z}^3)/\rho_0(\pi_1(Z^3))\big)$$
is homeomorphic to the product of a surface and $(0,\infty)$. So $\mathbb{H}^3/\rho_0(\pi_1(Z^3))$ is homeomorphic to $\mathcal{Z}\setminus \partial_p \mathcal{Z}$. 

Since $\rho_0(\pi_1(Z^3))<\text{Isom}_+(\mathbb{H}^3)$ is a geometrically finite subgroup and $\partial_p \mathcal{Z}$ corresponds to cusp ends of  $\mathbb{H}^3/\rho_0(\pi_1(Z^3))$, the convex core of $\mathbb{H}^3/\rho_0(\pi_1(Z^3))$ is homeomorphic to $\mathcal{Z}\setminus \partial_p \mathcal{Z}$. Note that all above homeomorphisms preserve natural orientations on involved manifolds.

\end{proof}

\bigskip

\subsection{Proof of quasi-isometric embedding}\label{qi2}

The following proposition is the main result of this subsection, which is the last technical piece of the proof of Theorem \ref{pi1injectivity}.

\begin{proposition}\label{deformation}
For any $t\in [0,1]$, $\tilde{j}_t:\tilde{Z}^3\to \mathbb{H}^3$ is a quasi-isometric embedding.
\end{proposition}

To prove Proposition \ref{deformation}, we need the following two lemmas. The first lemma appeared as Lemma 5.7 of \cite{Sun5}, which estimates the geometry of $\tilde{j}_t$ on $2$-dim pieces of $\tilde{Z}^3\setminus \tilde{Z}^{(1)}$. 

\begin{lemma}\label{estimate2dpiece}
For any $\delta\in(0,10^{-6})$, there exists $\epsilon_0>0$ and $R_0>0$, such that for any positive numbers $\epsilon\in (0,\epsilon_0)$, $R>R_0$ and any positive integer $R'$ greater than all of the $R_{ij}$ and $R_{ijk}$ in Lemma \ref{checkgoodcurve}, the following statement holds. 

If $\{\tilde{j}_t:\tilde{Z}^3\to \mathbb{H}^3\ |\ t\in [0,1]\}$ is constructed with respect to $\epsilon,R$ and $R'$, then for any $t\in[0,1]$ and any $x,y$ lying in the closure of a $2$-dim piece $C\subset \tilde{Z}^3\setminus \tilde{Z}^{(1)}$ such that $x\in \partial C$, we have:
\begin{align}\label{a}
\frac{1}{2}d_{\mathbb{H}^3}(\tilde{j}_t(x),\tilde{j}_t(y))\leq d_{\tilde{Z}^3}(x,y)\leq 2d_{\mathbb{H}^3}(\tilde{j}_t(x),\tilde{j}_t(y)).
\end{align}
Moreover, let $e$ be an edge in $\tilde{Z}^{(1)}$ containing $x$ (with a preferred orientation), if $d(x,y)\geq 100$, then we have
\begin{align}\label{b}
d_{S^2}\big(\Theta(x,y,e),\Theta(\tilde{j}_t(x),\tilde{j}_t(y),\tilde{j}_t(e))\big)<10\delta.
\end{align}
Here $\Theta(x,y,e)$ denotes the point in $S^2$ determined by the tangent vector of $\overline{xy}$ in $\mathbb{H}^3$, with respect to a coordinate of $T_x(\mathbb{H}^3)$ given by a frame ${\bf p}=(x,\vec{v},\vec{n})$, where $\vec{v}$ is tangent to $e$ and $\vec{n}$ is tangent to $C$ (points inward). Similarly, $\Theta(\tilde{j}_t(x),\tilde{j}_t(y),\tilde{j}_t(e))$ is defined with respect to a frame based at $\tilde{j}_t(x)$, with the first vector tangent to $\tilde{j}_t(e)$, and the second vector is $\epsilon$-close to be tangent to $\tilde{j}_t(C)$ (points inward).
\end{lemma}

The second lemma estimates the geometry of $\tilde{j}_t$ on $3$-dim pieces of $\tilde{Z}^3\setminus \tilde{Z}^{(1)}$. Since this lemma only concerns $3$-dim pieces of $\tilde{Z}^3$, actually we do not need the $R_{ij}$, $R_{ijk}$ and $R'$ part of $\tilde{Z}^3$, but we still state them in the following lemma, so that its statement is parallel with the statement of Lemma \ref{estimate2dpiece}. We will only give a sketch of the proof of this lemma and some computations are skipped.

\begin{lemma}\label{estimate3dpiece}
For any $\delta\in(0,10^{-6})$, there exists $\epsilon_0>0$ and $R_0>0$, such that for any positive numbers $\epsilon\in (0,\epsilon_0)$, $R>R_0$ and any positive integer $R'$ greater than all of the $R_{ij}$ and $R_{ijk}$ in Lemma \ref{checkgoodcurve}, the following statement holds. 

If $\{\tilde{j}_t:\tilde{Z}^3\to \mathbb{H}^3\ |\ t\in [0,1]\}$ is constructed with respect to $\epsilon,R$ and $R'$, then for any $t\in[0,1]$, the following hold. For any $x,y$ lying in the closure of a $3$-dim piece $C\subset \tilde{Z}^3\setminus \tilde{Z}^{(1)}$, we have:
\begin{align}\label{a}
(1-\delta)d_{\mathbb{H}^3}(\tilde{j}_t(x),\tilde{j}_t(y))\leq d_{\tilde{Z}^3}(x,y)\leq (1+\delta) d_{\mathbb{H}^3}(\tilde{j}_t(x),\tilde{j}_t(y)).
\end{align}

Moreover, if $x$ belongs to an oriented edge $e\subset \tilde{Z}^{(1)}$ contained in the boundary of $C$, then we have 
\begin{align}\label{angledifference}
d_{S^2}\big(\Theta(x,y,e),\Theta(\tilde{j}_t(x),\tilde{j}_t(y),\tilde{j}_t(e))\big)<10\delta.
\end{align}
Here $\Theta(x,y,e)$ denotes the point in $S^2$ determined by the tangent vector of $\overline{xy}$ in $\mathbb{H}^3$, with respect to a coordinate of $T_x(\mathbb{H}^3)$ given by a frame ${\bf p}=(x,\vec{v},\vec{n})$, where $\vec{v}$ is tangent to $e$ and $\vec{n}$ is tangent to a face of $C$ (points inward). $\Theta(\tilde{j}_t(x),\tilde{j}_t(y),\tilde{j}_t(e))$ is defined by a similar frame based at $\tilde{j}_t(x)$, given by $\tilde{j}_t(e)$ and $\tilde{j}_t(C)$.

\end{lemma}

\begin{proof}
At first, note that we did not give a precise definition of $\tilde{j}_t$ on ideal tetrahedra of $\tilde{Z}^3$ in Construction \ref{constructjt} (4), so it suffices to prove this lemma for some choice of $\tilde{j}_t$.

We can use the Klein model of the hyperbolic space to (non-canonically) identify each ideal tetrahedron in $\tilde{Z}^3$ with a $3$-simplex in the $3$-ball (with one vertex on the boundary). Then we use the linear structure of the $3$-simplex to define the $\tilde{j}_t$ map on each tetrahedron, which is piecewisely smooth on $C$. Moreover, if $\epsilon>0$ is small enough, the restriction of $\tilde{j}_t$ on each ideal tetrahedron is very close to an isometry, up to the second derivative, in the following sense.
\begin{enumerate} 
\item for any two unit tangent vectors $\vec{v}_1,\vec{v}_2$ based at the same point $z\in C$, we have
\begin{align}\label{vectordistortion}
|\langle\vec{v}_1,\vec{v}_2\rangle-\langle D\tilde{j}_t(\vec{v}_1),D\tilde{j}_t(\vec{v}_1)\rangle|<\delta^3.
\end{align}
Here if $z$ lies in the boundary of an ideal tetrahedron, then $\vec{v}_1$ and $\vec{v}_2$ point toward the same ideal tetrahedron.
\item For any geodesic $\gamma$ contained in an ideal tetrahedron contained in $C$, the geodesic curvature of $\tilde{j}_t\circ \gamma$ is always bounded above by $\delta^3$. 
\end{enumerate}

At first, equation (\ref{vectordistortion}) implies that $\tilde{j}_t|_C$ is a $(1+\delta)$-bi-Lipschitz map, so equation (\ref{a}) holds.

Now we work on the angle estimate, and we identify both $C$ and $\tilde{j}_t(C)$ with convex subsets of the upper half space model of $\mathbb{H}^3$, such that all vertices have $z$-coordinate $1$. We take projections of $x$ and $y$ to $\mathbb{R}^2$, and let the Euclidean distance bewteen these projections by $d$.

{\bf Case I.} We first suppose that $d\geq \frac{4R}{\delta}$. Since the $z$-coordinate of $x$ is at most $\sqrt{(\frac{R}{2})^2+1}$ (by construction \ref{constructjt} (3)), an elementary computation implies that the tangent vector of $\overline{xy}$ at $x$ is at most $\delta$ away from $(0,0,1)$. Since $\tilde{j}_t|_C$ is induced by an almost isometry of the equilaterial tessellation of $\mathbb{R}^2$, the Euclidean distance between the projections of $\tilde{j}_t(x)$ and $\tilde{j}_t(y)$ to $\mathbb{R}^2$ is at least $\frac{2R}{\delta}$. So the tangent vector of $\overline{\tilde{j}_t(x)\tilde{j}_t(y)}$ at $\tilde{j}_t(x)$ is at most $2\delta$ away from $(0,0,1)$. Since the geometry of $C$ and $\tilde{j}_t(C)$ are close on their boundaries, equation (\ref{angledifference}) holds in this case.

{\bf Case II.} Now we suppose that $d\leq \frac{4R}{\delta}$.  Since the equilateral tesselation of $\mathbb{R}^2$ has side length $R$, an elementary area estimate implies that $\overline{xy}$ intersects with $m\leq \frac{20}{\delta}$ ideal tetrahedra of $C$. Let $\gamma_1,\gamma_2,\cdots,\gamma_m$ be the intersections of $\overline{xy}$ with these tetrahedra such that their concatenation gives $\overline{xy}$. Then $\tilde{j}_t\circ \gamma_1,\tilde{j}_t\circ \gamma_2,\cdots,\tilde{j}_t\circ \gamma_m$ are smooth curves in tetrahedra of $\tilde{j}_t(C)$ such that their concatenation is homotopic to $\overline{\tilde{j}_t(x)\tilde{j}_t(y)}$, and their geodesic curvatures are always bounded above by $\delta^3$. 

For each $i=1,2,\cdots,m-1$, by equation (\ref{vectordistortion}), the angle between the terminal tangent vector of $\tilde{j}_t(\gamma_i)$ and the initial tangent vector of $\tilde{j}_t(\gamma_{i+1})$ is at most $4\delta^3$.

Let $\gamma_i'$ be the geodesic segment that share endpoints with $\tilde{j}_t(\gamma_i)$. The condition on geodesic curvatures implies that the initial tangent vectors of $\gamma_i'$ and $\tilde{j}_t(\gamma_i)$ differ by at most $2\delta^3$, and the same hold for their terminal tangent vectors. So the angle between the terminal tangent vector of $\gamma_i'$ and the initial tangent vector of $\gamma_{i+1}'$ is at most $10\delta^3$.

Since we have $m\leq \frac{20}{\delta}$ geodesic segments $\gamma_i'$, then initial tangent vectors of $\gamma_1'$ and $\overline{\tilde{j}_t(x)\tilde{j}_t(y)}$ differ by at most $10\delta^3\cdot \frac{20}{\delta}=200\delta^2\leq \delta$. Since the tangent vector of $\gamma_1'$ is $2\delta^3$-close to the tangent vector of $\tilde{j}_t\circ \gamma_1$, equation (\ref{angledifference}) holds in this case. So the proof is done.

\end{proof}

Given these two lemmas, we are ready to prove Proposition \ref{deformation}.

\begin{proof}[Proof of Proposition \ref{deformation}]
For $\eta_0$ given in Lemma \ref{modifiedsequenceangle}, we take $\delta \in (0,\frac{\eta_0}{40})$. Then we take small $\epsilon>0$ and large $R>0$ satisfying Lemmas \ref{estimate2dpiece} and \ref{estimate3dpiece}.

To prove $\tilde{j}_t$ is a quasi-isometric embedding, for any two points $x,y\in \tilde{Z}^3$, we want to prove the following inequality:
\begin{align}\label{6.11}
\frac{1}{4}d_{\tilde{Z}}(x,y)-5(L+\frac{3}{160}\log{R}+2)  \leq d_{\mathbb{H}^3}(\tilde{j}_t(x),\tilde{j}_t(y))\leq 4d_{\tilde{Z}^3}(x,y)+12(L+\frac{3}{160}\log{R}+2).
\end{align}

As in the proof of Proposition \ref{model}, we replace $x$ and $y$ by two $2$-close points in $\tilde{Z}''$ and then do Modifications I and II. After this modification process, the shortest path $\gamma$ in $\tilde{Z}^3$ from $x$ to $y$ satisfies Assumption \ref{assumption}.

Recall that the modification process moves both $x$ and $y$ by distance at most $L+\frac{3}{160}\log{R}+2$. Lemmas \ref{estimate2dpiece} and \ref{estimate3dpiece} imply that $\tilde{j}_t(x)$ and $\tilde{j}_t(y)$ are moved by distance at most $2(L+\frac{3}{160}\log{R}+2)$. So to prove equation (\ref{6.11}), it suffices to prove the following inequality for points $x,y\in \tilde{Z}''$ satisfying Assumption \ref{assumption}:
\begin{align}\label{6.12}
\frac{1}{4}d_{\tilde{Z}^3}(x,y)  \leq d_{\mathbb{H}^3}(\tilde{j}_t(x),\tilde{j}_t(y))\leq 4d_{\tilde{Z}^3}(x,y).
\end{align}

For the new $x$ and $y$, we take the shortest path $\gamma$ in $\tilde{Z}^3$ from $x$ to $y$, take the modified sequence $y_1,\cdots,y_n$. Let $\gamma_i'$ be the shortest path in $\tilde{Z}^3$ from $y_i$ to $y_{i+1}$ (in the closure of a component of $\tilde{Z}^3\setminus \tilde{Z}^{(1)}$), and let $\gamma'$ be the concatenation of $\gamma_i'$ for $i=0,1,\cdots, n$. Since $\gamma$ satisfies Assumption \ref{assumption}, Lemmas \ref{modifiedsequencelength} and \ref{modifiedsequenceangle} imply the following hold: 
\begin{itemize}
\item Each $\gamma_i'$ is either an unmodified $3$-dim piece, or $l(\gamma_i')\geq L$ holds.
\item If neither $\gamma_{i-1}'$ nor $\gamma_i'$ are unmodified $3$-dim pieces, then $\angle y_{i-1}y_iy_{i+1}\geq \eta_0$.
\item If  $\gamma_{i-1}'$ or $\gamma_i'$ is an unmodified $3$-dim piece, then $\angle y_{i-1}y_iy_{i+1}\geq \frac{\pi}{2}+ \eta_0$.
\end{itemize}

Let $z_i=\tilde{j}_t(y_i)$, and let $\delta_i$ be the geodesic segment in $\mathbb{H}^3$ from $z_i$ to $z_{i+1}$. Then we have $z_0=\tilde{j}_t(x)$ and $z_{n+1}=\tilde{j}_t(y)$.
By Lemmas \ref{estimate2dpiece} and \ref{estimate3dpiece}, since $\delta<\frac{\eta_0}{40}$, the following conditions hold for $\delta_i$.
\begin{enumerate}
\item For any $i=0,\cdots,n$, we have $\frac{1}{2}l(\gamma_i')\leq l(\delta_i)\leq 2l(\gamma_i')$. Moreover, if $\gamma_i'$ is not an unmodified $3$-dim piece, $l(\delta_i)\geq \frac{L}{2}$ holds.
\item If neither $\gamma_{i-1}'$ nor $\gamma_i'$ are unmodified $3$-dim pieces, then $\angle z_{i-1}z_iz_{i+1}\geq \frac{\eta_0}{2}$.

\item  If  $\gamma_{i-1}'$ or $\gamma_i'$ is an unmodified $3$-dim piece, then $\angle z_{i-1}z_iz_{i+1}\geq \frac{\pi}{2}+ \frac{\eta_0}{2}$.
\end{enumerate}
On one hand, we have 
$$d_{\mathbb{H}^3}(\tilde{j}_t(x),\tilde{j}_t(y))\leq \sum_{i=0}^nl(\delta_i)\stackrel{\text{item\ (1)}}{\leq} 2\sum_{i=0}^nl(\gamma_i')\stackrel{(\ref{6.8})}{\leq} 4d_{\mathbb{H}^3}(\tilde{j}_0(x),\tilde{j}_0(y))\leq 4d_{\tilde{Z}^3}(x,y).$$
On the other hand, since $L$ is large with respect to $\eta_0$, items (2) (3) and Proposition \ref{lengthshort} imply that 
$$d_{\mathbb{H}^3}(\tilde{j}_t(x),\tilde{j}_t(y))\stackrel{\text{Prop\ \ref{lengthshort}}}{\geq} \frac{1}{2}\sum_{i=0}^nl(\delta_i)\stackrel{\text{item\ (1)}}{\geq} \frac{1}{4}\sum_{i=0}^nl(\gamma_i')\geq \frac{1}{4}d_{\tilde{Z}^3}(x,y).$$

We have proved equation (\ref{6.12}) holds for the modified endpoints $x$ and $y$, thus equation (\ref{6.11}) holds for any $x,y\in \tilde{Z}^3$. So $\tilde{j}_t:\tilde{Z}^3\to \mathbb{H}^3$ is a quasi-isometric embedding.

\end{proof}

Now we finish the proof of Theorem \ref{pi1injectivity}.
\begin{proof}[Proof of Theorem \ref{pi1injectivity}]
By Proposition \ref{deformation}, each $\tilde{j}_t:\tilde{Z}^3\to \mathbb{H}^3$ is a quasi-isometric embedding. Moreover, since $\pi_1(Z^3)$ is torsion free and $\tilde{j}_t$ is $\rho_t$-equivariant, each representation $\rho_t:\pi_1(Z^3)\to \text{Isom}_+(\mathbb{H}^3)$ is injective. Again, since $\tilde{j}_t:\tilde{Z}^3\to \mathbb{H}^3$ is a quasi-isometric embedding and $Z^3=\tilde{Z}^3/\pi_1(Z^3)$ is compact after truncating cusp ends, $\rho_t(\pi_1(Z^3))<\text{Isom}_+(\mathbb{H}^3)$ is a geometrically finite subgroup.

So $\{\rho_t(\pi_1(Z^3))\ |\ t\in[0,1]\}$ forms a continuous family of geometrically finite subgroups of $\text{Isom}_+(\mathbb{H}^3)$. Then the convex core of $\mathbb{H}^3/j_*(\pi_1(Z^3))=\mathbb{H}^3/\rho_1(\pi_1(Z^3))$ is homeomorphic to the convex core of $\mathbb{H}^3/\rho_0(\pi_1(Z^3))$, which is homeomorphic to $\mathcal{Z}\setminus \partial_p \mathcal{Z}$ (as oriented manifolds) by Proposition \ref{model} (4).
\end{proof}

\end{document}

%% file: figure1.pdf_tex
%% Creator: Inkscape 1.2 (dc2aedaf03, 2022-05-15), www.inkscape.org
%% PDF/EPS/PS + LaTeX output extension by Johan Engelen, 2010
%% Accompanies image file 'figure1.pdf' (pdf, eps, ps)
%%
%% To include the image in your LaTeX document, write
%%   \input{<filename>.pdf_tex}
%%  instead of
%%   \includegraphics{<filename>.pdf}
%% To scale the image, write
%%   \def\svgwidth{<desired width>}
%%   \input{<filename>.pdf_tex}
%%  instead of
%%   \includegraphics[width=<desired width>]{<filename>.pdf}
%%
%% Images with a different path to the parent latex file can
%% be accessed with the `import' package (which may need to be
%% installed) using
%%   \usepackage{import}
%% in the preamble, and then including the image with
%%   \import{<path to file>}{<filename>.pdf_tex}
%% Alternatively, one can specify
%%   \graphicspath{{<path to file>/}}
%% 
%% For more information, please see info/svg-inkscape on CTAN:
%%   http://tug.ctan.org/tex-archive/info/svg-inkscape
%%
\begingroup%
  \makeatletter%
  \providecommand\color[2][]{%
    \errmessage{(Inkscape) Color is used for the text in Inkscape, but the package 'color.sty' is not loaded}%
    \renewcommand\color[2][]{}%
  }%
  \providecommand\transparent[1]{%
    \errmessage{(Inkscape) Transparency is used (non-zero) for the text in Inkscape, but the package 'transparent.sty' is not loaded}%
    \renewcommand\transparent[1]{}%
  }%
  \providecommand\rotatebox[2]{#2}%
  \newcommand*\fsize{\dimexpr\f@size pt\relax}%
  \newcommand*\lineheight[1]{\fontsize{\fsize}{#1\fsize}\selectfont}%
  \ifx\svgwidth\undefined%
    \setlength{\unitlength}{448.60599236bp}%
    \ifx\svgscale\undefined%
      \relax%
    \else%
      \setlength{\unitlength}{\unitlength * \real{\svgscale}}%
    \fi%
  \else%
    \setlength{\unitlength}{\svgwidth}%
  \fi%
  \global\let\svgwidth\undefined%
  \global\let\svgscale\undefined%
  \makeatother%
  \begin{picture}(1,0.86902489)%
    \lineheight{1}%
    \setlength\tabcolsep{0pt}%
    \put(0,0){\includegraphics[width=\unitlength,page=1]{figure1.pdf}}%
    \put(0.48417944,0.38128297){\color[rgb]{0,0,0}\makebox(0,0)[lt]{\lineheight{1.25}\smash{\begin{tabular}[t]{l}$n$\\\end{tabular}}}}%
    \put(0.45241436,0.51001485){\color[rgb]{0,0,0}\makebox(0,0)[lt]{\lineheight{1.25}\smash{\begin{tabular}[t]{l}$v_n$\end{tabular}}}}%
    \put(0,0){\includegraphics[width=\unitlength,page=2]{figure1.pdf}}%
    \put(0.18826275,0.61701303){\color[rgb]{0,0,0}\makebox(0,0)[lt]{\lineheight{1.25}\smash{\begin{tabular}[t]{l}$v_{\Delta_1}$\end{tabular}}}}%
    \put(0.46244545,0.74407337){\color[rgb]{0,0,0}\makebox(0,0)[lt]{\lineheight{1.25}\smash{\begin{tabular}[t]{l}$v_{\Delta_2}$\end{tabular}}}}%
  \end{picture}%
\endgroup%

%% file: figure2.pdf_tex
%% Creator: Inkscape 1.2 (dc2aedaf03, 2022-05-15), www.inkscape.org
%% PDF/EPS/PS + LaTeX output extension by Johan Engelen, 2010
%% Accompanies image file 'figure2.pdf' (pdf, eps, ps)
%%
%% To include the image in your LaTeX document, write
%%   \input{<filename>.pdf_tex}
%%  instead of
%%   \includegraphics{<filename>.pdf}
%% To scale the image, write
%%   \def\svgwidth{<desired width>}
%%   \input{<filename>.pdf_tex}
%%  instead of
%%   \includegraphics[width=<desired width>]{<filename>.pdf}
%%
%% Images with a different path to the parent latex file can
%% be accessed with the `import' package (which may need to be
%% installed) using
%%   \usepackage{import}
%% in the preamble, and then including the image with
%%   \import{<path to file>}{<filename>.pdf_tex}
%% Alternatively, one can specify
%%   \graphicspath{{<path to file>/}}
%% 
%% For more information, please see info/svg-inkscape on CTAN:
%%   http://tug.ctan.org/tex-archive/info/svg-inkscape
%%
\begingroup%
  \makeatletter%
  \providecommand\color[2][]{%
    \errmessage{(Inkscape) Color is used for the text in Inkscape, but the package 'color.sty' is not loaded}%
    \renewcommand\color[2][]{}%
  }%
  \providecommand\transparent[1]{%
    \errmessage{(Inkscape) Transparency is used (non-zero) for the text in Inkscape, but the package 'transparent.sty' is not loaded}%
    \renewcommand\transparent[1]{}%
  }%
  \providecommand\rotatebox[2]{#2}%
  \newcommand*\fsize{\dimexpr\f@size pt\relax}%
  \newcommand*\lineheight[1]{\fontsize{\fsize}{#1\fsize}\selectfont}%
  \ifx\svgwidth\undefined%
    \setlength{\unitlength}{540.46628815bp}%
    \ifx\svgscale\undefined%
      \relax%
    \else%
      \setlength{\unitlength}{\unitlength * \real{\svgscale}}%
    \fi%
  \else%
    \setlength{\unitlength}{\svgwidth}%
  \fi%
  \global\let\svgwidth\undefined%
  \global\let\svgscale\undefined%
  \makeatother%
  \begin{picture}(1,0.79151483)%
    \lineheight{1}%
    \setlength\tabcolsep{0pt}%
    \put(0,0){\includegraphics[width=\unitlength,page=1]{figure2.pdf}}%
    \put(-0.004201,0.74680818){\color[rgb]{0,0,0}\makebox(0,0)[lt]{\lineheight{1.25}\smash{\begin{tabular}[t]{l}$n_l$\end{tabular}}}}%
    \put(0.82644492,0.74880011){\color[rgb]{0,0,0}\makebox(0,0)[lt]{\lineheight{1.25}\smash{\begin{tabular}[t]{l}$n_j$\end{tabular}}}}%
    \put(0.43814687,0.00978427){\color[rgb]{0,0,0}\makebox(0,0)[lt]{\lineheight{1.25}\smash{\begin{tabular}[t]{l}$n_i$\end{tabular}}}}%
    \put(0.50361422,0.41003401){\color[rgb]{0,0,0}\makebox(0,0)[lt]{\lineheight{1.25}\smash{\begin{tabular}[t]{l}$e_{ijk}$\end{tabular}}}}%
    \put(0.43837167,0.52690953){\color[rgb]{0,0,0}\makebox(0,0)[lt]{\lineheight{1.25}\smash{\begin{tabular}[t]{l}$n_k$\end{tabular}}}}%
    \put(0.76010917,0.52051509){\color[rgb]{0,0,0}\makebox(0,0)[lt]{\lineheight{1.25}\smash{\begin{tabular}[t]{l}$e_{ij}$\end{tabular}}}}%
    \put(0,0){\includegraphics[width=\unitlength,page=2]{figure2.pdf}}%
  \end{picture}%
\endgroup%

%% file: figure3.pdf_tex
%% Creator: Inkscape 1.1 (c68e22c387, 2021-05-23), www.inkscape.org
%% PDF/EPS/PS + LaTeX output extension by Johan Engelen, 2010
%% Accompanies image file 'figure3.pdf' (pdf, eps, ps)
%%
%% To include the image in your LaTeX document, write
%%   \input{<filename>.pdf_tex}
%%  instead of
%%   \includegraphics{<filename>.pdf}
%% To scale the image, write
%%   \def\svgwidth{<desired width>}
%%   \input{<filename>.pdf_tex}
%%  instead of
%%   \includegraphics[width=<desired width>]{<filename>.pdf}
%%
%% Images with a different path to the parent latex file can
%% be accessed with the `import' package (which may need to be
%% installed) using
%%   \usepackage{import}
%% in the preamble, and then including the image with
%%   \import{<path to file>}{<filename>.pdf_tex}
%% Alternatively, one can specify
%%   \graphicspath{{<path to file>/}}
%% 
%% For more information, please see info/svg-inkscape on CTAN:
%%   http://tug.ctan.org/tex-archive/info/svg-inkscape
%%
\begingroup%
  \makeatletter%
  \providecommand\color[2][]{%
    \errmessage{(Inkscape) Color is used for the text in Inkscape, but the package 'color.sty' is not loaded}%
    \renewcommand\color[2][]{}%
  }%
  \providecommand\transparent[1]{%
    \errmessage{(Inkscape) Transparency is used (non-zero) for the text in Inkscape, but the package 'transparent.sty' is not loaded}%
    \renewcommand\transparent[1]{}%
  }%
  \providecommand\rotatebox[2]{#2}%
  \newcommand*\fsize{\dimexpr\f@size pt\relax}%
  \newcommand*\lineheight[1]{\fontsize{\fsize}{#1\fsize}\selectfont}%
  \ifx\svgwidth\undefined%
    \setlength{\unitlength}{504.49476186bp}%
    \ifx\svgscale\undefined%
      \relax%
    \else%
      \setlength{\unitlength}{\unitlength * \real{\svgscale}}%
    \fi%
  \else%
    \setlength{\unitlength}{\svgwidth}%
  \fi%
  \global\let\svgwidth\undefined%
  \global\let\svgscale\undefined%
  \makeatother%
  \begin{picture}(1,0.48593845)%
    \lineheight{1}%
    \setlength\tabcolsep{0pt}%
    \put(0,0){\includegraphics[width=\unitlength,page=1]{figure3.pdf}}%
    \put(0.20197248,0.01196273){\color[rgb]{0,0,0}\makebox(0,0)[lt]{\lineheight{1.25}\smash{\begin{tabular}[t]{l}(a)\\\end{tabular}}}}%
    \put(0.75467521,0.01423988){\color[rgb]{0,0,0}\makebox(0,0)[lt]{\lineheight{1.25}\smash{\begin{tabular}[t]{l}(b)\\\end{tabular}}}}%
  \end{picture}%
\endgroup%

%% file: figure4.pdf_tex
%% Creator: Inkscape 1.2 (dc2aedaf03, 2022-05-15), www.inkscape.org
%% PDF/EPS/PS + LaTeX output extension by Johan Engelen, 2010
%% Accompanies image file 'figure4.pdf' (pdf, eps, ps)
%%
%% To include the image in your LaTeX document, write
%%   \input{<filename>.pdf_tex}
%%  instead of
%%   \includegraphics{<filename>.pdf}
%% To scale the image, write
%%   \def\svgwidth{<desired width>}
%%   \input{<filename>.pdf_tex}
%%  instead of
%%   \includegraphics[width=<desired width>]{<filename>.pdf}
%%
%% Images with a different path to the parent latex file can
%% be accessed with the `import' package (which may need to be
%% installed) using
%%   \usepackage{import}
%% in the preamble, and then including the image with
%%   \import{<path to file>}{<filename>.pdf_tex}
%% Alternatively, one can specify
%%   \graphicspath{{<path to file>/}}
%% 
%% For more information, please see info/svg-inkscape on CTAN:
%%   http://tug.ctan.org/tex-archive/info/svg-inkscape
%%
\begingroup%
  \makeatletter%
  \providecommand\color[2][]{%
    \errmessage{(Inkscape) Color is used for the text in Inkscape, but the package 'color.sty' is not loaded}%
    \renewcommand\color[2][]{}%
  }%
  \providecommand\transparent[1]{%
    \errmessage{(Inkscape) Transparency is used (non-zero) for the text in Inkscape, but the package 'transparent.sty' is not loaded}%
    \renewcommand\transparent[1]{}%
  }%
  \providecommand\rotatebox[2]{#2}%
  \newcommand*\fsize{\dimexpr\f@size pt\relax}%
  \newcommand*\lineheight[1]{\fontsize{\fsize}{#1\fsize}\selectfont}%
  \ifx\svgwidth\undefined%
    \setlength{\unitlength}{551.07420872bp}%
    \ifx\svgscale\undefined%
      \relax%
    \else%
      \setlength{\unitlength}{\unitlength * \real{\svgscale}}%
    \fi%
  \else%
    \setlength{\unitlength}{\svgwidth}%
  \fi%
  \global\let\svgwidth\undefined%
  \global\let\svgscale\undefined%
  \makeatother%
  \begin{picture}(1,0.55383002)%
    \lineheight{1}%
    \setlength\tabcolsep{0pt}%
    \put(0,0){\includegraphics[width=\unitlength,page=1]{figure4.pdf}}%
    \put(0.00641537,0.00639731){\color[rgb]{0,0,0}\makebox(0,0)[lt]{\lineheight{1.25}\smash{\begin{tabular}[t]{l}$v$\end{tabular}}}}%
    \put(0.80889805,0.43556414){\color[rgb]{0,0,0}\makebox(0,0)[lt]{\lineheight{1.25}\smash{\begin{tabular}[t]{l}$e$\end{tabular}}}}%
    \put(0.37485266,0.20679313){\color[rgb]{0,0,0}\makebox(0,0)[lt]{\lineheight{1.25}\smash{\begin{tabular}[t]{l}$\gamma$\end{tabular}}}}%
    \put(0.50921589,0.5259016){\color[rgb]{0,0,0}\makebox(0,0)[lt]{\lineheight{1.25}\smash{\begin{tabular}[t]{l}$\cdots \qquad \cdots$\end{tabular}}}}%
  \end{picture}%
\endgroup%

%% file: figure5.pdf_tex
%% Creator: Inkscape 1.2 (dc2aedaf03, 2022-05-15), www.inkscape.org
%% PDF/EPS/PS + LaTeX output extension by Johan Engelen, 2010
%% Accompanies image file 'figure5.pdf' (pdf, eps, ps)
%%
%% To include the image in your LaTeX document, write
%%   \input{<filename>.pdf_tex}
%%  instead of
%%   \includegraphics{<filename>.pdf}
%% To scale the image, write
%%   \def\svgwidth{<desired width>}
%%   \input{<filename>.pdf_tex}
%%  instead of
%%   \includegraphics[width=<desired width>]{<filename>.pdf}
%%
%% Images with a different path to the parent latex file can
%% be accessed with the `import' package (which may need to be
%% installed) using
%%   \usepackage{import}
%% in the preamble, and then including the image with
%%   \import{<path to file>}{<filename>.pdf_tex}
%% Alternatively, one can specify
%%   \graphicspath{{<path to file>/}}
%% 
%% For more information, please see info/svg-inkscape on CTAN:
%%   http://tug.ctan.org/tex-archive/info/svg-inkscape
%%
\begingroup%
  \makeatletter%
  \providecommand\color[2][]{%
    \errmessage{(Inkscape) Color is used for the text in Inkscape, but the package 'color.sty' is not loaded}%
    \renewcommand\color[2][]{}%
  }%
  \providecommand\transparent[1]{%
    \errmessage{(Inkscape) Transparency is used (non-zero) for the text in Inkscape, but the package 'transparent.sty' is not loaded}%
    \renewcommand\transparent[1]{}%
  }%
  \providecommand\rotatebox[2]{#2}%
  \newcommand*\fsize{\dimexpr\f@size pt\relax}%
  \newcommand*\lineheight[1]{\fontsize{\fsize}{#1\fsize}\selectfont}%
  \ifx\svgwidth\undefined%
    \setlength{\unitlength}{519.29270414bp}%
    \ifx\svgscale\undefined%
      \relax%
    \else%
      \setlength{\unitlength}{\unitlength * \real{\svgscale}}%
    \fi%
  \else%
    \setlength{\unitlength}{\svgwidth}%
  \fi%
  \global\let\svgwidth\undefined%
  \global\let\svgscale\undefined%
  \makeatother%
  \begin{picture}(1,0.43129562)%
    \lineheight{1}%
    \setlength\tabcolsep{0pt}%
    \put(0,0){\includegraphics[width=\unitlength,page=1]{figure5.pdf}}%
    \put(0.24388889,0.14707087){\color[rgb]{0,0,0}\makebox(0,0)[lt]{\lineheight{1.25}\smash{\begin{tabular}[t]{l}$\gamma$\end{tabular}}}}%
    \put(0.22092969,0.35272789){\color[rgb]{0,0,0}\makebox(0,0)[lt]{\lineheight{1.25}\smash{\begin{tabular}[t]{l}$\cdots \quad \cdots$\end{tabular}}}}%
    \put(0,0){\includegraphics[width=\unitlength,page=2]{figure5.pdf}}%
    \put(0.79103173,0.13599403){\color[rgb]{0,0,0}\makebox(0,0)[lt]{\lineheight{1.25}\smash{\begin{tabular}[t]{l}$\gamma$\end{tabular}}}}%
    \put(0.27466788,0.01084709){\color[rgb]{0,0,0}\makebox(0,0)[lt]{\lineheight{1.25}\smash{\begin{tabular}[t]{l}(a)\end{tabular}}}}%
    \put(0.80531067,0.01120108){\color[rgb]{0,0,0}\makebox(0,0)[lt]{\lineheight{1.25}\smash{\begin{tabular}[t]{l}(b)\end{tabular}}}}%
  \end{picture}%
\endgroup%

%% file: figure6.pdf_tex
%% Creator: Inkscape 1.2 (dc2aedaf03, 2022-05-15), www.inkscape.org
%% PDF/EPS/PS + LaTeX output extension by Johan Engelen, 2010
%% Accompanies image file 'figure6.pdf' (pdf, eps, ps)
%%
%% To include the image in your LaTeX document, write
%%   \input{<filename>.pdf_tex}
%%  instead of
%%   \includegraphics{<filename>.pdf}
%% To scale the image, write
%%   \def\svgwidth{<desired width>}
%%   \input{<filename>.pdf_tex}
%%  instead of
%%   \includegraphics[width=<desired width>]{<filename>.pdf}
%%
%% Images with a different path to the parent latex file can
%% be accessed with the `import' package (which may need to be
%% installed) using
%%   \usepackage{import}
%% in the preamble, and then including the image with
%%   \import{<path to file>}{<filename>.pdf_tex}
%% Alternatively, one can specify
%%   \graphicspath{{<path to file>/}}
%% 
%% For more information, please see info/svg-inkscape on CTAN:
%%   http://tug.ctan.org/tex-archive/info/svg-inkscape
%%
\begingroup%
  \makeatletter%
  \providecommand\color[2][]{%
    \errmessage{(Inkscape) Color is used for the text in Inkscape, but the package 'color.sty' is not loaded}%
    \renewcommand\color[2][]{}%
  }%
  \providecommand\transparent[1]{%
    \errmessage{(Inkscape) Transparency is used (non-zero) for the text in Inkscape, but the package 'transparent.sty' is not loaded}%
    \renewcommand\transparent[1]{}%
  }%
  \providecommand\rotatebox[2]{#2}%
  \newcommand*\fsize{\dimexpr\f@size pt\relax}%
  \newcommand*\lineheight[1]{\fontsize{\fsize}{#1\fsize}\selectfont}%
  \ifx\svgwidth\undefined%
    \setlength{\unitlength}{499.75757928bp}%
    \ifx\svgscale\undefined%
      \relax%
    \else%
      \setlength{\unitlength}{\unitlength * \real{\svgscale}}%
    \fi%
  \else%
    \setlength{\unitlength}{\svgwidth}%
  \fi%
  \global\let\svgwidth\undefined%
  \global\let\svgscale\undefined%
  \makeatother%
  \begin{picture}(1,0.77598889)%
    \lineheight{1}%
    \setlength\tabcolsep{0pt}%
    \put(0,0){\includegraphics[width=\unitlength,page=1]{figure6.pdf}}%
    \put(0.43707092,0.72979477){\color[rgb]{0,0,0}\makebox(0,0)[lt]{\lineheight{1.25}\smash{\begin{tabular}[t]{l}$y_{i+1}$\end{tabular}}}}%
    \put(0.67417901,0.18001016){\color[rgb]{0,0,0}\makebox(0,0)[lt]{\lineheight{1.25}\smash{\begin{tabular}[t]{l}$x_{i+1}$\end{tabular}}}}%
    \put(-0.0045432,0.0355333){\color[rgb]{0,0,0}\makebox(0,0)[lt]{\lineheight{1.25}\smash{\begin{tabular}[t]{l}$y_i=x_i$\end{tabular}}}}%
    \put(0,0){\includegraphics[width=\unitlength,page=2]{figure6.pdf}}%
  \end{picture}%
\endgroup%

%% file: figure7.pdf_tex
%% Creator: Inkscape 1.2 (dc2aedaf03, 2022-05-15), www.inkscape.org
%% PDF/EPS/PS + LaTeX output extension by Johan Engelen, 2010
%% Accompanies image file 'figure7.pdf' (pdf, eps, ps)
%%
%% To include the image in your LaTeX document, write
%%   \input{<filename>.pdf_tex}
%%  instead of
%%   \includegraphics{<filename>.pdf}
%% To scale the image, write
%%   \def\svgwidth{<desired width>}
%%   \input{<filename>.pdf_tex}
%%  instead of
%%   \includegraphics[width=<desired width>]{<filename>.pdf}
%%
%% Images with a different path to the parent latex file can
%% be accessed with the `import' package (which may need to be
%% installed) using
%%   \usepackage{import}
%% in the preamble, and then including the image with
%%   \import{<path to file>}{<filename>.pdf_tex}
%% Alternatively, one can specify
%%   \graphicspath{{<path to file>/}}
%% 
%% For more information, please see info/svg-inkscape on CTAN:
%%   http://tug.ctan.org/tex-archive/info/svg-inkscape
%%
\begingroup%
  \makeatletter%
  \providecommand\color[2][]{%
    \errmessage{(Inkscape) Color is used for the text in Inkscape, but the package 'color.sty' is not loaded}%
    \renewcommand\color[2][]{}%
  }%
  \providecommand\transparent[1]{%
    \errmessage{(Inkscape) Transparency is used (non-zero) for the text in Inkscape, but the package 'transparent.sty' is not loaded}%
    \renewcommand\transparent[1]{}%
  }%
  \providecommand\rotatebox[2]{#2}%
  \newcommand*\fsize{\dimexpr\f@size pt\relax}%
  \newcommand*\lineheight[1]{\fontsize{\fsize}{#1\fsize}\selectfont}%
  \ifx\svgwidth\undefined%
    \setlength{\unitlength}{898.54325996bp}%
    \ifx\svgscale\undefined%
      \relax%
    \else%
      \setlength{\unitlength}{\unitlength * \real{\svgscale}}%
    \fi%
  \else%
    \setlength{\unitlength}{\svgwidth}%
  \fi%
  \global\let\svgwidth\undefined%
  \global\let\svgscale\undefined%
  \makeatother%
  \begin{picture}(1,0.22564627)%
    \lineheight{1}%
    \setlength\tabcolsep{0pt}%
    \put(0,0){\includegraphics[width=\unitlength,page=1]{figure7.pdf}}%
    \put(0.1408947,0.00671658){\color[rgb]{0,0,0}\makebox(0,0)[lt]{\lineheight{1.25}\smash{\begin{tabular}[t]{l}(a)\end{tabular}}}}%
    \put(0.42245881,0.00722804){\color[rgb]{0,0,0}\makebox(0,0)[lt]{\lineheight{1.25}\smash{\begin{tabular}[t]{l}(b)\end{tabular}}}}%
    \put(0.70641928,0.00791471){\color[rgb]{0,0,0}\makebox(0,0)[lt]{\lineheight{1.25}\smash{\begin{tabular}[t]{l}(c)\end{tabular}}}}%
    \put(0.13991209,0.2056087){\color[rgb]{0,0,0}\makebox(0,0)[lt]{\lineheight{1.25}\smash{\begin{tabular}[t]{l}$y_i$\end{tabular}}}}%
    \put(0.71468615,0.19659115){\color[rgb]{0,0,0}\makebox(0,0)[lt]{\lineheight{1.25}\smash{\begin{tabular}[t]{l}$y_i$\end{tabular}}}}%
    \put(0.42480371,0.20851784){\color[rgb]{0,0,0}\makebox(0,0)[lt]{\lineheight{1.25}\smash{\begin{tabular}[t]{l}$y_i$\end{tabular}}}}%
    \put(-0.00168459,0.08443531){\color[rgb]{0,0,0}\makebox(0,0)[lt]{\lineheight{1.25}\smash{\begin{tabular}[t]{l}$x_{i-1}$\end{tabular}}}}%
    \put(0.32572568,0.11640754){\color[rgb]{0,0,0}\makebox(0,0)[lt]{\lineheight{1.25}\smash{\begin{tabular}[t]{l}$x_{i-1}$\end{tabular}}}}%
    \put(0.59626548,0.11682792){\color[rgb]{0,0,0}\makebox(0,0)[lt]{\lineheight{1.25}\smash{\begin{tabular}[t]{l}$x_{i-1}$\end{tabular}}}}%
    \put(0.15415464,0.08323718){\color[rgb]{0,0,0}\makebox(0,0)[lt]{\lineheight{1.25}\smash{\begin{tabular}[t]{l}$x_i$\end{tabular}}}}%
    \put(0.64290306,0.07445384){\color[rgb]{0,0,0}\makebox(0,0)[lt]{\lineheight{1.25}\smash{\begin{tabular}[t]{l}$x_i$\end{tabular}}}}%
    \put(0.37747344,0.08280082){\color[rgb]{0,0,0}\makebox(0,0)[lt]{\lineheight{1.25}\smash{\begin{tabular}[t]{l}$x_i$\end{tabular}}}}%
    \put(0.25112407,0.08353298){\color[rgb]{0,0,0}\makebox(0,0)[lt]{\lineheight{1.25}\smash{\begin{tabular}[t]{l}$x_{i+1}$\end{tabular}}}}%
    \put(0.47302196,0.08000651){\color[rgb]{0,0,0}\makebox(0,0)[lt]{\lineheight{1.25}\smash{\begin{tabular}[t]{l}$x_{i+1}$\end{tabular}}}}%
    \put(0.7100724,0.07165978){\color[rgb]{0,0,0}\makebox(0,0)[lt]{\lineheight{1.25}\smash{\begin{tabular}[t]{l}$x_{i+1}$\end{tabular}}}}%
    \put(0.5260106,0.11818413){\color[rgb]{0,0,0}\makebox(0,0)[lt]{\lineheight{1.25}\smash{\begin{tabular}[t]{l}$x_{i+2}$\end{tabular}}}}%
    \put(0.77591167,0.11215332){\color[rgb]{0,0,0}\makebox(0,0)[lt]{\lineheight{1.25}\smash{\begin{tabular}[t]{l}$x_{i+k}$\end{tabular}}}}%
    \put(0.84780042,0.06893979){\color[rgb]{0,0,0}\makebox(0,0)[lt]{\lineheight{1.25}\smash{\begin{tabular}[t]{l}$x_{i+k+1}$\end{tabular}}}}%
    \put(0.72558959,0.11926187){\color[rgb]{0,0,0}\makebox(0,0)[lt]{\lineheight{1.25}\smash{\begin{tabular}[t]{l}...\\\end{tabular}}}}%
  \end{picture}%
\endgroup%

%% file: figure8.pdf_tex
%% Creator: Inkscape 1.2 (dc2aedaf03, 2022-05-15), www.inkscape.org
%% PDF/EPS/PS + LaTeX output extension by Johan Engelen, 2010
%% Accompanies image file 'figure8.pdf' (pdf, eps, ps)
%%
%% To include the image in your LaTeX document, write
%%   \input{<filename>.pdf_tex}
%%  instead of
%%   \includegraphics{<filename>.pdf}
%% To scale the image, write
%%   \def\svgwidth{<desired width>}
%%   \input{<filename>.pdf_tex}
%%  instead of
%%   \includegraphics[width=<desired width>]{<filename>.pdf}
%%
%% Images with a different path to the parent latex file can
%% be accessed with the `import' package (which may need to be
%% installed) using
%%   \usepackage{import}
%% in the preamble, and then including the image with
%%   \import{<path to file>}{<filename>.pdf_tex}
%% Alternatively, one can specify
%%   \graphicspath{{<path to file>/}}
%% 
%% For more information, please see info/svg-inkscape on CTAN:
%%   http://tug.ctan.org/tex-archive/info/svg-inkscape
%%
\begingroup%
  \makeatletter%
  \providecommand\color[2][]{%
    \errmessage{(Inkscape) Color is used for the text in Inkscape, but the package 'color.sty' is not loaded}%
    \renewcommand\color[2][]{}%
  }%
  \providecommand\transparent[1]{%
    \errmessage{(Inkscape) Transparency is used (non-zero) for the text in Inkscape, but the package 'transparent.sty' is not loaded}%
    \renewcommand\transparent[1]{}%
  }%
  \providecommand\rotatebox[2]{#2}%
  \newcommand*\fsize{\dimexpr\f@size pt\relax}%
  \newcommand*\lineheight[1]{\fontsize{\fsize}{#1\fsize}\selectfont}%
  \ifx\svgwidth\undefined%
    \setlength{\unitlength}{903.13319322bp}%
    \ifx\svgscale\undefined%
      \relax%
    \else%
      \setlength{\unitlength}{\unitlength * \real{\svgscale}}%
    \fi%
  \else%
    \setlength{\unitlength}{\svgwidth}%
  \fi%
  \global\let\svgwidth\undefined%
  \global\let\svgscale\undefined%
  \makeatother%
  \begin{picture}(1,0.22541566)%
    \lineheight{1}%
    \setlength\tabcolsep{0pt}%
    \put(0,0){\includegraphics[width=\unitlength,page=1]{figure8.pdf}}%
    \put(0.14526071,0.00668244){\color[rgb]{0,0,0}\makebox(0,0)[lt]{\lineheight{1.25}\smash{\begin{tabular}[t]{l}(a)\end{tabular}}}}%
    \put(0.42539381,0.0071913){\color[rgb]{0,0,0}\makebox(0,0)[lt]{\lineheight{1.25}\smash{\begin{tabular}[t]{l}(b)\end{tabular}}}}%
    \put(0.70791132,0.00787449){\color[rgb]{0,0,0}\makebox(0,0)[lt]{\lineheight{1.25}\smash{\begin{tabular}[t]{l}(c)\end{tabular}}}}%
    \put(0.1442831,0.20456377){\color[rgb]{0,0,0}\makebox(0,0)[lt]{\lineheight{1.25}\smash{\begin{tabular}[t]{l}$y_i$\end{tabular}}}}%
    \put(0.71613618,0.19559205){\color[rgb]{0,0,0}\makebox(0,0)[lt]{\lineheight{1.25}\smash{\begin{tabular}[t]{l}$y_i$\end{tabular}}}}%
    \put(0.42772678,0.2074581){\color[rgb]{0,0,0}\makebox(0,0)[lt]{\lineheight{1.25}\smash{\begin{tabular}[t]{l}$y_i$\end{tabular}}}}%
    \put(0.02831929,0.08400619){\color[rgb]{0,0,0}\makebox(0,0)[lt]{\lineheight{1.25}\smash{\begin{tabular}[t]{l}$x_{i-1}$\end{tabular}}}}%
    \put(0.34908292,0.11581593){\color[rgb]{0,0,0}\makebox(0,0)[lt]{\lineheight{1.25}\smash{\begin{tabular}[t]{l}$x_{i-1}$\end{tabular}}}}%
    \put(0.61990842,0.11623417){\color[rgb]{0,0,0}\makebox(0,0)[lt]{\lineheight{1.25}\smash{\begin{tabular}[t]{l}$x_{i-1}$\end{tabular}}}}%
    \put(0.15845326,0.08281415){\color[rgb]{0,0,0}\makebox(0,0)[lt]{\lineheight{1.25}\smash{\begin{tabular}[t]{l}$x_i$\end{tabular}}}}%
    \put(0.64471781,0.07407545){\color[rgb]{0,0,0}\makebox(0,0)[lt]{\lineheight{1.25}\smash{\begin{tabular}[t]{l}$x_i$\end{tabular}}}}%
    \put(0.3806371,0.08238){\color[rgb]{0,0,0}\makebox(0,0)[lt]{\lineheight{1.25}\smash{\begin{tabular}[t]{l}$x_i$\end{tabular}}}}%
    \put(0.25492987,0.08310845){\color[rgb]{0,0,0}\makebox(0,0)[lt]{\lineheight{1.25}\smash{\begin{tabular}[t]{l}$x_{i+1}$\end{tabular}}}}%
    \put(0.47569997,0.07959989){\color[rgb]{0,0,0}\makebox(0,0)[lt]{\lineheight{1.25}\smash{\begin{tabular}[t]{l}$x_{i+1}$\end{tabular}}}}%
    \put(0.71154588,0.07129558){\color[rgb]{0,0,0}\makebox(0,0)[lt]{\lineheight{1.25}\smash{\begin{tabular}[t]{l}$x_{i+1}$\end{tabular}}}}%
    \put(0.52841933,0.11758349){\color[rgb]{0,0,0}\makebox(0,0)[lt]{\lineheight{1.25}\smash{\begin{tabular}[t]{l}$x_{i+2}$\end{tabular}}}}%
    \put(0.77705054,0.11158333){\color[rgb]{0,0,0}\makebox(0,0)[lt]{\lineheight{1.25}\smash{\begin{tabular}[t]{l}$x_{i+k}$\end{tabular}}}}%
    \put(0.84857394,0.06858942){\color[rgb]{0,0,0}\makebox(0,0)[lt]{\lineheight{1.25}\smash{\begin{tabular}[t]{l}$x_{i+k+1}$\end{tabular}}}}%
    \put(0.7269842,0.11865576){\color[rgb]{0,0,0}\makebox(0,0)[lt]{\lineheight{1.25}\smash{\begin{tabular}[t]{l}...\\\end{tabular}}}}%
    \put(0,0){\includegraphics[width=\unitlength,page=2]{figure8.pdf}}%
    \put(-0.00167602,0.20210909){\color[rgb]{0,0,0}\makebox(0,0)[lt]{\lineheight{1.25}\smash{\begin{tabular}[t]{l}$y_{i-1}$\end{tabular}}}}%
    \put(0.59264354,0.20748921){\color[rgb]{0,0,0}\makebox(0,0)[lt]{\lineheight{1.25}\smash{\begin{tabular}[t]{l}$y_{i-1}$\end{tabular}}}}%
    \put(0.31741459,0.20837428){\color[rgb]{0,0,0}\makebox(0,0)[lt]{\lineheight{1.25}\smash{\begin{tabular}[t]{l}$y_{i-1}$\end{tabular}}}}%
  \end{picture}%
\endgroup%

%% file: figure9.pdf_tex
%% Creator: Inkscape 1.2 (dc2aedaf03, 2022-05-15), www.inkscape.org
%% PDF/EPS/PS + LaTeX output extension by Johan Engelen, 2010
%% Accompanies image file 'figure9.pdf' (pdf, eps, ps)
%%
%% To include the image in your LaTeX document, write
%%   \input{<filename>.pdf_tex}
%%  instead of
%%   \includegraphics{<filename>.pdf}
%% To scale the image, write
%%   \def\svgwidth{<desired width>}
%%   \input{<filename>.pdf_tex}
%%  instead of
%%   \includegraphics[width=<desired width>]{<filename>.pdf}
%%
%% Images with a different path to the parent latex file can
%% be accessed with the `import' package (which may need to be
%% installed) using
%%   \usepackage{import}
%% in the preamble, and then including the image with
%%   \import{<path to file>}{<filename>.pdf_tex}
%% Alternatively, one can specify
%%   \graphicspath{{<path to file>/}}
%% 
%% For more information, please see info/svg-inkscape on CTAN:
%%   http://tug.ctan.org/tex-archive/info/svg-inkscape
%%
\begingroup%
  \makeatletter%
  \providecommand\color[2][]{%
    \errmessage{(Inkscape) Color is used for the text in Inkscape, but the package 'color.sty' is not loaded}%
    \renewcommand\color[2][]{}%
  }%
  \providecommand\transparent[1]{%
    \errmessage{(Inkscape) Transparency is used (non-zero) for the text in Inkscape, but the package 'transparent.sty' is not loaded}%
    \renewcommand\transparent[1]{}%
  }%
  \providecommand\rotatebox[2]{#2}%
  \newcommand*\fsize{\dimexpr\f@size pt\relax}%
  \newcommand*\lineheight[1]{\fontsize{\fsize}{#1\fsize}\selectfont}%
  \ifx\svgwidth\undefined%
    \setlength{\unitlength}{918.13328874bp}%
    \ifx\svgscale\undefined%
      \relax%
    \else%
      \setlength{\unitlength}{\unitlength * \real{\svgscale}}%
    \fi%
  \else%
    \setlength{\unitlength}{\svgwidth}%
  \fi%
  \global\let\svgwidth\undefined%
  \global\let\svgscale\undefined%
  \makeatother%
  \begin{picture}(1,0.3034205)%
    \lineheight{1}%
    \setlength\tabcolsep{0pt}%
    \put(0,0){\includegraphics[width=\unitlength,page=1]{figure9.pdf}}%
    \put(0.14288742,0.00657327){\color[rgb]{0,0,0}\makebox(0,0)[lt]{\lineheight{1.25}\smash{\begin{tabular}[t]{l}(a)\end{tabular}}}}%
    \put(0.41844382,0.00707381){\color[rgb]{0,0,0}\makebox(0,0)[lt]{\lineheight{1.25}\smash{\begin{tabular}[t]{l}(b)\end{tabular}}}}%
    \put(0.69634579,0.00774584){\color[rgb]{0,0,0}\makebox(0,0)[lt]{\lineheight{1.25}\smash{\begin{tabular}[t]{l}(c)\end{tabular}}}}%
    \put(0.14192579,0.28290928){\color[rgb]{0,0,0}\makebox(0,0)[lt]{\lineheight{1.25}\smash{\begin{tabular}[t]{l}$y_i$\end{tabular}}}}%
    \put(0.70443632,0.27408411){\color[rgb]{0,0,0}\makebox(0,0)[lt]{\lineheight{1.25}\smash{\begin{tabular}[t]{l}$y_i$\end{tabular}}}}%
    \put(0.42073867,0.28575632){\color[rgb]{0,0,0}\makebox(0,0)[lt]{\lineheight{1.25}\smash{\begin{tabular}[t]{l}$y_i$\end{tabular}}}}%
    \put(0.02458911,0.1643213){\color[rgb]{0,0,0}\makebox(0,0)[lt]{\lineheight{1.25}\smash{\begin{tabular}[t]{l}$x_{i-1}$\end{tabular}}}}%
    \put(0.34337974,0.19561134){\color[rgb]{0,0,0}\makebox(0,0)[lt]{\lineheight{1.25}\smash{\begin{tabular}[t]{l}$x_{i-1}$\end{tabular}}}}%
    \put(0.60814711,0.19602275){\color[rgb]{0,0,0}\makebox(0,0)[lt]{\lineheight{1.25}\smash{\begin{tabular}[t]{l}$x_{i-1}$\end{tabular}}}}%
    \put(0.15586444,0.16314873){\color[rgb]{0,0,0}\makebox(0,0)[lt]{\lineheight{1.25}\smash{\begin{tabular}[t]{l}$x_i$\end{tabular}}}}%
    \put(0.63418471,0.1545528){\color[rgb]{0,0,0}\makebox(0,0)[lt]{\lineheight{1.25}\smash{\begin{tabular}[t]{l}$x_i$\end{tabular}}}}%
    \put(0.37441835,0.16272168){\color[rgb]{0,0,0}\makebox(0,0)[lt]{\lineheight{1.25}\smash{\begin{tabular}[t]{l}$x_i$\end{tabular}}}}%
    \put(0.2344274,0.16343822){\color[rgb]{0,0,0}\makebox(0,0)[lt]{\lineheight{1.25}\smash{\begin{tabular}[t]{l}$x_{i+1}$\end{tabular}}}}%
    \put(0.46792805,0.15998699){\color[rgb]{0,0,0}\makebox(0,0)[lt]{\lineheight{1.25}\smash{\begin{tabular}[t]{l}$x_{i+1}$\end{tabular}}}}%
    \put(0.69992097,0.15181835){\color[rgb]{0,0,0}\makebox(0,0)[lt]{\lineheight{1.25}\smash{\begin{tabular}[t]{l}$x_{i+1}$\end{tabular}}}}%
    \put(0.49691375,0.19735002){\color[rgb]{0,0,0}\makebox(0,0)[lt]{\lineheight{1.25}\smash{\begin{tabular}[t]{l}$x_{i+2}$\end{tabular}}}}%
    \put(0.76435539,0.19144789){\color[rgb]{0,0,0}\makebox(0,0)[lt]{\lineheight{1.25}\smash{\begin{tabular}[t]{l}$x_{i+k}$\end{tabular}}}}%
    \put(0.85104787,0.1491564){\color[rgb]{0,0,0}\makebox(0,0)[lt]{\lineheight{1.25}\smash{\begin{tabular}[t]{l}$x_{i+k+1}$\end{tabular}}}}%
    \put(0.71510711,0.19840477){\color[rgb]{0,0,0}\makebox(0,0)[lt]{\lineheight{1.25}\smash{\begin{tabular}[t]{l}...\\\end{tabular}}}}%
    \put(0,0){\includegraphics[width=\unitlength,page=2]{figure9.pdf}}%
    \put(-0.00164864,0.2804947){\color[rgb]{0,0,0}\makebox(0,0)[lt]{\lineheight{1.25}\smash{\begin{tabular}[t]{l}$y_{i-1}$\end{tabular}}}}%
    \put(0.58296116,0.2857869){\color[rgb]{0,0,0}\makebox(0,0)[lt]{\lineheight{1.25}\smash{\begin{tabular}[t]{l}$y_{i-1}$\end{tabular}}}}%
    \put(0.31222874,0.28665754){\color[rgb]{0,0,0}\makebox(0,0)[lt]{\lineheight{1.25}\smash{\begin{tabular}[t]{l}$y_{i-1}$\end{tabular}}}}%
    \put(0,0){\includegraphics[width=\unitlength,page=3]{figure9.pdf}}%
    \put(0.27176635,0.05430291){\color[rgb]{0,0,0}\makebox(0,0)[lt]{\lineheight{1.25}\smash{\begin{tabular}[t]{l}$y_{i+1}$\end{tabular}}}}%
    \put(0.52912658,0.05451797){\color[rgb]{0,0,0}\makebox(0,0)[lt]{\lineheight{1.25}\smash{\begin{tabular}[t]{l}$y_{i+1}$\end{tabular}}}}%
    \put(0.81825865,0.05154349){\color[rgb]{0,0,0}\makebox(0,0)[lt]{\lineheight{1.25}\smash{\begin{tabular}[t]{l}$y_{i+1}$\end{tabular}}}}%
  \end{picture}%
\endgroup%